\documentclass[a4paper, 11pt, reqno]{amsart}

\usepackage{amsmath, amsthm, amssymb}
\numberwithin{equation}{section}
\usepackage{xspace}
\usepackage{graphicx}
\usepackage{array}
\usepackage{braket}
 \usepackage{comment}
\usepackage{multicol}
\usepackage{mathtools}
\usepackage{enumerate}
\usepackage{delarray}
\usepackage{mathtools}
\usepackage{faktor} % For quotients
\usepackage{mathrsfs}
\usepackage{tikz-cd}
\usepackage{hyperref}
\usepackage[normalem]{ulem} % For underlining
\usepackage[italicdiff]{physics} % For differentials
\usepackage{bbm} % For indicator
\usepackage{float}
\usepackage{stmaryrd} % for double brackets
\usepackage{aligned-overset}
\usepackage{xcolor}
\usepackage[margin=15truemm]{geometry}

\theoremstyle{plain}
\newtheorem{thm}{Theorem}[section]
\newtheorem{lemma}[thm]{Lemma}
\newtheorem{corollary}[thm]{Corollary}

\newtheorem{prop}[thm]{Proposition}
\theoremstyle{remark}

\theoremstyle{definition}
\newtheorem{remark}[thm]{Remark}
\newtheorem{exmp}[thm]{Example}
\newtheorem{defn}[thm]{Definition}
\newtheorem*{notation*}{Notation}

\DeclarePairedDelimiterX{\inp}[2]{\langle}{\rangle}{#1, #2}

\makeatletter
\newcommand*{\bigcdot}{}% Check if undefined
\DeclareRobustCommand*{\bigcdot}{%
  \mathbin{\mathpalette\bigcdot@{}}%
}
\newcommand*{\bigcdot@scalefactor}{.5}
\newcommand*{\bigcdot@widthfactor}{1.15}
\newcommand*{\bigcdot@}[2]{%
  \sbox0{$#1\vcenter{}$}% math axis
  \sbox2{$#1\cdot\m@th$}%
  \hbox to \bigcdot@widthfactor\wd2{%
    \hfil
    \raise\ht0\hbox{%
      \scalebox{\bigcdot@scalefactor}{%
        \lower\ht0\hbox{$#1\bullet\m@th$}%
      }%
    }%
    \hfil
  }%
}
\makeatother

\newcommand{\C}{\mathbb{C}}

\newcommand{\Z}{\mathbb{Z}\xspace}

\newcommand{\B}{\mathbb{B}\xspace}
\newcommand{\M}{\mathbb{M}\xspace}

\newcommand{\X}{\mathbb{X}\xspace}
\newcommand{\Y}{\mathbb{Y}\xspace}
\renewcommand{\S}{\mathbb{S}\xspace}

\newcommand{\K}{\mathcal{K}}

\DeclareMathOperator{\supp}{supp}

\let\Span\relax
\DeclareMathOperator{\Span}{span}

\DeclareMathOperator{\Aut}{Aut}

\DeclareMathOperator{\id}{id}
\DeclareMathOperator{\nor}{nor}
\DeclareMathOperator{\opp}{op}
\DeclareMathOperator{\Ad}{Ad}
\DeclareMathOperator{\Dom}{Dom}
\newcommand\restr[2]{\ensuremath{\left.#1\right|_{#2}}}

\usepackage{enumitem}
\usepackage{color}
\newlist{steps}{enumerate}{1}
\usepackage[foot]{amsaddr}
\setlist[steps, 1]{label = Step \arabic*:}
\hypersetup{%
  colorlinks=true,%
  linkcolor=blue,%
  citecolor=blue,%
  filecolor=blue,%
  menucolor=blue,%
  urlcolor=blue,%
  pdfnewwindow=true,%
  pdfstartview=FitBH
}   
\title{On Relative Biexactness of Amalgamated Free Product von Neumann Algebras}
\author{Kai Toyosawa$^{1}$ and Zhiyuan Yang$^{2}$}
\address{$^{1}$Mathematik M\"unster, Universit\"at M\"unster, Einsteinstr.\ 62, 48149 M\"unster, Germany}
\email{$^{1}$\href{mailto:ktoyosaw@uni-muenster.de}{ktoyosaw@uni-muenster.de}}
\address{$^{2}$Department of Mathematics, Texas A\&M University, College Station, TX 77843, USA}
\email{$^{2}$\href{mailto:zhiyuanyang@tamu.edu}{zhiyuanyang@tamu.edu}}
\newcommand{\KK}{{\mathcal K}}
\newcommand{\FF}{{\mathcal F}}
\newcommand{\BB}{{\mathbb{ B}}}
\newcommand{\HH}{{\mathcal H}}

\begin{document}
\maketitle

\begin{abstract}
    Given separable weakly exact tracial von Neumann algebras $M_{1}, M_{2}$ with a common injective von Neumann subalgebra $B$, we prove that the amalgamated free product $M_{1}\overline{*}_{B}M_{2}$ is biexact relative to $\{M_{1},M_{2}\}$. When $M_1$ and $M_2$ are injective but not necessarily tracial, we further show that $M_{1}\overline{*}_{B}M_{2}$ is biexact relative to the amalgam $B$, and if $B$ is mixing in each of $M_1$ and $M_2$, $M_{1}\overline{*}_{B}M_{2}$ itself is biexact. We also develop a new criterion for upgrading relative biexactness and obtain structural decomposition and subalgebra absorption results for amalgamated free product von Neumann algebras that generalize the group case.
\end{abstract}

\section{Introduction}

In \cite{MR2211141}, Ozawa introduced a class of exact groups, which nowadays is referred to as groups of class $\mathcal{S}$ or biexact groups in \cite{MR2391387}, and proved that the group von Neumann algebra $L\Gamma$ of a biexact group $\Gamma$ is solid, that is, for any diffuse von Neumann subalgebra $A$ in $L\Gamma$, the relative commutant $A' \cap L\Gamma$ is injective \cite{MR2079600}. As a consequence, one recovers the result from \cite{MR1609522} proved via different techniques from Voiculescu's free probability theory that every free group factor $L\mathbb{F}_{n}$ is prime, i.e., if one decomposes $L\mathbb{F}_{n}$ into tensor product of two von Neumann algebras, then one of which is finite dimensional. 

In \cite{MR2391387}, a relative version of biexact groups was introduced: given a countable group $\Gamma$ and a collection $\mathcal{G}$ of subgroups of $\Gamma$, let $c_{0}(\Gamma;\mathcal{G})$ denote the set of bounded functions $f$ on $\Gamma$ such that for every $\varepsilon>0$, the set $\{x\in \Gamma: |f(x)|>\varepsilon\}$ is contained in a finite union of $s\Lambda t$'s, where $s,t\in \Gamma$ and $\Lambda\in\mathcal{G}$. The \textit{small-at-infinity boundary of $\Gamma$ relative to $\mathcal{G}$} is defined to be 
\[\mathbb{S}(\Gamma;\mathcal{G}) = \{f\in \ell^{\infty}\Gamma: f- R_{t}f \in c_{0}(\Gamma;\mathcal{G}),\ \forall t\in \Gamma\},\]
where $R_{t}$ denotes the right translation action of $t\in \Gamma$ on $\ell^{\infty}\Gamma$. The countable group $\Gamma$ is called \textit{biexact relative to $\mathcal{G}$} if $\Gamma$ is exact and the left translation action of $\Gamma$ on $\mathbb{S}(\Gamma;\mathcal{G})$ is amenable. In particular, the group $\Gamma$ is biexact if and only if it is biexact relative to $\mathcal{G}= \{\{1\}\}$. 

Relative biexactness is in general a weaker condition on groups than the genuine biexactness, but this property is often enough to conclude rigidity results in group von Neumann algebras. For example, it is shown in \cite[Proposition 15.3.12]{MR2391387} that if $\Gamma_{1},\Gamma_{2}$ are countable exact groups with a common amenable subgroup $\Lambda$, then the amalgamated free product $\Gamma = \Gamma_{1}*_{\Lambda}\Gamma_{2}$ is biexact relative to $\mathcal{G} = \{\Gamma_{1},\Gamma_{2}\}$, although the amalgamated free product group $\Gamma$ is not biexact in general. Combining with Popa's powerful intertwining-by-bimodule theorems \cite{MR2215135}, Ozawa proved a Kurosh-type theorem for free products of group von Neumann algebras associated with nonamenable exact groups \cite{MR2211141}.

Ding and Peterson extended the notion of relative biexactness to the setting of von Neumann algebras in \cite{ding2023biexact}. Given a von Neumann algebra $M$ and a collection $\mathcal{B}$ of von Neumann subalgebras of $M$ with normal conditional expectations, Ding, Kunnawalkam Elayavalli, and Peterson introduced in \cite{MR4675043} the small-at-infinity boundary of $M$ relative to $\mathcal{B}$, which is an operator subsystem of $\BB(L^2 M)$ that contains $M$. They defined $M$ to be biexact relative to $\mathcal{B}$ if the inclusion $M$ into the small-at-infinity boundary is $M$-nuclear, meaning the inclusion map can be pointwise approximated in a locally convex topology depending $M$ by contractive completely positive (c.c.p.)\ maps that factor through finite-dimensional matrix algebras. 

This generalizes relative biexactness in the group case, as Ding and Peterson showed in \cite[Theorem 6.2]{ding2023biexact} that the countable discrete group $\Gamma$ is biexact relative to a collection $\mathcal{G}$ of subgroups of $\Gamma$ if and only if the group von Neumann algebra $L\Gamma$ is biexact relative to $\{L\Lambda:\Lambda \in \mathcal{G}\}$.

In particular, a von Neumann algebra $M$ is called biexact if it is biexact relative to the scalars $\mathcal{B} = \{\mathbb{C}1\}$. As noted above, this notion recovers the classical definition: $L\Gamma$ is biexact in the sense of \cite{ding2023biexact} if and only if $\Gamma$ is biexact in the sense of \cite{MR2391387}. However, Ding and Peterson also observed that many examples of biexact von Neumann algebras do not arise from groups, including the von Neumann algebras of the dual of free orthogonal and unitary quantum groups \cite{Iso15biexactquantum}, free Araki-Woods von Neumann algebras \cite{houdayer2007thesis}, and $q$-Gaussian von Neumann algebras associated to finite-dimensional real Hilbert space \cite{Shl04freeentropyestimate,Kuz23}.

In \cite[Theorem 5.10]{ding2023biexact}, Ding and Peterson proved that when $M_{1}\subset \mathbb{B}(\mathcal{H}_{1}), M_{2}\subset \mathbb{B}(\mathcal{H}_{2})$ are weakly exact $\sigma$-finite von Neumann algebras with fixed vector states $\omega_{\xi_{i}}$, the canonical embedding of the free product von Neumann algebra $M = M_{1}\overline{*}M_{2}$ into the free product $(\mathbb{B}(\mathcal{H}_{1}),\omega_{\xi_{1}})\overline{*} (\mathbb{B}(\mathcal{H}_{2}),\omega_{\xi_{2}})$ is $M$-nuclear. Furthermore, \cite[Proposition 6.16]{ding2023biexact} shows that $(\mathbb{B}(\mathcal{H}_{1}),\omega_{\xi_{1}})\overline{*} (\mathbb{B}(\mathcal{H}_{2}),\omega_{\xi_{2}})$ lies in the small-at-infinity boundary of $M$ relative to $M_{1}, M_{2}$. These results together imply that $M$ is biexact relative to $\{M_{1},M_{2}\}$, thereby allowing a recovery of Kurosh-type theorems for free products of weakly exact von Neumann algebras that were first proved in \cite{MR3594283}. 

This paper investigates biexactness in the setting of amalgamated free products of von Neumann algebras and groups, and derives structural results in this context. In what follows, we abbreviate “amalgamated free product” as AFP. 

Our first main result adapts \cite[Theorem 5.10]{ding2023biexact} to the the setting of AFP von Neumann algebras with injective amalgam. Note that the theorems and corollaries below are proved for AFPs of arbitrarily finitely many von Neumann algebras, but for simplicity, we state them only in the case of two.

%In order to adapt the notion of relative biexactness to the settings of von Neumann algebras, Ding and Peterson introduced in \cite{ding2023biexact} the $M$-topology, which is a locally convex topology defined on operator system bimodules of von Neumann algebra $M$, and the $M$-nuclear maps, which are the maps that can be pointwise approximated in the $M$-topology by c.c.p.\ (completely contractive positive) maps that factor through finite-dimensional matrix algebras. 

\begin{thm} \label{thm: inclusion M-nuclear}
    Suppose $M_1,M_2$ are weakly exact von Neumann algebras and $B\subset M_i$ is a common unital von Neumann subalgebra with faithful normal conditional expectation $ E_{i} \colon M_i\to B $. Let $\langle M_i,e_B\rangle $ denote the (Jones') basic construction associated with $E_{i}$ and let the (not necessarily faithful) normal conditional expectation from $ \langle M_i,e_B\rangle$ onto $ B$ be given by $ E_i(x)\coloneqq i^*_B xi_B $ for $x\in \langle M_i,e_B\rangle$, where $i_B\colon L^2B\to L^2M_i$ is the canonical inclusion. 
    
    Let $M_r = M_1*_B M_2$ be the reduced AFP $C^*$-algebra and $M = M_1 \overline{*}_B M_2$ be the AFP von Neumann algebra. 
    If each inclusion $M_i\subset \langle M_i, e_B\rangle$ is $M_i$-nuclear,
    then the canonical inclusion
    \[ M_r=M_1 *_B M_2\hookrightarrow \langle M_1,e_B\rangle *_B \langle M_2,e_B\rangle \] is $(M_r \subset M)$-nuclear.
\end{thm}

Here $(M_r \subset M)$-nuclearity is a relative version of $M$-nuclearity; see Section \ref{sec: weakly exact von Neumann algebra} for precise definition.

Our approach to Theorem \ref{thm: inclusion M-nuclear} is different from that of \cite[Theorem 5.10]{ding2023biexact}, in which Ding and Peterson showed the nuclearity of the embedding map of the free product von Neumann algebra $M_{1} \overline{*} M_{2}$ by constructing free products of the unital completely positive (u.c.p.)\ maps that factor through finite-dimensional matrix algebras arising from the weak exactness of $M_{1},M_{2}$. %Also, $\sigma$-finiteness of $M_{i}$ was used in their proof to induce pure states on the matrix algebras and ensure their free products are nuclear C$^{*}$-algebras by \cite[Theorem 1.1]{MR1915435}. 
However, this strategy does not extend to the setting of AFP von Neumann algebras, where taking free products of u.c.p.\ maps is no longer feasible due to the lack of $B$-linearity, so instead we adopt an approach inspired by \cite{MR4833744}. We show the embedding map into $\langle M_1,e_B\rangle {\ast}_B \langle M_2,e_B\rangle$ factors through a Toeplitz-Pimsner algebra, and we apply its universal property to deduce the nuclearity of the embedding map. An advantage of our method is that it does not require any assumption on the states of $M_i$, and in particular, it remains valid when $M_1, M_2$ are not $\sigma$-finite.

When $M_{1},M_{2}$ are tracial, we show that $\langle M_1,e_B\rangle\ast_B \langle M_2,e_B\rangle$ lies in the small-at-infinity boundary relative to $\{M_{1},M_{2}\}$. Combining with Theorem \ref{thm: inclusion M-nuclear}, we can generalize the results in \cite[Proposition 15.3.12]{MR2391387} and \cite[Proposition 6.16]{ding2023biexact} and prove the relative biexactness for AFPs of weakly exact tracial von Neumann algebras over injective amalgams.

\begin{corollary}\label{cor: relative biexactness tracial}
If $M_{1}, M_{2}$ are weakly exact tracial von Neumann algebras admitting a common injective unital von Neumann subalgebra $B$ with trace preserving conditional expectations, then the AFP von Neumann algebra $M_{1}\overline{*}_{B}M_{2}$ is biexact relative to $\{M_{1},M_{2}\}$. 
\end{corollary}

As in the group case, we immediately obtain from the above corollary and \cite[Proposition 6.13]{ding2023biexact} a relative solidity result for von Neumann subalgebras $N$ in $M_{1}\overline{*}_{B}M_{2}$, which states that either the relative commutant $N'\cap M_{1}\overline{*}_{B}M_{2}$ is amenable, or $N$ intertwines into one of the $M_{i}$'s (in the sense of Popa, see e.g., \cite{MR2215135}\cite{anantharaman2017introduction}). However, we do not pursue this direction further, as a more general result is already proved in \cite[Theorem 4.4]{MR3594283} without the tracial or weakly exact assumptions.

We note that the injectivity assumption on the amalgam $B$ is essential, as counterexamples exist even for the group case when the amalgam is not an amenable group (see e.g., the remark after \cite[Proposition 15.3.12]{MR2391387}). Also, in general we cannot improve the relative biexactness to actual biexactness. % (i.e., relative to $\mathbb{C}1$). Indeed, 
For example, if we let $\mathcal{R} $ be the separable hyperfinite $\text{II}_1$ factor and consider $L(\mathbb{Z})\otimes \mathcal{R}$ with the canonical conditional expectation onto $\mathbb{C}1\otimes  \mathcal{R}$, then $ ( L(\mathbb{Z})\otimes \mathcal{R} )\overline{*}_{ \mathbb{C}1\otimes  \mathcal{R} } ( L(\mathbb{Z})\otimes \mathcal{R} ) = L(\mathbb{F}_2)\otimes \mathcal{R} $ is a non-prime factor and therefore not solid. In particular, it is not biexact.

When $M_1,M_2$ are injective, we can remove the tracial assumption in Corollary \ref{cor: relative biexactness tracial} and improve the relative biexactness.
\begin{thm}\label{thm: relative biexactness injective}
    If $M_1,M_2$ are injective von Neumann algebras admitting a common unital von Neumann subalgebra $B$ with faithful normal conditional expectations, then the AFP von Neumann algebra $ M_1\overline{\ast}_B M_2 $ is biexact relative to $B$.
\end{thm}

The second main result of this paper is a decomposition theorem for relatively biexact tracial von Neumann algebras. We recall the notion of proper proximality for a group $\Gamma$, introduced by Boutonnet, Ioana, and Peterson in \cite{MR4258166} as a generalization of nonamenable biexact groups, defined by the nonexistence of a $\Gamma$-invariant state on the small-at-infinity boundary $\mathbb{S}(\Gamma)$. This concept was later extended to the von Neumann algebraic setting in \cite{MR4675043}. In \cite{DS24structure}, Ding and Kunnawalkam Elayavalli showed that if $\Gamma$ is biexact relative to a finite collection of almost malnormal subgroups, then any von Neumann subalgebra $A \subset L\Gamma$ admits a decomposition into a sum of relatively amenable components and a complementary part that is properly proximal. By combining this structure with tools from Popa’s deformation/rigidity theory, they established a subalgebra absorption theorem for group von Neumann algebras associated with free products of exact groups. In this paper, we refine the methods of \cite{DS24structure} to prove a general decomposition theorem for tracial von Neumann algebras, and extend the subalgebra absorption result to AFPs of tracial von Neumann algebras as an application of Corollary \ref{cor: relative biexactness tracial}.

\begin{thm} \label{thm: split rel biex}
      Let $M$ be a separable tracial von Neumann algebra and $\mathcal{F}=\{N_1,\cdots,N_n\}$ be a finite family of mixing von Neumann subalgebra of $M$. Let $p\in M$ be a projection. If $M$ is biexact relative to $ \mathbb{X}_\FF $, then for any von Neumann subalgebra $ A\subset pMp $, there exist projections $f_0, f_{n+1}\in \mathcal{Z}(A)$ and $f_i\in \mathcal{Z}(A'\cap pMp)$, $i=1,\cdots,n$, such that $\sum_{i=0}^{n+1} f_i=p$, $Af_0$ is properly proximal, $Af_{n+1}$ is amenable, and $Af_i$ is amenable relative to $N_i$ inside $M$ for all $i=1,\cdots,n$.
\end{thm}

\begin{corollary} \label{cor: AFP absorbing}
    Let $(M_1,\tau_1)$ and $(M_2,\tau_2)$ be two separable tracial weakly exact von Neumann algebras with a common amenable subalgebra $B$ such that $\tau_1|_B=\tau_2|_B$, and $ M_1\overline{\ast}_{B}M_2 $ be the AFP von Neumann algebra. Assume that $B$ is mixing in $M_i$ for $i=1,2$. If $A\subset M$ is a von Neumann subalgebra with no properly proximal direct summand and $A\cap M_1 \npreceq_{M_1} B $, then $ A\subset M_1 $.
\end{corollary}

Applying Corollary \ref{cor: relative biexactness tracial} and Theorem \ref{thm: split rel biex}, we obtain a Kurosh-type decomposition for free products of weakly exact, nonamenable, non–properly proximal II$_1$ factors, generalizing \cite[Corollary 1.7]{DS24structure} from the case of group von Neumann algebras. %We also adapt the flip automorphism method in \cite{IsoMarra22}, originally used for prime decompositions of tensor products, to the free product setting, which allows us to remove the assumptions on one side of the Kurosh-type result. 
Note that here non–proper proximality is a stronger condition than non–biexactness, since every nonamenable biexact von Neumann factor is properly proximal.

\begin{corollary} \label{cor: Kurosh-type}
    If $M = M_1 \overline{\ast}\cdots \overline{\ast}M_m = N_1 \overline{\ast}\cdots \overline{\ast}N_n$, where every $M_i, N_j$ is a separable weakly exact nonamenable non-properly proximal $\text{II}_1$ factor, then $ m=n $, and up to a permutation of indices, $M_i$ is unitarily conjugate to $N_i$ in $M$.
\end{corollary}

The third main result of this paper is a new technique at the von Neumann algebra level for upgrading relative biexactness to biexactness, achieved by analyzing the inclusion of the von Neumann algebra into the bidual of its small-at-infinity boundary. 

\begin{thm} \label{thm: upgrading biexactness}
    Let $M$ be a separable von Neumann algebra and $N\subset M$ be a mixing subalgebra with expectation. If $M$ is biexact relative to $N$ and the embedding $ M\hookrightarrow \langle M,e_N\rangle $ is $M$-nuclear, then $M$ is biexact.
\end{thm}

From this theorem we deduce two permanence properties of biexactness under AFP. One consequence, combined with Theorem \ref{thm: relative biexactness injective}, is that amalgamated free products of separable injective von Neumann algebras over a mixing amalgam are biexact.

\begin{corollary} \label{cor: biexact AFP}
    Let $M_1,M_2$ be separable injective von Neumann algebras admitting a common subalgebra $B$ with expectation. If $B$ is mixing in both of $M_1, M_2$, then the AFP von Neumann algebra $ M= M_1\overline{*}_B M_2 $ is biexact.
\end{corollary}

For example, if $\Gamma = (\Z/2\Z) \wr \Z$ is the lamplighter group, $(X_0, \mu_0)$ is a non-trivial standard probability measure space, and $\Gamma\curvearrowright (X,\mu) \coloneqq (X_0, \mu_0)^\Gamma$ is the Bernoulli action, then the crossed product von Neumann algebra $L^\infty(X,\mu)\rtimes\Gamma$ is injective and $L\Gamma \subset L^\infty(X,\mu)\rtimes\Gamma$ is a mixing subalgebra since the Bernoulli action is mixing. Therefore, by the above corollary,  $(L^\infty(X,\mu)\rtimes\Gamma)\overline{*}_{L\Gamma}(L^\infty(X,\mu)\rtimes\Gamma)   $ is biexact.

\iffalse

Combining with \cite[Proposition 6.13]{ding2023biexact}, we obtain %as a corollary that for every finite von Neumann subgalgebra $N\subset M =  M_1\overline{\ast}_B M_2$ without amenable direct summand, the relative commutant $N'\cap M$ is amenable. In particular, %if $M_{1}, M_{2}$ are moreover $\sigma$-finite, then we have following dichotomy result for the finite subfactors of the amalgamated free product von Neumann algebra.

\begin{corollary}\label{cor: dichotomy in  injective}
    If $M_1, M_2$ are $\sigma$-finite injective von Neumann algebras admitting a common unital von Neumann subalgebra $B$ with faithful normal conditional expectations, then for every finite von Neumann subalgebra $N\subset M =  M_1\overline{\ast}_B M_2$ with expectation, either $N $ intertwines into $B$, or the relative commutant $N'\cap M$ is amenable. 
    
    In particular, every separable finite subfactor $N\subset M$ with expectation is either McDuff or prime.
\end{corollary}

We remark that various results concerning the primeness and McDuff properties of AFP von Neumann algebras have been established in the literature. For example, Chifan and Houdayer showed in \cite{MR2641939} that non-amenable AFP factors over abelian amalgams are prime. Moreover, it was shown in \cite{MR4091070} that every AFP von Neumann algebra $M_{1} \overline{*}_{B} M_{2}$, where $B$ is a type I von Neumann algebra and $M_{1}$ does not embed into $B$ inside $M_{1}$, is not McDuff.

\fi

\subsection*{Acknowledgments} The authors thank Changying Ding for helpful comments, and Srivatsav Kunnawalkam Elayavalli for drawing their attention to the upgrading proper proximality method in \cite{DS24structure}. They are also grateful to Srivatsav Kunnawalkam Elayavalli and Amine Marrakchi for kindly explaining the ideas of the free product decomposition theorems. The authors thank the referee for their careful reading of the paper and for their helpful corrections and suggestions.

\section{Amalgamated free product of von Neumann algebras}
In this section, we briefly recall the construction of AFP C$^*$-algebras and von Neumann algebras, originating from Voiculescu’s free probability theory \cite{MR1217253}. Given two unital C$^*$-algebra $A_1,A_2$ with a common unital subalgebra $B$ and two conditional expectation $E:A_i\to B$, the reduced AFP C$^*$-algebra $ (A_1,E)\ast_B (A_2,E) $ is defined using the free product of right Hilbert $B$-module \cite{voiculescu2006symmetries}: Let $X_i $ be the GNS-Hilbert $B$-module for $(A_i,E)$ and $X_i^o \coloneqq X_i\ominus B$, then the free product right Hilbert $B$-module is
$$ X \coloneqq B\oplus \bigoplus_{n\geqslant 1} \bigoplus_{i_1\neq\cdots \neq i_n}X_{i_1}^o\otimes_B \cdots \otimes_B X_{i_n}^o. $$
Each $A_i$ has a natural left action on $X$, and $(A_1,E)\ast_B (A_2,E)$ is defined as the subalgebra $C^*(A_1,A_2)$ of the C$^*$-algebra of left adjointable operators $\mathcal{L}(X)$ generated by the left action of $A_1$ and $A_2$.

Similarly, if $A_i=M_i$ are von Neumann algebras and $B$ is a common von Neumann subalgebra with normal conditional expectations, then we can also define the AFP von Neumann algebra $ (M_1,E)\overline{\ast}_B (M_2,E) $ by considering the self-dual completion $\widehat{X}$ of $X$ \cite{paschke1973inner}. Since $\widehat{X}$ is self-dual, $\mathcal{L}(\widehat{X})$ is a $W^*$-algebra, and thus $  (M_1,E)\overline{\ast}_B (M_2,E) $ can be defined as $ W^*( M_1,M_2 )\subset \mathcal{L}(\widehat{X}) $.

In this paper, however, we focus on constructing AFP von Neumann algebras through the framework of amalgamated free products of $W^*$-correspondences (i.e., Hilbert spaces with normal left and right actions), which is equivalent to using self-dual right Hilbert $B$-modules (see, for example, \cite{anantharaman1995amenable}). The reason why we choose to stick to $W^*$-correspondence is that, given a von Neumann subalgebra $B \subset M$ equipped with a faithful normal conditional expectation, the category of right Hilbert $B$-modules over $B \subset M$ does not admit the natural right action of $M$ as morphisms, since such actions are not right $B$-linear. However, in the context of analyzing biexactness, it is essential to study both left and right actions of a von Neumann algebra. Therefore, we begin by reviewing the construction of the standard representation of the AFP von Neumann algebra, which involves the use of relative tensor products of $W^*$-correspondence. We note that our construction of AFP von Neumann algebras agrees with that of \cite{MR1198815} in the case when the involved von Neumann algebras are tracial, and with \cite{MR1738186} when the von Neumann algebras are $\sigma$-finite.

\subsection{Relative tensor product of $W^*$-correspondence}

Since we do not assume the $\sigma$-finiteness of our von Neumann algebras, we begin by briefly discussing the tensor product of modules (correspondence) over von Neumann algebras in the absence of a faithful normal state. Although this construction is standard, it is rarely explicitly addressed in the context of AFP von Neumann algebras. For the sake of completeness and to maintain a self-contained exposition, we recall the relevant definitions. For details on this construction, see for example \cite[Chapter IX]{takesaki2003theory} and \cite[Appendix B]{connes1994noncommutative}.

Fix a von Neumann algebra $M$. We first recall the definition of n.s.f.\ weight $ \varphi $ on $M$:
\begin{defn}
    A weight on a von Neumann algebra $M$ is a map $ \varphi:M_+\to [0,\infty] $ such that $ \varphi(\lambda x+y)=\lambda \varphi(x)+\varphi(y) $ for all $ x,y\in M_+ $ and $\lambda\geqslant 0$. We say $ \varphi $ is
			\begin{itemize}
			\item faithful if $ \varphi(x)\neq 0 $ for all positive $x\neq 0$;
			\item semifinite if the span of the set of elements with finite weight $ \{ x\in M_+: \varphi(x)<\infty \} $ generate $ M $;
			\item normal if for every bounded increasing net $ (x_i)_i $ in $ M_+ $, $ \varphi(\sup x_i)=\sup \varphi(x_i) $.
			\item a n.s.f.\ weight if $\varphi$ is normal, semifinite, and faithful.
			\end{itemize}
\end{defn}
For a n.s.f.\ weight $\varphi$, we denote $$ \Dom(\varphi) \coloneqq \Span\{ x\in M_+: \varphi(x)<\infty \},\quad \Dom_{1/2}(\varphi) \coloneqq \{x\in M: \varphi(x^*x)<\infty \}. $$
We also denote $L^2M=L^2(M,\varphi)$ the Hilbert space obtained by the completion of $ \Dom_{1/2}(\varphi) $ with respect to the inner product $ \langle x,y\rangle\coloneqq  \varphi(x^*y) $. The canonical inclusion $\eta_\varphi:\Dom_{1/2}(\varphi)\to L^2M  $ is formally denoted by $\eta_\varphi(x)=x\varphi^{1/2}\in L^2M $. The map $ \eta_\varphi $ is also called a semicyclic representation of $M$ since $ y\eta_\varphi(x)= yx\varphi^{1/2}= \eta_\varphi(yx) $ for $x\in \Dom_{1/2}(\varphi)$ and $y\in M$.

Tomita's involution operator $S_\varphi$ is the closable conjugate linear operator on $L^2M$ given by $ S_\varphi\colon x\varphi^{1/2}\mapsto x^*\varphi^{1/2}  $ for $x\in \Dom_{1/2}(\varphi)^*\cap \Dom_{1/2}(\varphi)$. The polar decomposition of $S_\varphi$ is denoted by $ S_\varphi= J\Delta_\varphi^{1/2} $, where $ J=J^{-1}$ is anti-unitary and $\Delta_\varphi$ is positive (and unbounded if $S_{\varphi}$ is unbounded).

Note that if we identify $ M^{\opp}\simeq M' $ via the isomorphism $x^{\opp}\mapsto Jx^*J$, then the opposite n.s.f.\ weight $\varphi^{\opp}$ on $M^{\opp}$ given by $ \varphi^{\opp}(x^{\opp}) = \varphi(x)$ coincides with the weight $ \varphi(J\cdot J) $ on $M'$. One can check that in fact the map $$ x^{\opp}\mapsto Jx^*J\varphi^{1/2}\coloneqq Jx^*\varphi^{1/2},\quad x^{\opp}\in \Dom_{1/2}(\varphi^{\opp})= (\Dom_{1/2}(\varphi)^*)^{\opp} $$ is also a semicyclic representation for $M^{\opp}$ equivalent to the one given by $\varphi^{\opp}$. In particular, we have $ L^2M\simeq L^2M' $ as $ Jx^*\varphi^{1/2}$ is dense in $L^2M$. For this reason, we identify $ L^2M $ and $L^2M'$, and simply denote the canonical inclusion $\eta_{\varphi^{\opp}}\colon \Dom_{1/2}(\varphi^{\opp})= (\Dom_{1/2}(\varphi)^*)^{\opp}\to L^2M $,
$$ \eta_{\varphi^{\opp}}(x^{\opp}) = Jx^*\varphi^{1/2} \eqqcolon \varphi^{1/2}x,\quad x\in \Dom_{1/2}(\varphi)^*,  $$
where the notation $ Jx^*\varphi^{1/2}=\varphi^{1/2}x $ is due to the formal relation $ \varphi^{1/2}x\varphi^{-1/2}\varphi^{1/2} = \Delta_{\varphi}^{1/2}x\varphi^{1/2}=Jx^*\varphi^{1/2} $.

Suppose the Hilbert space $ \HH_1$ is a normal right $M$-module, and similarly $\HH_2$ is a normal left $M$-module, i.e., we have two normal representations $ \pi_1:M^{\opp}\to B(\HH_1) $ and $ \pi_2:M\to B(\HH_2) $. In short, we write
$$ x\xi = \pi_2(x)\xi,\quad \zeta x=\pi_1(x^{\opp})\zeta,\quad\forall x\in M, \xi\in \HH_2,\zeta\in \HH_1.$$

We now define the tensor product $ \HH_1\otimes_M \HH_2 $. We consider the following dense subspace of the right module $\HH_1$:
\[ D( \HH_1,\varphi )\coloneqq \{ \xi\in \HH_1: \exists C_\xi>0,\forall x\in \Dom_{1/2}(\varphi)^*, \| \xi x \|^2\leqslant C_\xi^2\varphi(xx^*) \}. \]
Each $\xi\in D( \HH_1,\varphi )$ gives rise to an right $M$-linear operator $L_\varphi(\xi)\in \BB( L^2M,\HH_1 )$ by
$$ L_\varphi(\xi)\varphi^{1/2}x \coloneqq \xi x,\quad x\in \Dom_{1/2}(\varphi)^{*}. $$

Symmetrically, we also denote
$$ D'(\HH_2,\varphi) \coloneqq \{ \zeta\in \HH_2:\exists C_\zeta>0 \text{ such that }\forall x\in \Dom_{1/2}(\varphi), \| x\zeta \|^2\leqslant C_\zeta^2\varphi(x^*x) \}, $$
and define $ R_\varphi(\xi)\in \BB( L^2M,\HH_2) $ to be the left $M$-linear map given by $R_\varphi(\xi)x\varphi^{1/2}\coloneqq  x\xi$ for $x\in \Dom_{1/2}(\varphi)$.

\begin{defn}
    The relative tensor product $ \HH_1\otimes_M \HH_2 $ is the Hilbert space formed by separation and completion of $ D(\HH_1,\varphi)\odot\HH_2 $ with respect to the positive sesquilinear form
    $$ \langle \xi_1\otimes \zeta_1,\xi_2\otimes \zeta_2 \rangle\coloneqq  \langle \zeta_1,\pi_2(L^*_{\varphi}(\xi_1)L_\varphi(\xi_2))\zeta_2 \rangle. $$
    Equivalently, $ \HH_1\otimes_M \HH_2 $ can also be obtained from $ \HH_1\odot D'(\HH_2,\varphi) $ with respect to the positive sesquilinear form
    $$ \langle \xi_1\otimes \zeta_1,\xi_2\otimes \zeta_2\rangle \coloneqq \langle \xi_1,\pi_1( R^*_\varphi(\zeta_1)R_\varphi(\zeta_2) ) \xi_2\rangle.$$
\end{defn}

In fact, the above definition does not depend on the choice of $\varphi$ (\cite[Chapter IX, Theorem 3.21]{takesaki2003theory}, \cite[Appendix B, Proposition 12]{connes1994noncommutative}). Since every von Neumann algebra admits a n.s.f.\ weight, $ \HH_1\otimes_M \HH_2 $ is well-defined for any von Neumann algebra $M$.

\begin{exmp}\label{exmp: notation for relative tensor product}
    If $ \HH_1 = L^2M $, then we have precisely $ D(L^2M,\varphi) = \eta_\varphi(\Dom_{1/2}(\varphi)) = \{x\varphi^{1/2} : x\in\Dom_{1/2}(\varphi)\} $. In particular, the inner product on $L^2M\otimes_M \HH_2$ coincides with the inner product on $ \HH_2 $ because
    \[ \langle x_1\varphi^{1/2}\otimes \zeta_1, x_2\varphi^{1/2}\otimes \zeta_2\rangle = \langle  \zeta_1,x_1^*x_2\zeta_2 \rangle= \langle x_1\zeta_1,x_2\zeta_2\rangle, \quad \forall x_1, x_2 \in \Dom_{1/2}(\varphi), \zeta_1,\zeta_2 \in \HH_2.\]
    Therefore, we actually have $ L^2M\otimes_M \HH_2 \simeq \HH_2$ via the identification $ x\varphi^{1/2}\otimes \zeta \mapsto x\zeta$. On the other hand, the inner product $\langle x_1\varphi^{1/2}\otimes \zeta_1, x_2\varphi^{1/2}\otimes \zeta_2\rangle = \langle  \zeta_1,x_1^*x_2\zeta_2 \rangle $ can be (formally) extended to $x_1,x_2\in M$. In particular, the identification $ L^2M\otimes_M \HH_2 =\HH_2 $ can be represented as
    $$ L^2M\otimes_M\HH_2 \ni 1_M\varphi^{1/2}\otimes \zeta = \varphi^{1/2}\otimes\zeta \mapsto \zeta\in \HH_2. $$
    In general, under this isomorphism, a vector $ \xi\otimes \zeta \in L^2M\otimes_M \HH_2$ with $\zeta\in D'(\HH_2,\varphi)$ is identified with $ R_\zeta \xi \in \HH_2 $.
    
    Similarly, we also have $  \HH_1\otimes_M L^2M = \HH_1$ via the isomorphism $ \xi\otimes \varphi^{1/2}\mapsto \xi\in \HH_1$.
\end{exmp}

\subsection{Amalgamated free product} \label{sec: Amalgamated free product}
Let $M_1,M_2$ be two von Neumann algebra having a fixed unital von Neumann subalgebra $B$ with faithful normal conditional expectations $E_i: M_i\to B $. Let $L^2B$ be the standard Hilbert space for $B$, $ L^2M_i $ be the standard Hilbert space for $M_i$, and $ L^2M_i^o = L^2M_i \ominus L^2B$.  We also denote by $ N_i = \langle M_i,e_B \rangle = (J_iBJ_i)'\subset \mathbb{B}(L^2M_i)$ the (Jones') basic construction for $B\subset M$ is the von Neuamm subalgebra generated by $M_{i}$ and the projection $e_{B}$ in $\mathbb{B}(L^{2}M_{i})$. Note that if we let $i_B: L^2B\to L^2M_i$ be the canonical inclusion, then we also have a (not necessarily faithful) conditional expectation $ E_i\colon  N_i\to B $ given by $ E_i(x)\coloneqq i^*_B xi_B $. Sometimes we omit the subscript $i$ and simply denote $ E_i(x)=E(x) $ when it is clear that $x\in N_i$.

We denote the free product Hilbert space \[ \mathcal{F} = L^2B\oplus \bigoplus_{n\geqslant 1} \bigoplus_{i_1\neq\cdots \neq i_n}L^2M_{i_1}^o\otimes_B \cdots \otimes_B L^2M_{i_n}^o, \]
where $ \otimes_B $ is the relative tensor product.

Let $ \mathcal{F}(i)\coloneqq  L^2B\oplus \bigoplus_{n\geqslant 1} \bigoplus_{\substack{i_1\neq i\\i_1\neq\cdots \neq i_n}}L^2M_{i_1}^o\otimes_B \cdots \otimes_B L^2M_{i_n}^o $, then we have the isomorphism $ \mathcal{F} \simeq L^2M_i\otimes_B \mathcal{F}(i) $ (cf.\ \cite[Section 2]{MR1738186}). Note that the tensor product $ L^2M_i\otimes_B \mathcal{F}(i) $ is also a $ N_i $-$B$ bimodule as $ L^2M_i $ is a $ N_i $-$B$ bimodule and $ \mathcal{F}(i) $ is a $B$ bimodule. Therefore, we have the right $B$-linear $*$-representation of $ N_i $ on $ \mathcal{F}\simeq L^2M_i\otimes_B \FF(i) $ by acting on the first tensor components
$$ \lambda_i:  N_i \coloneqq \langle M_i,e_B\rangle \to \mathbb{B}(\mathcal{F}),$$
and the AFP von Neumann algebra $M_1\overline{\ast}_{B}M_2 $ is defined as
\[ M= M_1\overline{\ast}_{B}M_2 \coloneqq (\lambda_1(M_1)\cup\lambda_2(M_2))'' .\]
Essentially, this construction is the same as in \cite{MR1738186} except that we do not assume that $\varphi$ is a state when taking the relative tensor products.

The faithful normal conditional expectation from $ M_1\overline{\ast}_{B}M_2 $ onto $ B $ is then given by
\[ E(x) = i_B^{*}xi_B, \quad \forall x\in M_1\ast_{B}M_2, \]
where $ i_B\colon  L^2B\to \mathcal{F} $ is the canonical embedding.

We also consider the reduced AFP C$^{*}$-algebra \[ M_r = M_1\ast_{B}M_2 \coloneqq  C^*(\lambda_1(M_1),\lambda_2(M_2)) \subset  B(\mathcal{F}) ,\]
which is a C$^{*}$-subalgebra of the larger reduced AFP C$^{*}$-algebra
\[ N_r= N_1\ast_{B}N_2 \coloneqq C^*( \lambda_1(N_1),\lambda_2(N_2)).\]
Similarly, $ E(x)=i^*_Bxi_B $ also gives a conditional expectation from $ N_r $ onto $B$.

The right action of $M$ on $\FF$ can be defined in a similar way. Let $  \mathcal{F}'(i)\coloneqq  L^2B\oplus \bigoplus_{n\geqslant 1} \bigoplus_{\substack{i_n\neq i\\i_1\neq\cdots \neq i_n}}L^2M_{i_1}^o\otimes_B \cdots \otimes_B L^2M_{i_n}^o  $, then we have the isomorphism $\FF \simeq \FF'(i)\otimes_B L^2M_i$ as a $ B $-$M_i$ (in fact $ B $-$N_i$) bimodule. Therefore, we have the left $B$-linear $*$-representation $\rho_i$ of $ M_i^{op} $ on $ \mathcal{F}\simeq \FF'(i)\otimes_B L^2M_i $ by acting on the second tensor components. $M'\in \BB(\FF)$ can now be identified with $ \{\rho_i(M_i^{op}):i=1,2\}'' $.

\section{A C$^{*}$-bimodule characterization of weak exactness and biexactness}
In this prelimnary section we collect some basic facts of a locally convex topology relative to von Neumann algebras on C$^{*}$-bimodules introduced by Magajna in \cite{MR1616512} and \cite{MR1750836}, and later adapted and generalized in \cite{MR4675043} and \cite{ding2023biexact}. This topology has been used to characterize weakly exact von Neumann algebras \cite[Theorem 5.1]{ding2023biexact} and biexact von Neumann algebras \cite[Definition 6.1]{ding2023biexact}. We refer the reader to \cite{MR1750836}, \cite{MR4675043}, and \cite{ding2023biexact} for the details on this topology.

\subsection{C$^{*}$-bimodule topologies associated with von Neumann algebras}

Throughout this section, we assume that $M$ is a von Neumann algebra and $A$ is a unital, ultraweakly dense C$^{*}$-subalgebra of $M$. A \textit{(concrete) operator $A$-system} $X$ consists of a concrete embedding of an operator system $X\subset \mathbb{B}(\mathcal{H})$ into the space of bounded linear operators on a Hilbert space $\mathcal{H}$ and a faithful non-degenerate representation $\pi\colon A\to \mathbb{B}(\mathcal{H})$ such that $X$ is a $\pi(A)$-bimodule. We call such an embedding of $X\subset \mathbb{B}(\mathcal{H})$ together with $\pi$ \textit{a concrete realization of $X$ as an $A$-system}. We say the operator $A$-system $X$ is \textit{$(A\subset M)$-normal} if the concrete realization can be made so that $\pi$ extends to a normal representation of $M$. If $X$ is moreover a unital C$^{*}$-algebra, we will say that it is an \textit{operator $A$-C$^{*}$-algebra} or an \textit{$(A\subset M)$-normal} $A$-C$^{*}$-algebra if $X$ is so as an operator system. We will often drop $\pi$ in the notation.

We let $A^{\sharp}$ denote the subspace of $A^{*}$ consisting of linear functionals that can be extended to normal linear functionals on $M$. Given positive linear functionals $\omega,\rho \in A^{\sharp}$, we consider the seminorm on $X$ as in \cite[Section 1]{MR1750836} and \cite[Section 3]{ding2023biexact} given by
\[s^{\rho}_{\omega}(x) = \inf\{\rho(a^{*}a)^{\frac{1}{2}}\left\Vert y \right\Vert\omega(b^{*}b)^{\frac{1}{2}} : x= a^{*}yb, a,b\in A, y\in X\}, \quad x\in X.\]
We call the topology on $X$ induced by the seminorms $\{s^{\rho}_{\omega}:\omega,\rho \in (A^{\sharp})_{+}\}$ the \textit{$(A\subset M)$-topology} on $X$, or simply the \textit{$M$-topology} when $A=M$.

We denote by $X^{A\sharp A}$, or just by $X^{\sharp}$ if no confusion will arise, the space of linear functionals $\varphi \in X^{*}$ such that for any $x\in X$, the map $A\times A \ni (a,b)\mapsto \varphi(axb)$ extends to a separately ultraweakly continuous bilinear form on $M$, i.e., the linear functionals $A\ni a\mapsto \varphi(ax)$ and $A\ni b\mapsto \varphi(xb)$ are normal. We call the $\sigma(X, X^{A\sharp A})$-topology the \textit{weak $(A\subset M)$-topology}, or simply the \textit{weak $M$-topology} when $A=M$. By \cite[Thoerem 3.7]{MR1750836} or \cite[Proposition 3.3]{ding2023biexact}, a functional $\varphi\in X^{*}$ is continuous in the $(A\subset M)$-topology if and only if $\varphi\in X^{A\sharp A}$.

In the sequel, we will also need another topology on $X$, called the \textit{$(A\subset M)$-$\mathbb{C}$-topology}, or simply the \textit{$M$-$\mathbb{C}$-topology} when $A = M$, which is the locally convex topology on $X$ generated by the seminorms of the form
\[s^{\rho}(x) = \inf\{\rho(a^{*}a)^{\frac{1}{2}}\left\Vert y \right\Vert : x= a^{*}y, a\in A, y\in X\}, \quad x\in X, \rho \in A^{\sharp}.\]
We similarly define the seminorms $s_{\omega}$ on $X$ for $\omega \in A^{\sharp}$ by
\[s_{\omega}(x) = \inf\{\left\Vert y \right\Vert\omega(b^{*}b)^{\frac{1}{2}} : x= yb, b\in A, y\in X\}, \quad x\in X\]
and define the \textit{$\mathbb{C}$-$(A\subset M)$-topology}, or the \textit{$\mathbb{C}$-$M$-topology} when $A = M$, on $X$ by these seminorms $s_{\omega}$.

%We remark that every operator system $X$ can be regarded as an operator $\mathbb{C}$-system via the embedding $\mathbb{C} = \mathbb{C}1 \subset X$. In this case, the $\mathbb{C}$-topology on $X$ is the norm topology and the weak $\mathbb{C}$-topology is the weak topology $\sigma(X,X^{*})$.

Now, suppose $B$ is a unital $(A\subset M)$-normal $A$-C$^{*}$-algebra. We let $p_{\nor}\in A^{**}\subset B^{**}$ denote the projection corresponding to the support of the identity representation $A\to M$. As explained in \cite[Section 2]{MR4675043} and \cite[Section 3.1]{ding2023biexact}, we may identify $(p_{\nor}B^{**}p_{\nor})_{*} \simeq B^{\sharp}$ by considering the restriction map to $B$, so the dual map naturally gives $B^{\sharp *}$ a von Neumann algebra structure so that 
\[B^{\sharp *}\simeq p_{\nor}B^{**}p_{\nor}\]
as von Neumann algebras. Since $p_{\nor}$ commutes with $A\subset B^{**}$, the above isomorphism preserves the natural $A$-bimodule structures on $B^{\sharp *}$ and $p_{\nor}B^{**}p_{\nor}$, so we can view $M = A^{\sharp *}$ as a von Neumann subalgebra of $B^{\sharp *}$. 

If $i_{B}\colon B\to B^{\sharp *}$ denotes the canonical inclusion map of $B$, then $i_{B}$ is an $A$-bimodular complete order isomorphism, but it is not a $*$-homomorphism in general because $p_{\nor}$ might not be central in $B^{**}$.

\subsection{Weakly exact von Neumann algebra} \label{sec: weakly exact von Neumann algebra}
We now recall a definition of weak exactness of ultraweakly dense C$^{*}$-algebras, \cite[Definition 3.1.1]{MR3004955}, which generalizes the original definition of weakly exact von Neumann algebra in \cite{MR1403994}. If $M$ is a von Neumann algebra and $A\subset M$ is a (unital) ultraweakly dense C$^{*}$-algebra, then we say \textit{$A$ is weakly exact in $M$} if for any unital C$^{*}$-algebra $B$ with a closed two-sided ideal $J$ and any representation $\pi\colon A\otimes B\to \mathbb{B}(\mathcal{K})$ with $A\otimes J\subset \ker\pi$ such that $\restr{\pi}{A\otimes \mathbb{C}}$ extends to an ultraweakly continuous representation of $M$, the induced representation $\widetilde{\pi}\colon A\odot B/J\to \mathbb{B}(\mathcal{K})$ is min-continuous. 

In particular, when $A=M$, $M$ is weakly exact in $M$ if and only if $M$ is weakly exact in the sense of Kirchberg \cite{MR1403994}. Also, if $M = A^{**}$, then $A$ is weakly exact in $A^{**}$ if and only if $A$ is an exact C$^{*}$-algebra.

Recall that if $E$ is an operator system or a C$^{*}$-algebra and if $F$ is an operator $A$-system or $A$-C$^{*}$-algebra for a (unital) ultraweakly dense C$^{*}$-subalgebra $A$ of a von Neumann algebra $M$, then as defined in \cite[Section 4.1]{ding2023biexact}, we say a c.c.p.\ map $\phi\colon E\to F$ is \textit{$(A\subset M)$-nuclear} (or just \textit{$M$-nuclear} when $A = M$) if there exist nets of c.c.p.\ maps $\phi_{i}\colon E\to \mathbb{M}_{n(i)}(\mathbb{C})$ and $\psi_{i}\colon \mathbb{M}_{n(i)}(\mathbb{C})\to F$ such that $\psi_{i}\circ \phi_{i}(x)$ converges to $\phi(x)$ in the $(A\subset M)$-topology for every $x\in E$. 

The result \cite[Theorem 5.1]{ding2023biexact} characterize weak exactness in these terms:  if $M\subset \mathbb{B}(\mathcal{H})$ is a von Neumann algebra and $A\subset M$ is a (unital) ultraweakly dense C$^{*}$-algebra, then $A$ is weakly exact in $M$ if and only if the inclusion map $A\subset \mathbb{B}(\mathcal{H})$ is $(A\subset M)$-nuclear. 

As next lemma shows, in certain cases, one can replace the ambient space $\mathbb{B}(\mathcal{H})$ with a smaller operator system in the characterization of nuclear embeddings arising from weak exactness. %Recall that given a von Neumann algebra $M$ and a von Neumann subalgebra $B\subset M$ with a normal conditional expectation $E\colon M\to B$, (Jones') basic construction for $B\subset M$ is the von Neuamm subalgebra 
% $N = \langle M, e_{B}\rangle$ generated by $M$ and the projection $e_{B}$ in $\mathbb{B}(L^{2}M)$.

\begin{lemma} \label{lem: weak exactness into smaller injective}
    Suppose $M$ is a von Neumann algebra and $A\subset M$ is an ultraweakly dense C$^{*}$-subalgebra. If $B\subset M$ is an injective von Neumann subalgebra with a normal conditional expectation $E\colon M\to B$, then $A$ is weakly exact in $M$ if and only if the inclusion map $A\hookrightarrow N= \langle M, e_{B}\rangle$ is $(A\subset M)$-nuclear.
\end{lemma}
\begin{proof}
    If $A$ is weakly exact in $M$, then by \cite[Lemma 4.3]{ding2023biexact}, there exist nets of u.c.p.\ maps $\phi_{i}\colon A\to \mathbb{M}_{n(i)}(\mathbb{C})$ and $\widetilde{\psi}_{i}\colon \mathbb{M}_{n(i)}(\mathbb{C})\to \mathbb{B}(L^{2}M)$ such that $\widetilde{\psi}_{i}\circ \phi_{i}$ converges to the canonical inclusion map $A\subset \mathbb{B}(L^{2}M)$ in the point-$(A\subset M)$-topology, and such that for each $x\in A$ there exist $T_{i}\in \mathbb{B}(L^{2}M)$ and $a_{i}\in A$ such that $\widetilde{\psi}_{i}\circ \phi_{i}(x) - x = T_{i}-a_{i}$, $\left\Vert T_{i}\right\Vert \to 0$, $\sup_{i}\left\Vert a_{i}\right\Vert_{i} < \infty$, and $a_{i}\to 0$ ultrastrongly as $i\to \infty$.

    Since $B$ is injective, the basic construction $N$ is injective as well, so there exists a conditional expectation $E\colon \mathbb{B}(L^{2}M)\to N$ such that the restriction $\restr{E}{N}$ is the identity map on $N$. In particular, $E$ is normal on $M$. We let $\psi_{i} = E\circ \widetilde{\psi}_{i}\colon \mathbb{M}_{n(i)}(\mathbb{C})\to N$. For every $f\in \mathbb{B}(L^{2}M)^{A\sharp A}$, the restriction $\restr{f}{A}$ to $A$ is ultraweakly continuous, so for every $x\in A$,
    \[\lim_{i\to\infty}f(\psi_{i}\circ \phi_{i}(x) - x) = \lim_{i\to\infty}f(E(\widetilde{\psi}_{i}\circ \phi_{i}(x)-x)) = \lim_{i\to\infty}f(E(T_{i}-a_{i})) = \lim_{i\to\infty}f(E(T_{i})) - f(a_{i}) = 0.\]
    This shows that $\psi_{i}\circ \phi_{i}\colon A\to N$ is a net of c.c.p.\ maps that factors through the matrix algebras and converges to the canonical inclusion map $A\subset N$ in the point-weak-$(A\subset M)$-topology. Hence the inclusion $A\subset N$ is $(A\subset M)$-nuclear.
\end{proof}

\subsection{Biduals $\B(L^2M)^{**}$, $  \B(L^2M)^{\sharp *} $, and $ \B(L^2M)^{\sharp *}_J $} \label{sec: biduals}
To define proper proximality and biexactness for a von Neumann algebra $M\subset \B(L^2M)$, it is necessary to consider the normal $M$-$M$ and $JMJ$-$JMJ$ operator bimodules in $ \B(L^2M)$. These properties are often more naturally formulated in terms of biduals, so we also work within the bidual von Neumann algebra $ \B(L^2M)^{**} $. Let $ \iota \colon \B(L^2M)\to \B(L^2M)^{**}$ denote the canonical unital (non-normal) inclusion. %For any $*$-subalgebra $S\subset \B(L^2M)$, its bidual $S^{**}\subset \B(L^2M)$ can also be described as the subalgebra 
%\[S^{**}= \iota(S)''1_{S^{**}}\subset \B(L^2M)^{**},\] 
%where $ 1_{S^{**}} = \sup\{{\iota(a)\in S}:{ 0\leqslant a\leqslant 1}\} $ is the support of $ S^{**} $.

We again denote by $ p_{\nor} \in \mathcal{P}(\mathcal{Z}( M^{**} ))$ the support projection of the identity map $M\to M$. Equivalently, $p_{\nor}$ is the projection onto the orthogonal complement of the kernel of the $*$-homomorphism 
\[(\id_M^*\big|_{M_*})^*\colon M^{**} \to M \]
which arises from the universal property of the bidual $M^{**}$ of the von Neumann algebra $M$. If $ \B(L^2M)^{**}\subset \B(\K) $ is a normal unital representation, then $ p_{\nor} \K $ can be identified with the space of vectors $\xi\in \K$ such that the vector functional $ \langle \xi,\,\cdot\,\xi\rangle $ is normal on $ M$. 

We similarly define $p_{\nor}^{J} \in \mathcal{P}(\mathcal{Z}( (JMJ)^{**} ))$ as the support projection corresponding to the identity map on $JMJ$. Since $\iota(M)$ and $\iota(JMJ)$ are weak$^*$ dense in $ M^{**} $ and $JMJ^{**}$, respectively, and commute in $\B(L^2M)^{**}$, we have $[M^{**},JMJ^{**}] =0$, so $ p_{\nor}p_{\nor}^{J}= p_{\nor}^{J}p_{\nor} $. We thus define
\[ p_{\nor}^{\sharp} \coloneqq p_{\nor}p_{\nor}^{J} \in (M^{**})'\cap (JMJ^{**})'\subset \B(L^2M)^{**}.\] 
The predual of the von Neumann subalgebra $ p_{\nor}^{\sharp} \B(L^2M)^{**} p_{\nor}^{\sharp} $ is identified with
\[p_{\nor}^{\sharp} \B(L^2M)^{*} p_{\nor}^{\sharp}\simeq \B(L^2M)^{M\sharp M}\cap\B(L^2M)^{JMJ\sharp JMJ}.\] 
We denote by 
\[ \iota^{\sharp}\colon \B(L^2M)\to p_{\nor}^{\sharp} \B(L^2M)^{**} p_{\nor}^{\sharp}, \quad  \iota^{\sharp}(x)\coloneqq p_{\nor}^{\sharp}\iota(x)p_{\nor}^{\sharp}\] 
the canonical compression map on the bidual. Note that since $ p_{\nor}^{J} \in (M^{**})'$, the restriction $ \iota^{\sharp}|_{M}\colon M\to p_{\nor}^{J}p_{\nor} M p_{\nor}p_{\nor}^{J} $ is a normal $*$-isomorphism because  $ M\simeq p_{\nor}^{\sharp} M p_{\nor}^{\sharp} $. Therefore, we may regard $ M\simeq \iota^{\sharp}(M) $ and similarly $ JMJ\simeq \iota^{\sharp}(JMJ)$ as von Neumann subalgebras of $ p_{\nor}^{\sharp} \B(L^2M)^{**} p_{\nor}^{\sharp}$. 

Even though $ \iota $ and $\iota^{\sharp}$ are not normal maps on all of $\B(L^2M)$, the fact that $ p_{\nor}^{\sharp}\in (M^{**})'\cap (JMJ^{**})' $ guarantees that both $ M $ and $JMJ$ lie in the multiplicative domain of $\iota^{\sharp}$, and $\iota^{\sharp}$ is normal on both $M $ and $JMJ$.

\begin{remark}
    Although $\iota$ and $ \iota^{\sharp}$ are not normal maps, they are still completely order isomorphism to its image.
\end{remark}

For any C$^*$-subalgebra $A\subset \B(L^2M) $ (with $ \overline{AL^2M}=L^2M $) such that $ M,JMJ\subset \M(A) \subset \B(L^2M)$, recall that $ A^{M\sharp M} $ (resp.\ $A^{JMJ\sharp JMJ}$) is the space of $f\in A^*$ such that the map $ M\times M \ni (a,b)\mapsto  f(axb) $ (resp.\ $JMJ\times JMJ \ni (a,b)\mapsto f(axb)$) is separately normal in each variable. We set $ A_J^{\sharp} \coloneqq A^{M\sharp M}\cap A^{JMJ\sharp JMJ} $. When $\id_{L^2M}\in A$, we can identify $ (A_J^{\sharp})^{*} = p_{\nor}^{\sharp}A^{**}p_{\nor}^{\sharp} $ as a von Neumann algebra. In general, since $ M,JMJ\subset \M(A) $, we have $ p_{\nor}^{\sharp}\in  \M(A)^{**} $, and thus we can also consider $ M,JMJ $ as von Neumann subalgebras of $ p_{\nor}^{\sharp}\M(A)p_{\nor}^{\sharp} $. As $ A $ is an ideal of $\M(A)$, $ A^{**}= q_{A}\M(A)^{**} $ for some $ q_{A}\in \mathcal{P}(  \mathcal{Z}(\M(A)^{**})) $. Therefore, we can also identify $  (A_J^{\sharp})^{*} \simeq q_{A}p_{\nor}^{\sharp}{\M(A)}^{**}p_{\nor}^{\sharp} = q_{A}p_{\nor}^{\sharp}{A}^{**}p_{\nor}^{\sharp} $ as a von Neumann algebra.

While for a unital $*$-subalgebra $A \subset \B(L^2M)$, the bidual $A^{**}$ is always a unital von Neumann subalgebra of $\B(L^2M)^{**}$, in the non-unital case, the unit $\overline{q}_A \coloneqq \mathbf{1}_{A^{**}}$ of $A^{**}$ is, in general, only a nontrivial projection in $\B(L^2M)^{**}$. In the case when $[\overline{q}_{A},p_{\nor}^{\sharp}]=0$, we define $ q_{A} \coloneqq p^{\sharp}_{\nor}\overline{q}_A$.

If $ A \subset \B(L^2M) $ is a C$^*$-subalgebra such that $\iota(A)$ commutes with $p^{\sharp}_{\nor}$, then for any approximate identity $(e_i)$ of $A$, we have
\[ q_{A} = \overline{q}_Ap^{\sharp}_{\nor} = \lim_{i}\iota^{\sharp}(e_i),\]
and $q_{A}$ is the identity of $ p^{\sharp}_{\nor}A^{**}p^{\sharp}_{\nor}=A^{**}p^{\sharp}_{\nor}\subset  \B(L^2M)^{\sharp*}_{J}  $. Indeed, the identity of $A^{**}$ in $ \B(L^2M) $ is $ \overline{q}_A=\lim \iota(e_i) $. Since $ \Ad(p^{\sharp}_{\nor}) $ is normal, we have $$ \lim_{i}\iota^{\sharp}(e_i) = \Ad(p^{\sharp}_{\nor})( \lim_{i}\iota(e_i)  ) = p^{\sharp}_{\nor}\overline{q}_A p^{\sharp}_{\nor} = p^{\sharp}_{\nor}\overline{q}_A=q_{A}.$$

In particular, if $\mathbb{X} = \overline{A\B(L^2M)A}$ is the hereditary $C^*$-subalgebra of $\B(L^2M)$ generated by $ A $, then $(e_i)$ is also an approximate identity for $\mathbb{X}$ since $ A\B(L^2M)A $ is norm dense in $\mathbb{X}$. Hence we also have
\[ \overline{q}_{\mathbb{X}} = \overline{q}_{A}\in \{p^{\sharp}_{\nor}\}',\quad {q}_{\mathbb{X}} = {q}_{A} = \overline{q}_{A}p^{\sharp}_{\nor}. \]

\subsection{(Relative) proper proximality and biexactness}

We recall the notions of properly proximal von Neumann algebras in \cite{MR4675043} and biexact von Neumann algebras in \cite{ding2023biexact} which are generalizations of properly proximal countable groups in \cite{MR4258166} and relatively biexact countable groups in \cite{MR2391387}, respectively.

Let $M\subset \mathbb{B}(L^{2}M)$ be a von Neumann algebra. A hereditary C$^{*}$-subalgebra $\mathbb{X}$ of $\mathbb{B}(L^{2}M)$ is called an \textit{$M$-boundary piece} if $\mathbb{M}(\mathbb{X})\cap M$ and $\mathbb{M}(\mathbb{X})\cap M'$ are ultraweakly dense in $M$ and $M'$, respectively, where $\mathbb{M}(\mathbb{X})$ is the multiplier algebra of $\mathbb{X}$ and $M' = M'\cap \mathbb{B}(L^{2}M)$ is the commutant of $M$ in $\mathbb{B}(L^{2}M)$. To avoid pathological examples, we always assume $\mathbb{X} \neq \{0\}$, so then $\mathbb{K}(L^2M) \subset \mathbb{X}$ by the assumption of $\mathbb{X}$. Denote by
\[ \mathbb{K}_{\mathbb{X}}^{L}(M) = \overline{\overline{\mathbb{B}(L^{2}M)\mathbb{X}}^{\mathbb{C}\text{-}M}}^{\mathbb{C}\text{-}M'} = \overline{\overline{\mathbb{B}(L^{2}M)\mathbb{X}}^{\mathbb{C}\text{-}M'}}^{\mathbb{C}\text{-}M}\]
the closure of $\mathbb{B}(L^{2}(M)\mathbb{X}$ in both $\mathbb{C}$-$M$ and $\mathbb{C}$-$M'$-topology in $\mathbb{B}(L^{2}M)$, so then $\mathbb{K}_{\mathbb{X}}^{L}(M)$ is a left ideal of $\mathbb{B}(L^{2}M)$ containing $M$ and $M'$ in its space of right multipliers. Note that by \cite[Proposition 2.3]{MR4675043} (\cite[Proposition 2.2]{MR1616512}), an operator $T \in \mathbb{B}(L^{2}M)$ is contained in $\mathbb{K}_{\mathbb{X}}^{L}(M)$ if and only if there exist orthogonal families of projections $\{f_{i}\}_{i\in I}, \{e_{j}\}_{j\in J}\subset M$ such that $TJe_{j}Jf_{i} \in \overline{\mathbb{B}(L^{2}M)\mathbb{X}}^{\|\,\cdot\,\|}$ for every $i\in I, j\in J$, so the order of taking the $\mathbb{C}$-$M$ and $\mathbb{C}$-$M'$-topology closure does not matter. Thus $ \mathbb{K}_{\mathbb{X}}^{L}(M) $ is also the closure of $\mathbb{B}(L^2M)\mathbb{X}$ under the finest locally convex topology on $\BB(L^2M)$ contained in both $\mathbb{C}$-$M$ and $\mathbb{C}$-$M'$-topologies, which is given by the seminorms
$$  r^r_{\omega}(T)\coloneqq  \inf\{ \|Z\|(\omega(Jc^*cJ)+\omega(d^*d))^{1/2}\;\big|\; T= Z  \begin{pmatrix}
    c\\
    d
\end{pmatrix}, c\in M',d\in M,Z\in R_2(\BB(L^2M))\} , \quad \omega \text{ normal state on } M.$$

Let
\[\mathbb{K}_{\mathbb{X}}(M) = (\mathbb{K}_{\mathbb{X}}^{L}(M))^{*} \cap \mathbb{K}_{\mathbb{X}}^{L}(M) = (\mathbb{K}_{\mathbb{X}}^{L}(M))^{*}\mathbb{K}_{\mathbb{X}}^{L}(M) \subset \mathbb{B}(L^{2}M)\]
be the hereditary C$^{*}$-subalgebra of $\mathbb{B}(L^{2}M)$ associated with $\mathbb{K}_{\mathbb{X}}(M)$, and note that both $M$ and $M'$ are in its multiplier algebra. We also define
\[\mathbb{K}_{\mathbb{X}}^{\infty,1}(M) \coloneqq  \overline{\overline{\mathbb{K}_{\mathbb{X}}(M)}^{M\text{-}M}}^{M'\text{-}M'}\]
to be the closure of $\mathbb{K}_{\mathbb{X}}(M)$ in $\mathbb{B}(L^{2}M)$ in (the finest locally convex topology contained in) the $M$-$M$ and $M'$-$M'$-topology.

In fact, since $  \mathbb{K}_{\mathbb{X}}(M)$ is a linear subspace, $ \overline{\overline{\mathbb{K}_{\mathbb{X}}(M)}^{M\text{-}M}}^{M'\text{-}M'} $ is also the closure of $ \mathbb{K}_{\mathbb{X}}(M) $ with respect to the weak topology given by the functionals $\varphi\in \BB(L^2M)^{ M\sharp M}\cap \BB(L^2M)^{ M'\sharp M'}$. Since $ \BB(L^2M)^{ M\sharp M}\cap \BB(L^2M)^{ M'\sharp M'} = \BB(L^2M)^{ M\sharp M'}\cap \BB(L^2M)^{ M'\sharp M} $ by \cite[Proposition 2.1]{MR4675043},  $ \overline{\overline{\mathbb{K}_{\mathbb{X}}(M)}^{M\text{-}M}}^{M'\text{-}M'} $ is also the closure of $ \mathbb{K}_{\mathbb{X}}(M) $ with respect to the finest locally convex topology on $\mathbb{X}$ contained in the $M$-$M$, $M$-$M'$, $M'$-$M$, and $M'$-$M'$-topology, which is given by seminorms of the form
\begin{multline*} r_\omega(T)\coloneqq  \inf\{ (\omega(Ja^*aJ)+\omega(b^*b))^{1/2}\|Z\|(\omega(Jc^*cJ)+\omega(d^*d))^{1/2} \\ \;\big|\; T= \begin{pmatrix}
    a\\
    b
\end{pmatrix}^* Z  \begin{pmatrix}
    c\\
    d
\end{pmatrix}, a,c\in M',b,d\in M,Z\in \mathbb{M}_2(\BB(L^2M)\}, 
\end{multline*}
where $\omega$ is a normal state on $M$.

The space $\mathbb{K}_{\mathbb{X}}^{\infty,1}(M)$ can also be obtained by taking the closure of $ \mathbb{X}$ under the $M$-$M$ and $M'$-$M'$-topology. The case when $M$ is tracial has been proven in \cite[Proposition 3.6]{MR4675043}, and we prove for the general case in the following lemma.

\begin{lemma}[cf.\ \protect{\cite[Proposition 3.6]{MR4675043}}]\label{lem:density of X}
    $ \mathbb{K}_{\mathbb{X}}^{\infty,1}(M) = \overline{\overline{\mathbb{X}}^{M\text{-}M}}^{M'\text{-}M'} $.
\end{lemma}

\begin{proof}

Since $\mathbb{X}\subset \mathbb{K}_{\mathbb{X}}(M)$, it suffices to show that $ \mathbb{K}_{\mathbb{X}}(M) \subset \overline{\overline{\mathbb{X}}^{M\text{-}M}}^{M'\text{-}M'}  $. Fix an element $S^*R\in \mathbb{K}_{\mathbb{X}}(M)$ with $S,R\in \mathbb{K}_{\mathbb{X}}^L(M)$. Choose nets $ S_n$ and $R_n \in  \mathbb{B}(L^2M)\mathbb{X} $ converging to $S$ and $R$, respectively, in the (finest topology contained in) $\mathbb{C}$-$M$ and $\mathbb{C}$-$M'$ topology. We must prove that $ S_n^* R_n $ converges to $S^*R$ in the $M$-$M$ and $M'$-$M'$-topology. 

Let $\omega_1,\cdots, \omega_k$ be normal states on $M$. Since $ S_n $ converges to $S$ in the $\mathbb{C}$-$M$ and $\mathbb{C}$-$M'$ topology, we have $ r^r_{\omega_i}(S_n-S)\to 0 $ for each $i$. %where again for each faithful normal state $\omega$, 
%\[ r^r_{\omega}(T)\coloneqq  \inf\{ \|Z\|%(\omega(Jc^*cJ)+\omega(d^*d))^{1/2}\;\big|\; T= Z  \begin{pmatrix}
%    c\\
%    d
%\end{pmatrix}, c\in M',d\in M,Z\in R_2(X)\} \]
%is the seminorm generating the $\mathbb{C}$-$M$ and $\mathbb{C}$-$M'$ topology. 
Similarly, $ r^r_{\omega_i}(R_n-R)\to 0 $. Hence, for any $\varepsilon >0$, there exists $n_0$ such that for all $n>n_0$ and for every $i$, we can find $Z_n, Z_n' \in R_2(\BB(L^2M))$, $a_n, c_n \in M'$, $b_n, d_n \in M$ satisfying
\[S_n - S = Z_n \begin{pmatrix}
    c_n\\
    d_n
\end{pmatrix}, \quad R_n-R= Z'_n \begin{pmatrix}
    a_n\\
    b_n
\end{pmatrix},\]
with
\[\|Z_n\|(\omega_i(Jc_n^*c_nJ)+\omega_i(d_n^*d_n))^{1/2}<\varepsilon, \quad  \|Z'_n\|(\omega_i(Ja_n^*a_nJ)+\omega_i(b_n^*b_n))^{1/2} <\varepsilon.\]
It follows that
\[ r_{\omega_i}( (S_n-S)^*(R_n-R) ) \leqslant (\omega_i(Jc_n^*c_nJ)+\omega_i(d_n^*d_n))^{1/2}\|Z^*_nZ'_n\|(\omega_i(Ja_n^*a_nJ)+\omega_i(b_n^*b_n))^{1/2}< \varepsilon^2, \]
so $ (S_n-S)^*(R_n-R)\to 0 $ in the $M$-$M$ and $M'$-$M'$-topology. Finally, since 
\[ S_n^*R_n-S^*R =(S_n-S)^*(R_n-R)+ (S_n^*-S^*)R+S^*(R_n-R), \]
we also obtain $ S_n^*R_n-S^*R  \to 0$ in the $M$-$M$ and $M'$-$M'$-topology, because $ (S_n^*-S^*)R\to 0 $ in the $M$-$\mathbb{C}$ and $M'$-$\mathbb{C}$-topology, which is stronger than the $M$-$M$ and $M'$-$M'$-topology, and similarly $ S^*(R_n-R) \to 0$ in the $\mathbb{C}$-$M$ and $\mathbb{C}$-$M'$-topology.
\end{proof}

We now recall the definition of small-at-infinity boundary and proper proximality in \cite{MR4675043} and relative biexactness in \cite{ding2023biexact}.
\begin{defn}[\protect{\cite{MR4675043}}]
    The \textit{small-at-infinity boundary $\mathbb{S}_{\mathbb{X}}(M)$ of $M$ with respect to the $M$-boundary piece $\mathbb{X}$} is
\[\mathbb{S}_{\mathbb{X}}(M) = \{T\in \mathbb{B}(L^{2}M): [T,x] \in \mathbb{K}_{\mathbb{X}}^{\infty,1}(M) ,\  \forall x \in M'\}.\]
In general, $\mathbb{S}_{\mathbb{X}}(M)$ is only an operator system containing $M$; it does not necessarily form a C$^{*}$-algebra.
\end{defn}

An advantage of this definition is that, to check whether
$T\in\mathbb S_{\mathbb X}(M)$, it suffices to verify
$[T,a]\in\mathbb K_{\mathbb X}^{\infty,1}(M)$ for $a$ in a
self-adjoint generating set of $M'$.
Indeed, $M'$-bimodularity of
$\mathbb K_{\mathbb X}^{\infty,1}(M)$ and the Leibniz rule
extend the condition to the generated unital $*$-algebra, which can then be extended to $M'$ by \cite[Lemma 6.1]{MR4675043}.

\begin{defn}[\protect{\cite[Theorem 6.2]{MR4675043}}]
    The von Neumann algebra $M$ is \textit{properly proximal relative to the boundary piece $\mathbb{X}$} if there does not exist an $M$-central state on $\mathbb{S}_{\mathbb{X}}(M)$ that is normal on $M$.

    More generally, if $p\in M$ is a projection and $A\subset pMp$ is a von Neumann subalgebra, then we say $A$ is \textit{properly proximal relative to $\mathbb{X}$ inside $M$} if there is no $A$-central state on $p\mathbb{S}_{\mathbb{X}}(M)p$ that is normal on $pMp$.
    
    $M$ is called \textit{properly proximal} if $M$ is properly proximal relative to $\mathbb{K}(L^2M)$.
\end{defn}

We have a bidual characterization of tracial properly proximal von Neumann algebras.

\begin{lemma}[\protect{\cite[Theorem 8.5]{MR4675043}}]
    Let $M$ be a separable tracial von Neumann algebra with an $M$-boundary piece $\mathbb{X}$. Then M is properly proximal relative to $\mathbb{X}$ if and only if there does not exist $M$-central state $\phi$ on
    \[\widetilde{\mathbb{S}}_{\mathbb{X}}(M) \coloneqq \bigg\{T \in \bigg(\mathbb{B}(L^2M)_{J}^{\sharp}\bigg)^{*} : [T,a]\in \bigg(\mathbb{K}_{\mathbb{X}}(M)_{J}^{\sharp}\bigg)^{*} \text{ for all }a\in JMJ\bigg\}\]
    such that $\restr{\phi}{M}$ is normal.
\end{lemma}

\begin{defn}[\protect{\cite[Definition 6.1]{ding2023biexact}}]
    The von Neumann algebra $M$ is \textit{biexact relative to the boundary piece $\mathbb{X}$} if the inclusion $M \subset \mathbb{S}_{\mathbb{X}}(M)$ is $M$-nuclear. $M$ is called \textit{biexact} if $M$ is biexact relative to $\mathbb{K}(L^2M)$.
\end{defn}
Note that a relatively biexact von Neumann algebra $M$ is necessarily weakly exact because the small-at-infinity boundaries are always contained in $\mathbb{B}(L^{2}M)$.

We recall following example of boundary pieces in \cite{ding2023biexact} which is related to our main results.

\begin{exmp}[boundary piece generated by subalgebras]
    Suppose $\mathcal{B} = \{(B_{i}, E_{i})\}_{i\in I}$ is a collection of unital von Neumann subalgebras $B_{i}$ of $M$ with normal faithful conditional expectations $E_{i}\colon M\to B_{i}$, and $e_{B_{i}} \in \mathbb{B}(L^{2}M)$ is the orthogonal projection onto the space $L^{2}B_{i} \subset L^{2}M$ corresponding to $E_{i}$. We let $\mathbb{X}_{\mathcal{B}}$ denote the boundary piece in $\mathbb{B}(L^{2}M)$ generated by $\{xJyJ e_{B_{i}}: x,y\in M, i\in I\}$, and we say that \textit{$M$ is biexact relative to $\mathcal{B}$} if it is biexact relative to boundary piece $\mathbb{X}_{\mathcal{B}}$. If the conditional expectations are understood from the context, then we may simply write the boundary piece by $\mathbb{X}_{\{B_{i}\}_{i\in I}}$ and say $M$ is biexact relative to $\{B_{i}\}_{i \in I}$ instead. If moreover $\{B_{i}\}_{i\in I} = \{B\}$ consists of only one subalgebra $B$ of $M$, then we further omit the curly brackets in the notation.

    As noted in the introduction, this generalizes the notion of relative biexactness in the group case defined in \cite[Definition 15.1.2]{MR2391387}, as it is proved in \cite[Theorem 6.2]{ding2023biexact} that a countable discrete group $\Gamma$ is biexact relative to a collection $\mathcal{G}$ of subgroups of $\Gamma$ if and only if the group von Neumann algebra $L\Gamma$ is biexact relative to $\{L\Lambda:\Lambda \in \mathcal{G}\}$ in the above sense.

    When $M = M_{1}\overline{*}_{B}M_{2}$ is the AFP von Neumann algebras over a common unital von Neumann subalgebra $B$, we let $e_i\in \mathbb{B}(L^2M) $ be the projection onto the subspace $L^2M_i\subset  L^2M$. Then the boundary piece generated by $e_{i}\mathbb{B}(L^2 M_1)e_j$ is the boundary piece $\mathbb{X}_{\{M_{1},M_{2}\}}$.
\end{exmp}

As mentioned in the introduction, using \cite[Proposition 6.13]{ding2023biexact} we can easily obtain a structural indecomposability result for subalgebras of the biexact von Neumann algebras relative to subalgebras. In the remaining of this section, we outline a sketch of proof of this result for the reader's convenience. We first recall following celebrated intertwining-by-bimodule theorem for tracial von Neumann algebras due to Popa \cite[Theorem 2.1]{MR2231961}, which was then extended by Houdayer and Isono \cite[Theorem 4.3]{MR3570140} to the type III case.

\begin{thm}[\protect{\cite[Theorem 2.1]{MR2231961}, \cite[Theorem 4.3]{MR3570140}}] \label{thm: intertwining by bimodule}
    Let $M$ be a von Neumann algebra, $B \subset M$ a von Neumann subalgebra with expectation, and $N\subset M$ a finite von Neumann subalgebra with expectation. The following are equivalent:
    \begin{enumerate}
        \item There exist projections $e\in\mathcal{P}(N), f \in \mathcal{P}(B)$, a nonzero partial isometry $v \in eMf$, and a unital normal $*$-homomorphism $\phi\colon eNe \to fBf$ such that the inclusion $\phi(eNe) \subset fBf$ is with expectation and $xv = v\phi(x)$ for every $x\in eNe$.
        \item There does not exist a net of unitaries $(u_{j}) \subset \mathcal{U}(N)$ such that $E_{B}(au_{i}b) \to 0$ in the ultrastrong-$*$ topology for all $a,b\in M$.
    \end{enumerate}
    If one of the above equivalent statements holds, we say \textit{$N$ embeds with expectation into $B$ inside $M$} and write $N \preceq_{M} B$.
\end{thm}

\begin{prop}[\protect{\cite[Proposition 6.13]{ding2023biexact}}]\label{prop DP23 6 13}
    Let $ M $ be a $\sigma$-finite von Neumann algebra, and $ \{B_j\}_{j\in J} $ a countable family of von Neumann subalgebras with expectation. If $M$ is biexact relative to $\mathbb{X}_{\{B_j\}_{j\in J}}$ and $N\subset M$ is a finite von Neumann subalgebra with expectation, then either $N \preceq_{M} B_j$ for some $j\in J$, or $ N'\cap M $ is amenable.
\end{prop}
\begin{proof}
     Suppose $N \npreceq_{M} B_j$ for any $j\in J$.  By Lemma \ref{lem: not intertwine ad conv 0} below, there exists a net $(u_i)\subset \mathcal{U}(N)$ such that $ \Ad(u_i)(S)=u_i^*Su_i\to 0 $ ultraweakly for all $ S\in \mathbb{K}_{\mathbb{X}_\mathcal{B}}^{\infty,1}(M)$. Up to passing to a subnet, we may assume the net $ \Ad(u_i) \colon \BB(L^2M)\to \BB(L^2M) $ admits a limit $ \theta\coloneqq  \lim_{i} \Ad(u_i) $ in the point-ultraweak topology on $ \mathbb{B}(L^2M)_1 $. Then $ \theta $ is a u.c.p.\ map such that $ \theta|_{\mathbb{K}_{\mathbb{X}_\mathcal{B}}^{\infty,1}  } = 0 $ and $\theta|_{N'\cap M} = \id$. Therefore, from the definition of $\mathbb{S}_{\mathbb{X}_\mathcal{B} 
 }(M)$, we have $ \theta(\mathbb{S}_{\mathbb{X}_\mathcal{B} 
 }(M)  ) \subset M$.

 Let $\varphi$ be a faithful normal state on $M$ such that $\restr{\varphi}{N}$ is tracial and preserves the conditional expectation from $M$ onto $N$. As $u_i$ lies in the centralizer of $\varphi$, the u.c.p.\ map $\theta|_M$ preserves $\varphi$ and is therefore normal from $M$ to itself. Now consider
 \[ N'\cap M\subset M\subset \mathbb{S}_{\mathbb{X}_\mathcal{B} 
 }(M) \xrightarrow{\theta} M. \]
 Since the inclusion $M \subset \mathbb{S}_{\mathbb{X}_\mathcal{B}}(M)$ is $M$-nuclear and $\theta$ is weak-$M$ continuous because $\restr{\theta}{M}$ is normal (see \cite[Lemma 3.7]{ding2023biexact}), the inclusion $N' \cap M \subset M$ is weakly nuclear, as the weak-$M$ topology on $M$ coincides with the ultraweak topology. Moreover, since $N' \cap M$ is globally invariant under the modular automorphism group $\sigma^\varphi$, there exists a faithful normal conditional expectation $E\colon M \to N' \cap M$. Consequently, the composition $N' \cap M \xrightarrow{\theta} M \xrightarrow{E} N' \cap M$ is weakly nuclear, which implies that $N' \cap M$ is amenable.
\end{proof}

\begin{lemma}[\protect{{\cite[Lemma 6.6, Lemma 6.12]{ding2023biexact}}}] \label{lem: not intertwine ad conv 0}
     Let $ M $ be a $\sigma$-finite von Neumann algebra and $ \mathcal{B}=\{B_j\}_{j\in J} $ be a countable family of von Neumann subalgebras with expectation. If $N\subset M$ is a finite von Neumann subalgebra with expectation and either $N \npreceq_{M} B_j$ for all $j\in J$, then there exists a net $(u_i) \subset \mathcal{U}(N)$ such that 
     \[ \Ad (u_i)(S)=u_i^*Su_i\to 0 \] 
     ultraweakly for all $ S\in \mathbb{K}_{\mathbb{X}_\mathcal{B}}^{\infty,1}(M) $.
\end{lemma}
\begin{proof}
     Since $N \npreceq_{M} B_j$ for any $j\in J$, by \cite[Theorem 4.3]{MR3570140}, considering the direct sum $ \tilde{B}\coloneqq \oplus_{j\in J} B_j \subset   \tilde{M}= \oplus_{j\in J} M $ so that $ N \npreceq_{\tilde{M}} \tilde{B}$ under the identification $N\simeq \{(x,\cdots,x,\cdots):x\in N\}\subset \tilde{M}$, we can choose a net of unitaries $ (u_i)\subset N$ such that $ E_j(au_ib)\to 0 $ strongly for all $a,b\in M$ and $j\in J$.
     
     We first claim that $ \Ad(u_i)(S)\to 0 $ ultraweakly for all $S\in \mathbb{X}_{ \mathcal{B} }$. Indeed, for any $ T\in \mathbb{B}(L^2M) $, $x_1, x_2,a_1,a_2\in M$, and $j_1,j_2\in J$, choose faithful normal states $\varphi_l$ on $M$ with $ \varphi_l\circ E_{j_l}=\varphi_l $ for $l=1,2$. Then
     \[ | \langle u_i^*a_1e_{B_{j_1}}  Te_{B_{j_2}}a_2u_i x_2\varphi_2^{1/2},x_1\varphi_1^{1/2}\rangle |\leqslant \|T\|\| E_{j_1}(a_1^*u_ix_1)\varphi_1^{1/2} \|\| E_{j_2}(a_2u_ix_2)\varphi_2^{1/2} \|\to 0. \]
    Since $\{x_l\varphi_{l}^{1/2}:x_l \in M\}$ is dense in $L^2M$ for each $l=1,2$, it follows that $ u_i^*a_1e_{B_{j_1}}  Te_{B_{j_2}}a_2u_i \to 0 $ ultraweakly as it forms a uniformly bounded net. Moreover, since each $u_i$ commutes with $M'$, for any $ b_1,b_2\in M' $, we also have $ \Ad(u_i)( a_1b_1e_{B_{j_1}}  Te_{B_{j_2}}b_2a_2 ) \to 0 $ ultraweakly. Since the span of elements $ a_1b_1e_{B_{j_1}}  Te_{B_{j_2}}b_2a_2 $ is norm dense in $ \mathbb{X}_{\mathcal{B} } $, we conclude that $ \Ad(u_i)(S)\to 0 $ ultraweakly for any $ S\in \mathbb{X}_{ \mathcal{B} } $.

    We now extend this to $ S\in \mathbb{K}_{\mathbb{X}_\mathcal{B}}^{\infty,1}(M) = \overline{\overline{\mathbb{X}_{\mathcal{B}}}^{M\text{-}M}}^{M'\text{-}M'} $. Fix a faithful normal state $\varphi$ on $M$ with $\varphi|_{N}$ tracial and preserves the conditional expectation from $M$ onto $N$. Again, since $\Ad(u_i)(S)$ is uniformly bounded, it suffices to show that $ \langle x\varphi^{1/2},\Ad(u_i)(S)y\varphi^{1/2} \rangle \to 0 $ for all $x,y\in M$ analytic with respect to  $\sigma^{\varphi}$. Suppose $S$ has the decomposition \[S= \begin{pmatrix}
    a\\
    b
\end{pmatrix}^* Z  \begin{pmatrix}
    c\\
    d
\end{pmatrix}, \quad a,c\in M',b,d\in M,Z\in \mathbb{M}_2(X), \]
then
\begin{equation*}
\begin{aligned}
    &|\langle x\varphi^{1/2},\Ad(u_i)(S)y\varphi^{1/2} \rangle| = \langle \begin{pmatrix}
    a\\
    b
\end{pmatrix}u_ix\varphi^{1/2},  Z  \begin{pmatrix}
    c\\
    d
\end{pmatrix}u_iy\varphi^{1/2} \rangle \\
\leqslant& \left[ \langle \varphi^{1/2},x^*u_i^*(a^*a+b^*b)u_ix \varphi^{1/2}\rangle \right]^{1/2} \|Z\|\left[ \langle \varphi^{1/2},y^*u_i^*(c^*c+d^*d)u_iy \varphi^{1/2}\rangle \right]^{1/2}.
\end{aligned}
\end{equation*}
Since $ \langle \varphi^{1/2},x^*u_i^*a^*au_ix \varphi^{1/2}\rangle \leqslant \|x\|^2\langle \varphi^{1/2},a^*a \varphi^{1/2}\rangle $ and $$ \langle \varphi^{1/2},x^*u_i^*b^*bu_ix \varphi^{1/2}\rangle = \langle J\sigma_{-\tfrac{i}{2}}(x^*)J\varphi^{1/2},u_i^*b^*bu_i J\sigma_{-\tfrac{i}{2}}(x^*)J\varphi^{1/2}\rangle \leqslant \|\sigma_{-\tfrac{i}{2}}(x^*)\|^2 \langle \varphi^{1/2},b^* b\varphi^{1/2} \rangle$$ as $ u_i $ is in the centralizer, repeating the same estimates for $c$ and $d$, we obtain
\[ |\langle x\varphi^{1/2},\Ad(u_i)(S)y\varphi^{1/2} \rangle|\leqslant r_{\varphi}(S)\max\{\|x\|,\|y\|,\|\sigma_{-\tfrac{i}{2}}(x^*)\|,\|\sigma_{-\tfrac{i}{2}}(y^*) \| \}. \]
By Lemma \ref{lem:density of X}, there exists a sequence $S_n \in \mathbb{X}_{\mathcal{B}}$ such that $ r_\varphi( S_n-S )\to 0 $, so
\[ \limsup_{i} |\langle x\varphi^{1/2},\Ad(u_i)(S)y\varphi^{1/2} \rangle|\leqslant \limsup_{i} |\langle x\varphi^{1/2},\Ad(u_i)(S-S_n)y\varphi^{1/2} \rangle| =0. \qedhere\]
\end{proof}

\section{Relative biexactness of amalgamated free product}

In this section, we will always assume that $M_{i}$, $i=1,2,\cdots,n$, are von Neumann algebras containing a copy of an injective von Neumann subalgebra $B$ with faithful normal conditional expectation $E_{i}\colon M_{i}\to B$. Moreover, we assume that each $M_{i}$ contains a unital ultraweakly dense C$^{*}$-subalgebra $A_{i}$ that is weakly exact in $M_i$ and a fixed unital C$^{*}$-subalgebra $D\subset A_{i}$ such that $E_{i}(A_{i}) = D$. For each $i$, we let $N_{i} = \langle M_{i}, e_{i}\rangle$ denote the basic construction for $B\subset M_{i}$. We denote by $M = M_{1}\overline{*}_{B}\cdots \overline{*}_{B} M_{n}$ the AFP von Neumann algebra of $M_i$'s and $A = A_{1}*_{D}\cdots *_{D} A_{n}$ the (reduced) AFP C$^{*}$-algebra of $A_i$'s. 

By Lemma \ref{lem: weak exactness into smaller injective}, the inclusion $A_i \subset N_i$ is $(A_i \subset M_i)$-nuclear. However, this condition is not always sufficient, and in certain cases, a stronger assumption is required: for each $i = 1, \dots, n$, there exists a C$^*$-subalgebra $N_{i,0} \subset N_i$ such that $C^*(M_i, e_B) \subset N_{i,0} \subset N_i$, and the inclusion $A_i \subset N_{i,0}$ is $(A_i \subset M_i)$-nuclear.

Our main goal in the next section is to prove the following theorem.

\begin{thm}\label{thm: main 3}
     For $i = 1,\cdots,n$, let $M_i$ be von Neumann algebras with a common von Neumann subalgebra $B\subset M_i$ and faithful normal conditional expectations $ E_{i}:M_i\to B $. For each $i$, let also $A_i\subset M_i$ be an ultraweakly dense unital C$^*$-subalgebra such that $ E_{i}(A_i)=D $ is a common C$^*$-subalgebra of $A_i$. Suppose that $ N_{i,0} $ is a C$^*$-subalgebra of $N_i=\langle M_i,e_B \rangle$ such that $C^*(A_i,B,e_B)\subset N_{i,0}\subset N_i$.

    If the inclusion $A_i\subset N_{i,0}$ is $(A_i\subset M_i)$-nuclear for $i=1,\cdots,n$, then the inclusion $$ A=A_1\ast_D\cdots \ast_D A_n \subset N_{1,0}\ast_{B}\cdots \ast_B N_{n,0}$$ is $ (A\subset M) $-nuclear.
\end{thm}

\subsection{Toeplitz-Pimsner algebra and $M$-nuclear embedding}

We begin by treating the case of the AFP of two von Neumann algebras, that is, when $n = 2$. Let again $M = M_{1}\overline{*}_{B} M_{2}$ be the AFP von Neumann algebra, and $A = A_{1}*_{D}A_{2} = C^*(\lambda_1(A_1),\lambda_2(A_2))\subset \mathbb{B}(\FF)$ be the (reduced) AFP C$^{*}$-algebra, where $\mathcal{F}$ is the associated free product Hilbert space in Section \ref{sec: Amalgamated free product}.

The following constructions are inspired by \cite[Section 4.8]{MR2391387} and \cite{MR4833744}. We choose to formulate the construction without referring to the n.s.f.\ weight (or state) on $B$ so that the arguments still hold for non $\sigma$-finite $M_i$.

We consider the Hilbert space $ \KK \coloneqq  \KK_1\oplus \KK_2$, where
$$ \KK_i \coloneqq  \bigoplus_{n=1}^{\infty} \bigoplus_{\substack{i_1\neq \cdots\neq i_n\\ i_1 = i}} L^2M_{i_1}\otimes_B \cdots \otimes_B L^2M_{i_n}.$$
%Let $ N_i \coloneqq  \langle M_i, e_{B}\rangle\subset  \BB(L^2M_i)$. 
Since each $L^2M_{i_1}\otimes_B \cdots \otimes_B L^2M_{i_n}$ is a left $ N_{i_1}$-module, we consider the left action of $ N_{i_1} $ on $ L^2M_{i_1}\otimes_B \cdots \otimes_B L^2M_{i_n} $ as the natural action if $i=i_1$ and as $0$ if $i\neq i_1$. In particular, we obtain a left representation $$ \iota:N_1\oplus N_2 \to \mathbb{B}( \KK) .$$ In the sequel, we will identify $N_{1}\oplus 0$ with $N_{1}$ and $0 \oplus N_{2}$ with $N_{2}$, and often simply write $n_{1}$ instead of $n_{1}\oplus 0$ for $n_{1}\in N_{1}$, and similarly write $n_{2}$ instead of $0\oplus n_{2}$ for $n_{2}\in N_{2}$.

Let $T\in \BB(\KK) $ be the isometry sending $ L^2M_{i_1}\otimes_B \cdots \otimes_B L^2M_{i_n} $ to $ L^2M_j\otimes_B L^2M_{i_1}\otimes_B \cdots \otimes_B L^2M_{i_n} $ via the embedding
$$ L^2M_{i_1}\otimes_B \cdots \otimes_B L^2M_{i_1}\simeq L^2B\otimes_B L^2M_{i_1}\otimes_B \cdots \otimes_B L^2M_{i_n}\hookrightarrow L^2M_j\otimes_B L^2M_{i_1}\otimes_B \cdots \otimes_B L^2M_{i_n}, $$
where $j=1,2$ such that $j\neq i_1$. Namely, $T$ is a ‘‘creation operator'' that exchanges $\KK_1$ and $\KK_2$. One can directly check that $T^*\in \BB(\KK)$ is the composition
\[L^2M_{i_1}\otimes_B \cdots \otimes_B L^2M_{i_n}\to   L^2B\otimes_B L^2M_{i_2}\otimes_B\cdots \otimes_B L^2M_{i_n} \simeq L^2M_{i_2}\otimes_B\cdots \otimes_B L^2M_{i_n},\]
where the first map is the projection onto $ L^2B\otimes_B L^2M_{i_2}\otimes_B\cdots \otimes_B L^2M_{i_n} $. Also, when $n=1$, $T^*(L^2M_i) = 0$. Using the notation in Example \ref{exmp: notation for relative tensor product}, we can write $T$ and $T^*$ as
\[ T\Xi = \varphi^{1/2}\otimes \Xi,\quad T^* (\varphi^{1/2}\otimes \Xi) = \Xi,\quad \forall \Xi\in  L^2M_{i_1}\otimes_B \cdots \otimes_B L^2M_{i_n},\]
and
\[ T^*(v\otimes \Xi) =R_{ \Xi }e_B(v),\quad \forall \Xi\in D'( L^2M_{i_1}\otimes_B \cdots \otimes_B L^2M_{i_n},\varphi ).\]
This leads to the following formula, with the $\sigma$-finite case proven in \cite{MR4833744}:
\begin{lemma}\label{lemma T*aT} 
Keeping the same notations as above, we have
    \[ T^*(a_1\oplus a_2)T =  E_2(a_2)\oplus E_1(a_1),\quad \forall a_i\in N_i,\]
    where $E_i\colon N_i\to B$ is the canonical conditional expectation introduced in Section \ref{sec: Amalgamated free product}.
\end{lemma}
\begin{proof}
    Let $ \{e_j\}_{j\in J} $ be an increasing net in $B$ strongly converging to $1$ such that $ \varphi(e_j)<\infty $. Then for any $ \Xi\in  D'(L^2M_{i_1}\otimes_B \cdots \otimes_B L^2M_{i_n},\varphi) $,
    $ T\Xi = \lim_{j}e_j\varphi^{1/2}\otimes \Xi$. Therefore,
    \begin{equation*}
    \begin{aligned}
        T^*(a_1\oplus 0)T\Xi &= \delta_{2,i_1}\lim_{j}T^*a_1e_j\varphi^{1/2}\otimes\Xi =\delta_{2,i_1}\lim_{j}R_\Xi\left(E_1(a_1)e_j\varphi^{1/2}\right)\\ &=  \delta_{2,i_1}\lim_{j}E_1(a_1) e_j\Xi = \delta_{2,i_1}E(a_1)\Xi = (0\oplus E_1(a_1))\Xi.
    \end{aligned}
    \end{equation*}
    Similarly, we have $ T^*(0\oplus a_2)T= E_2(a_2)\oplus 0 $.
\end{proof}

The normal u.c.p.\ map
\[\phi\colon M_{1}\oplus M_{2} \ni (m_{1},m_{2})\mapsto (E_{2}(m_{2}),E_{1}(m_{1})) \in B\oplus B \subset M_{1}\oplus M_{2}\]
restricts to a u.c.p.\ map $A_{1}\oplus A_{2} \ni (a_{1} \oplus a_{2})\mapsto (E_{2}(a_{2}) \oplus E_{1}(a_{1})) \in D\oplus D \subset A_{1}\oplus A_{2}$. Since the algebraic tensor product $(A_{1}\oplus A_{2})\odot (A_{1}\oplus A_{2})$ admits the left and right $A_{1}\oplus A_{2}$-actions given by $a\cdot (x\otimes y)\cdot b = ax\otimes yb$ for $a,b,x,y\in A_{1}\oplus A_{2}$ and the $A_{1}\oplus A_{2}$-valued semi-inner product 
\[\langle a_{1}\otimes b_{1}, a_{2}\otimes b_{2}\rangle = b_{1}^{*}\phi(a_{1}^{*}a_{2})b_{2}, \quad a_{1},a_{2},b_{1},b_{2}\in A_{1}\oplus A_{2},\]
we let $\mathcal{H} = \mathcal{H}^{\phi}_{A_{1}\oplus A_{2}}$ denote the Hilbert C$^{*}$-correspondence over $A_{1}\oplus A_{2}$ associated with $\phi$ obtained from $(A_{1}\oplus A_{2})\odot (A_{1}\oplus A_{2})$ by separation and completion with respect to the inner product mentioned above. As in \cite[Example 2.2]{MR4833744}, $\mathcal{H}$ is $(A_{1}\oplus A_{2}\subset M_{1}\oplus M_{2})$-normal, i.e., for any $\xi,\eta\in \mathcal{H}$ and $f\in (A_1\oplus A_2)^{(M_1\oplus M_2)\sharp(M_1\oplus M_2)}$, the map $A_1\oplus A_2\ni a\mapsto f(\langle \xi,a\eta\rangle)$ extends to a normal linear functional on $M_1\oplus M_2$.

Let $\tau\colon \mathcal{H} \to \mathbb{B}(\mathcal{K})$ be the linear map:
\[\tau((a_{1}\oplus a_{2})\otimes (b_{1} \oplus b_{2})) = (a_{1}\oplus a_{2})T(b_{1}\oplus b_{2}), \quad (a_{1}\oplus a_{2})\otimes (b_{1} \oplus b_{2})\in \mathcal{H}.\]
By \cite[Lemma 3.1]{MR4833744} and Lemma \ref{lemma T*aT}, $(\iota, \tau)$ is a representation of the C$^{*}$-correspondence $\mathcal{H}$ on $\mathbb{B}(\mathcal{K})$, and applying the gauge-invariant uniqueness theorem (see e.g., \cite[Theorem 4.6.18]{MR2391387} or \cite{MR1986889}) one can check that the Toeplitz-Pimsner C$^{*}$-algebra $\mathcal{T}(\mathcal{H})$ associated with $\mathcal{H}$ is canonically $*$-isomorphic to the C$^{*}$-subalgebra $C^{*}(A_{1}\oplus A_{2},T)$ of $\mathbb{B}(\mathcal{K})$:
\begin{lemma}\cite[Lemma 3.1]{MR4833744}
    $(\iota,\tau)$ is a universal representation of the C$^*$-correspondence $\HH$ on $\mathbb{B}(\KK)$. In particular, we have the isomorphism $$ C^*(A_1\oplus A_2,T) = \mathcal{T}(\HH) .$$
\end{lemma}

We define $p \coloneqq 1- T^{2}(T^{*})^{2}\in C^{*}(A_{1}\oplus A_{2},T)$, which is the orthogonal projection onto
\[\bigoplus_{i\neq j} \bigg(L^{2}M_{i} \oplus (L^{2}M_{i}\otimes_{B} L^{2}M_{j}) \oplus \big(\big((L^{2}M_{i}\otimes_{B} L^{2}M_{j}) \ominus (L^2B \otimes_{B} L^2B)\big)\otimes_{B} \mathcal{K}_{i}\big)\bigg).\]

For $i,j=1,2$, we define following subspaces of $p\mathcal{K}$:
\begin{equation*}
  \begin{aligned}
    \mathcal{K}_{1,1} &= L^{2}M_{1} \oplus \bigoplus L^{2}M_{1}\otimes_{B} L^{2}M_{2}^{o} \otimes_{B} \cdots \otimes_{B} L^{2}M_{2}^{o} \otimes_{B} L^{2}M_{1}\\
    \mathcal{K}_{1,2} &= (L^{2}M_{1}\otimes_{B} L^{2}M_{2}) \oplus \bigoplus L^{2}M_{1}\otimes_{B} L^{2}M_{2}^{o} \otimes_{B} \cdots \otimes_{B} L^{2}M_{1}^{o} \otimes_{B} L^{2}M_{2}\\
    \mathcal{K}_{2,1} &= (L^{2}M_{2}\otimes_{B} L^{2}M_{1}) \oplus \bigoplus L^{2}M_{2}\otimes_{B} L^{2}M_{1}^{o} \otimes_{B} \cdots \otimes_{B} L^{2}M_{2}^{o} \otimes_{B} L^{2}M_{1}\\
    \mathcal{K}_{2,2} &= L^{2}M_{2} \oplus \bigoplus L^{2}M_{2}\otimes_{B} L^{2}M_{1}^{o} \otimes_{B} \cdots \otimes_{B} L^{2}M_{1}^{o} \otimes_{B} L^{2}M_{2}
  \end{aligned}
\end{equation*}
Namely, $\mathcal{K}_{i,j}$ is the subspace consisting of the tensors beginning with vectors from $L^{2}M_{i}$, ending with those from $L^{2}M_{j}$, and in the middle vectors are from the orthogonal complements of $L^2B$, if there are any. Via the standard identifications $L^2B \otimes_{B} L^{2}M_{2}^{o} \simeq L^{2}M_{2}^{o}$ and $L^{2}M_{2}^{o} \otimes_{B} L^2B \simeq L^{2}M_{2}^{o}$, we have a natural left $ N_1 $-linear isomorphism $V_{1,1}\colon \mathcal{K}_{1,1}\to L^{2}M$. Similarly, by the identification $L^2B \otimes_{B} L^{2}M_{1}^{o} \simeq L^{2}M_{1}^{o}$ and $L^{2}M_{1}^{o} \otimes_{B} L^2B \simeq L^{2}M_{2}^{o}$, we have natural left $N_i$-linear identification maps $V_{i,j}\colon \mathcal{K}_{i,j}\to L^{2}M$ for $i,j=1,2$. %Notice that each $V_{i,j}$ is $N_{i}$-modular, in the sense that $V_{i,j}n_{i} = \lambda_{i}(n_{i})V_{i,j}$ for every $n_{i}\in N_{i}$. Indeed, this can be checked directly if $n_{i}\in M_{i}$ or $n_{i} = a_{i}e_{B}b_{i}$ for some $a_{i},b_{i}\in M_{i}$, and since the linear span of such elements is strongly-dense in $N_{i}$, our claim follows.

\iffalse
Let $u = p(T+ T^{*})p \in C^{*}(A,T)$. As seen in \cite[Section 4.8]{MR2391387}, $u$ is a self-adjoint partial isometry with $u^{2} = p$ that interchanges $p\mathcal{K}_{1}$ and $p\mathcal{K}_{2}$, via the canonical identifications among subspaces of $p\mathcal{K}$ as indicated in the following diagram:
\begin{equation*} 
  \begin{tikzcd}[column sep=small]
    L^{2}M_{1} \arrow[dr, leftrightarrow] & \xi_{1}B\otimes_{B}L^{2}M_{2}\arrow[dl, crossing over,leftrightarrow]&  & L^{2}M_{1}^{o}\otimes_{B}L^{2}M_{2} \arrow[dr,leftrightarrow] & \xi_{1}B\otimes_{B}L^{2}M_{2}^{o}\otimes_{B} L^{2}M_{1}\arrow[dl, crossing over,leftrightarrow] & \cdots\\
    L^{2}M_{2}  & \xi_{2}B\otimes_{B}L^{2}M_{1}&  & L^{2}M_{2}^{o}\otimes_{B}L^{2}M_{1}  & \xi_{2}B\otimes_{B}L^{2}M_{1}^{o}\otimes_{B} L^{2}M_{2} & \cdots
\end{tikzcd}
\end{equation*}
The two-sided arrows in the above diagram represent canonical identification maps. %In particular, from above diagram we see that $u$ interchanges $\mathcal{K}_{1,1}$ with $\mathcal{K}_{2,1}$, and $\mathcal{K}_{1,2}$ with $\mathcal{K}_{2,2}$ via canonical identifications, i.e., for any $\xi \in L^{2}M$, 
%\begin{equation} \label{graph: u action}
%  u(V_{1,1}^{*}\xi) = V_{2,1}^{*}\xi\,, \quad u(V_{2,1}^{*}\xi) = V_{1,2}^{*}\xi\,,\quad u(V_{1,2}^{*}\xi) = V_{2,2}^{*}\xi\,,\quad % u(V_{2,2}^{*}\xi) = V_{1,2}^{*}\xi.
%\end{equation}
\fi

%The following lemma is straightforward but very useful.

Recall that for each $i=1,2$, we denote by $e_B \colon L^2M_i \to L^2B$ the orthogonal projection onto $L^2B$. In what follows, we will primarily work with their images $\lambda_{i}(e_B) \in \B(\FF)$ under the $*$-homomorphisms $\lambda_i\colon N_i = \langle M_i, e_B\rangle \to \B(\FF)$ introduced in Section \ref{sec: Amalgamated free product}, so no confusion will arise from this abuse of notation of $e_B$. The following lemma about $\lambda_{i}(e_B)$ is straightforward but will be frequently used.
\begin{lemma}
     $\lambda_i(e_B) $ is the orthogonal projection from $ \FF $ onto the subspace of vectors consisting of those in $L^2B$ or tensors beginning with $ L^2M_{i+1}^o$, where the indices are interpreted as in $\mathbb{Z}/2\mathbb{Z}$. In particular, we have $ \lambda_1(e_B)\lambda_2(e_B)  = \lambda_2(e_B)\lambda_1(e_B) = e_B$, where the last operator $e_B$ in the equation denotes the orthogonal projection from $\mathcal{F}$ onto $L^2 B$.
     %and $ \lambda_{i+1}(e_B)\lambda_i(e_B) $ is the orthogonal projection onto the subspace of vectors in $ \FF $ beginning with $ L^2M_i^o $.
\end{lemma}

\begin{lemma}\label{lem: relation between T and Vij}
Keeping the same notations as above, we have
    \begin{equation*}
        T^*V^*_{i,j} = V^*_{i+1,j}V_{i+1,j}T^*V^*_{i,j},\quad V_{i,j}T = V_{i,j}TV^*_{i+1,j}V_{i+1,j},
    \end{equation*}
    where the indices are interpreted as in $\mathbb{Z}/2\mathbb{Z}$. Moreover, we have
    \[ V_{i+1,j}T^*V^*_{i,j} = \begin{cases}
        \lambda_{i+1}(e_B^\perp)\lambda_i(e_B) = \lambda_{i+1}(e_B)^\perp,\quad i=j\\
        \lambda_i(e_B),\quad i\neq j
    \end{cases} \]
\end{lemma}
\begin{proof}
    The proof for the first claim follows directly from the definition of $T$. Indeed, for $ T^*V^*_{1,1} $, we notice that $ T^*( \KK_{1,1} )\subset  \KK_{2,1} $. Since $V^{*}_{1,1}$ maps onto $ \KK_{1,1} $, and $ V^{*}_{i,j}V_{i,j} $ is the orthogonal projection onto $ \KK_{i,j} $, we must have $T^*V^*_{1,1} = V^{*}_{2,1}V_{2,1}T^*V^*_{1,1} $. The other cases follow in a similar way.

    The second claim is also straightforward if one assumes $B$ admits a faithful normal state. The proof for a general n.s.f.\ weight $\varphi$ follows as in the state case, provided one carefully represents vectors in  $D(\HH,\varphi) $ for a given $B$-module $\HH$. To see this, we check for $ V_{2,1}T^*V^*_{1,1} $. 
    
    For a vector $ \xi\in L^2B $, we have $ V^{*}_{1,1}\xi = \xi\in L^2M_1\subset\KK_1 $ is a vector with tensor length $1$. Since $T^*$ vanish on vectors with tensor length $1$, we have $T^*V^{*}_{1,1}\xi =0$.
    
    For a vector with tensor beginning with $ L^2M^o_2 $, it is of the form $ \xi^o \otimes \Xi \in L^2M^o_{2}\otimes_B \cdots\otimes_B L^2M^o_{i_n}$, where $ \xi^o\in L^2M_2^o $ and $ \Xi\in D'( L^2M^o_{1}\otimes_B \cdots\otimes_B L^2M^o_{i_n},\varphi ) $. Then
    \[V^*_{1,1}\xi^o\otimes \Xi = \varphi^{1/2}\otimes \xi^o\otimes \Xi\in L^2M_1\otimes_B L^2M_2\otimes_B \cdots \otimes_B L^2M_{i_n},\] so $ T^* V^*_{1,1}\xi^o\otimes \Xi = \xi^o\otimes \Xi\in L^2M_2\otimes_B \cdots \otimes_B L^2M_{i_n} $ and $ V_{2,1}T^* V^*_{1,1}\xi^o\otimes \Xi = \xi^o\otimes \Xi $. This shows that $ V_{2,1}T^* V^*_{1,1}$ is identity on $ \lambda_{2}(e_B^{\perp})\lambda_{1}(e_B)\FF $. 
    
    On the other hand, a vector beginning with $ L^2M_1^o $ is of the form $\xi^o\otimes \Xi = \xi^o\otimes \Xi\in L^2M_1^o\otimes_B L^2M_2^o\otimes_B \cdots \otimes_B L^2M_{i_n}^o$, where $ \xi^o\in L^2M_1^o $ and $ \Xi\in D'( L^2M^o_{2}\otimes_B \cdots\otimes_B L^2M^o_{i_n},\varphi) $. We have $ V^*_{1,1}\xi^o\otimes \Xi = \xi^o\otimes\Xi \in L^2M_1\otimes_B L^2M_2\otimes_B \cdots \otimes_B L^2M_{i_n} $. Since $ \xi^o\in L^2M_1^o $, we have $ T^*V^*_{1,1}\xi^o\otimes \Xi = 0 $. Hence $ V_{2,1}T^*V_{1,1}^* $ vanishes on $\lambda_{2}(e_B)$. 
    
    The same computation applies for general indices $i,j$.
\end{proof}

We define a u.c.p.\ map $\Phi\colon \mathbb{B}(\mathcal{K})\to \mathbb{B}(L^{2}M)$ by $\Phi(x) = \frac{1}{4}\sum_{i,j,k,l=1}^{2} V_{i,j}xV_{k,l}^{*}$, so that the support projection of $\Phi$ is no larger than the orthogonal projection from $\mathcal{K}$ onto $\bigoplus_{i,j=1,2}\mathcal{K}_{i,j}$. 
Therefore, $p\geqslant \supp(\Phi)$ and thus $\Phi(p) = 1$.

By Lemma \ref{lemma T*aT}, each nonzero word in $ N_1\cup N_2\cup \{T\} $ can be reduced to the form
\[ c_1T\cdots Tc_nTaT^*d_1T^*\dots T^* d_m,\quad c_i,a,d_i\in N_1\cup N_2. \]
This fact is used in the next corollary to show that $ \Phi( C^*(N_1\oplus N_2,T)) \subset N_1 {*}_B N_2\subset  \BB(\FF) $.

\begin{corollary}\label{cor: image of Phi 1}
    For each $i=1,2$, let $ N_{i,0} $ be a C$^*$-subalgebra of $N_i=\langle M_i,e_B\rangle$ such that $  C^*(A_i,B,e_B)\subset N_{i,0}\subset N_i  $, then
    \[\Phi( C^*(N_{1,0}\oplus N_{2,0},T)) \subset N_{1,0} {*}_B N_{2,0}\subset  \BB(\FF).\]
\end{corollary}
\begin{proof}
    We need to show that  $\Phi( C^*(N_{1,0}\oplus N_{2,0},T))\subset  N_{1,0}\ast_B N_{2,0}$. As $ \Phi(x) = \frac{1}{4}\sum_{i,j,k,l=1}^{2} V_{i,j}xV_{k,l}^{*} $, it suffices to show that $$ V_{i,j}c_1T\cdots Tc_mTaT^*d_1T^*\cdots T^* d_nV_{k,l}^{*}\in N_{1,0}\ast_B N_{2,0},\quad \forall c_s,a,d_s\in N_{1,0}\cup N_{2,0},\;i,j,k,l\in\{1,2\}.$$
    For this, we first note that for $ d\in N_i $, $ dV^*_{k,l} = \delta_{i,k} dV^*_{k,l} = \delta_{i,k} V^*_{k,l}\lambda_{i}(d) $ as $V_{k,l}$ is left $N_k$-linear. Therefore, as long as $ dV^*_{k,l} $ is nonzero, $ d $ commutes with $ V^*_{k,l} $. By Lemma \ref{lem: relation between T and Vij}, for a nonzero term $ d_1T^*d_2V^*_{k,l} $, we also have
    \[ d_1T^*d_2V^*_{k,l} = d_1V^*_{k+1,l}V_{k+1,l}T^*V^*_{k,l}\lambda_{k}(d_2) = V^*_{k+1,l}\lambda_{k+1}(d_1)\left(V_{k+1,l}T^*V^*_{k,l}\right)\lambda_{k}(d_2). \]
    Repeating the same process, we obtain that for any nonzero term $ V_{i,j}c_1T\cdots Tc_mTaT^*d_1T^*\cdots T^* d_nV_{k,l}^{*}$,
    \begin{align*}
        &V_{i,j}c_1T\cdots Tc_mTaT^*d_1T^*\dots T^* d_nV_{k,l}^{*}\\ =& \lambda_{i}(c_1)\left( V_{i,j}TV_{i+1,j}^* \right)\cdots \left( V_{i+m-1,j}TV_{i+m,j}^* \right)V_{i+m,j}aV_{k+n,l}^*\left( V_{k+n,l}T^*V^*_{k+n-1,l}
 \right)\cdots\left( V_{k+1,l}T^*V^*_{k,l}
 \right)\lambda_{k}(d_n).
    \end{align*}
    Finally, since $ V_{i,j}TV_{i+1,j}^*\in N_{1,0}\ast_B N_{2,0} $ by Lemma \ref{lem: relation between T and Vij}, and $ V_{i+m,j}aV_{k+n,l}^* = 0 $ if $ a\in N_{k+n+1,0} $ and $ V_{i+m,j}aV_{k+n,l}^* = V_{i+m,j}V_{k+n,l}^*\lambda_{k+n}(a) = \delta_{ i+m,k+n }\delta_{j,l}\lambda_{k+n}(a)\subset N_{1,0}\ast_B N_{2,0} $, we have
    \[V_{i,j}c_1T\cdots Tc_mTaT^*d_1T^*\dots T^* d_nV_{k,l}^{*}\in N_{1,0}\ast_B N_{2,0}. \qedhere\]
\end{proof}

From the computation above, we have $ \Phi(T) = \frac{1}{4}(\lambda_1(e_B)+ \lambda_1(e_B)^{\perp}+\lambda_2(e_B)+\lambda_2(e_B)^\perp) = \frac{1}{2} $. Similarly,
\begin{equation}\label{eq: Phi of T2T*2}
    4\Phi(T^2(T^*)^2) = \sum_{i=1,2}\lambda_i(e_B)\lambda_i(e_B)^\perp\lambda_i(e_B)^\perp\lambda_i(e_B) +\sum_{i=1,2}\lambda_i(e_B)^\perp\lambda_i(e_B)\lambda_i(e_B)\lambda_i(e_B)^{\perp} = 0.
\end{equation}
In fact, $\Phi(T^m(T^*)^n) = 0$ whenever $ \max\{m,n\}\geqslant 2 $.

We also have 
\begin{equation*}
	\Phi((n_{1}\oplus n_{2})) = \frac{1}{2}\lambda_{1}(n_{1})+\frac{1}{2}\lambda_{2}(n_{2}), \qquad (n_{1}\oplus n_{2})\in N_{1}\oplus N_{2},
\end{equation*}
and in particular when we restrict to $A_{1}\oplus A_{2}$,
\begin{equation} \label{eqn: Phi(a_1 oplus a_2)}
	\Phi((a_{1}\oplus a_{2})) = \frac{1}{2}\lambda_{1}(a_{1})+\frac{1}{2}\lambda_{2}(a_{2}) \in A\,, \qquad (a_{1}\oplus a_{2})\in A_{1}\oplus A_{2}.
\end{equation}
Therefore, $\Phi$ is a normal u.c.p.\ map that maps $A_{1}\oplus A_{2}$ into $A$, and hence by \cite[Lemma 2.1]{MR4833744}, $\Phi$ extends to a normal u.c.p.\ map 
\[\Phi^{\sharp *}\colon (C^*(N_{1,0}\oplus N_{2,0},T)^{(A_{1}\oplus A_{2})\sharp (A_{1}\oplus A_{2})})^{*} \to ((N_{1,0} {*}_B N_{2,0})^{A\sharp A})^{*}\]
satisfying
\[i_{N_{1,0} {*}_B N_{2,0}} \circ \Phi = \Phi^{\sharp *} \circ i_{C^*(N_{1,0}\oplus N_{2,0},T)},\]
where $i_{C^*(N_{1,0}\oplus N_{2,0},T)}\colon C^*(N_{1,0}\oplus N_{2,0},T)\to  (C^*(N_{1,0}\oplus N_{2,0},T))^{(A_{1}\oplus A_{2})\sharp (A_{1}\oplus A_{2})})^{*}$ and $i_{N_{1,0} {*}_B N_{2,0}}\colon N_{1,0} {*}_B N_{2,0}\to  ((N_{1,0} {*}_B N_{2,0})^{A\sharp A})^{*}$ are the canonical inclusion maps. Hence, if we let $p_{\nor} \in C^*(N_{1,0}\oplus N_{2,0},T)^{**}$ denote the projection corresponding to the identity representation $A_{1}\oplus A_{2}\subset M_{1}\oplus M_{2}$ and $q_{\nor} \in (N_{1,0} {*}_B N_{2,0})^{**}$ denote the projection corresponding to the identity representation $A\subset M$, then $\Phi^{\sharp *}(p_{\nor}) = \Phi^{\sharp *} \circ i_{C^*(N_{1,0}\oplus N_{2,0},T)}(1) = i_{N_{1,0} {*}_B N_{2,0}} \circ \Phi(1) = q_{\nor}$.

\iffalse
Since the partial isometry $u$ interchanges $\mathcal{K}_{1,1}$ with $\mathcal{K}_{2,1}$, and $\mathcal{K}_{1,2}$ with $\mathcal{K}_{2,2}$ via canonical identifications, by (\ref{graph: u action}) we see that for any $\xi \in L^{2}M$, 
\[\Phi(u)(\xi) = \frac{1}{4}\sum_{i,j,=1}^{2} V_{i,j}\bigg(\sum_{k,l=1}^{2}uV_{k,l}^{*}\xi\bigg) = \frac{1}{4}\sum_{i,j=1}^{2} V_{i,j}\sum_{k,l=1}^{2}V_{k,l}^{*}\xi = \xi.\]

Since $\Phi$ maps $p$ to the identity $1\in \mathbb{B}(L^{2}M)$, we have 
\begin{equation} \label{eqn: Phi(u)=1}
\Phi(T+T^{*}) = \Phi(p(T+T^{*})p) = \Phi(u) = 1.
\end{equation} 
It also follows from (\ref{graph: u action}) that $\Phi(u(m_{1}\oplus m_{2})u) = \frac{1}{2}\lambda_{1}(m_{1})+\frac{1}{2}\lambda_{2}(m_{2})$ for $(m_{1}\oplus m_{2}) \in M_{1}\oplus M_{2}$.

\fi

Let $u = p(T+ T^{*})p \in C^{*}(A_1\oplus  A_2,T)$, where again $p = 1-T^2(T^*)^2$. For $i=1,2$, we define u.c.p.\ maps $\psi_{i}\colon A_{i}\to pC^{*}(A_{1}\oplus A_{2},T)p$ by
\[\psi_{i}(a) = pap+ uau, \quad a\in A_{i}.\]
It is straightforward to see that $\psi_{1}(b) = \psi_{2}(b)$ for $b\in D$. By \cite[Theorem 4.8.2]{MR2391387}, there exists a u.c.p.\ map $\Psi\colon A \to pC^{*}(A_{1}\oplus A_{2},T)p\subset \mathcal{T}(\mathcal{H})$ such that $\Psi(b) = \psi_{1}(b) = \psi_{2}(b)$ for $b\in D$ and
\[\Psi(a_{1}\cdots a_{n}) = \psi_{i_{1}}(a_{1})\cdots \psi_{i_{n}}(a_{n})= \Psi(a_{1})\cdots \Psi(a_{n}), \quad a_{k}\in A^{o}_{i_{k}}\textit{ with }i_{1}\neq \cdots \neq i_{k}.\]

\iffalse
Moreover, for every $a\in A_{i}$, 
\[\Phi\circ \Psi(a) = \Phi(pap+uau) = \Phi(pap) + \Phi(uau) = \frac{1}{2}\lambda_{i}(a) + \frac{1}{2}\lambda_{i}(a) = \lambda_{i}(a).\]

\fi
As shown in \cite{MR4833744}, the u.c.p.\ map $\Phi\circ \Psi$ is the canonical inclusion map when restricted to $A_{1}\cup A_{2}$, and since $A_{1}$ and $A_{2}$ generate $A = A_{1}*_{D}A_{2}$ as a C$^{*}$-algebra, $\Phi\circ \Psi\colon A\to N_1 \overline{*}_B N_2$ is in fact the canonical inclusion map.

Recall that the canonical inclusion $A_{1}\oplus A_{2}\subset N_{1,0}\oplus N_{2,0}$ is $(A_{1}\oplus A_{2}\subset M_{1}\oplus M_{2})$-nuclear by assumption. This allows us to apply the following Lemma in \cite{MR4833744}. %Note that we have changed the symbols in the lemma to avoid confusion with previously used notations.

\begin{lemma}\label{Lem: Lemma 2.4 of Toy25}\cite[Lemma 2.3, 2.4]{MR4833744}
    Suppose $M$ is a von Neumann algebra and $M_{0}$ is a unital, ultraweakly dense C$^*$-subalgebra of $M$. Let $\HH$ be an $(M_0\subset M)$-normal C$^*$-correspondence over $M_0$ and let $B$ be a unital $(M_{0}\subset M)$-normal $M_{0}$-C$^*$-algebra.

    If $\phi: \mathcal{T}(\HH)\to B$ is a $*$-homomorphism whose restriction to $M_{0}$, $\restr{\phi}{M_{0}}:M_{0}\to B$ is the canonical inclusion and is $(M_{0}\subset M)$-nuclear, then there exists  a weakly nuclear $*$-homomorphism
    $$ \pi: \mathcal{T}(\HH)\to (B^{M_{0}\sharp M_{0} })^* $$
    such that $ \pi(x)= (i_B\circ \phi)(x) = p_{\nor}\phi(x)p_{\nor} $ for all $x\in M_{0}\cup \tau(\HH)\cup \tau(\HH)^*$.
    Moreover, we have $$ \pi(\tau(\xi)) =  p_{\nor}\phi( \tau(\xi) )p_{\nor} = \phi( \tau(\xi) )p_{\nor},\quad \forall \xi\in \HH. $$
    Here $p_{\nor} \in B^{**}$ denotes the projection corresponding to the support of the identity representation $M_0\to M$.
\end{lemma}

Applying this lemma to $M_{0} = A_{1}\oplus A_{2}$, $B = C^*(N_{1,0}\oplus N_{2,0},T)$, and $\phi$ being the canonical inclusion $\mathcal{T}(\mathcal{H}) = C^*(A_1\oplus A_2,T) \subset C^*(N_{1,0}\oplus N_{2,0},T)$, we obtain

\begin{lemma}
    Suppose that the inclusion $A_{i}\subset N_{i,0}$ is $(A_{i}\subset M_{i})$-nuclear. There exists a weakly nuclear $*$-homomorphism $$\pi\colon \mathcal{T}(\mathcal{H})\to (C^*(N_{1,0}\oplus N_{2,0},T)^{(A_{1}\oplus A_{2})\sharp (A_{1}\oplus A_{2})})^{*}$$ such that $\pi(x) = i_{C^*(N_{1,0}\oplus N_{2,0},T)}(x)$ for all $x\in (A_1\oplus A_2)\cup \tau(\mathcal{H})\cup \tau(\mathcal{H})^{*}$.
\end{lemma}

Combining all the maps derived above, we obtain the following diagram, which is commutative only in solid arrows:

\begin{equation*}
\begin{tikzcd}[column sep=tiny]
& &(C^{*}(N_{1,0}\oplus N_{2,0},T)^{(A_{1}\oplus A_{2})\sharp (A_{1}\oplus A_{2})})^{*} \arrow[rr, "\Phi^{\sharp *}"]  & & ((N_{1,0}{*}_{B}N_{2,0})^{A\sharp A})^{*} \\
A \arrow[r, "\Psi"] \arrow[rrrr, bend right=17, "\subset", "\text{canonical}" swap] &\mathcal{T}(\mathcal{H})= C^{*}(A_{1}\oplus A_{2},T)\arrow[rr, hook, "\subset", "\text{canonical}" swap] \arrow[ur, dashed, "\pi"]& & C^{*}(N_{1,0}\oplus N_{2,0},T) \arrow[r, "\Phi"]  \arrow[ul, "i_{C^*(N_{1,0}\oplus N_{2,0},T)}", labels=above right] & N_{1,0}{*}_B N_{2,0} \arrow[u, "i_{N_{1}{*}_{B}N_{2}}", labels=right]
\end{tikzcd}
\end{equation*}

Applying the same argument as in the proof of \cite[Theorem 3.2]{MR4833744}, we can check that although the diagram does not commute, we still have $ \Phi^{\sharp *}\circ \pi \circ \Psi = i_{N_{1,0} \ast_B N_{2,0} }\circ\Phi\circ \Psi = i_{N_{1,0}\ast_B N_{2,0}}\big|_{A} $ by showing that $A$ is in the multiplicative domain of $ \Phi^{\sharp *}\circ \pi \circ \Psi $. As $\pi$ is weakly nuclear and $ \Phi^{\sharp *} $ is normal, the inclusion $ i_{N_{1,0}\ast_B N_{2,0}}\big|_{A}: A\to ((N_{1,0}\ast_B N_{2,0})^{A\sharp A})^* $ is weakly nuclear. This then implies that the inclusion $A=A_1\ast_D A_2 \subset N_{1,0}\ast_B N_{2,0}$ is $ (A\subset M)  $-nuclear by \cite[Lemma 4.4]{ding2023biexact}. To summarize, we obtain the following proposition, which is precisely the desired statement in Theorem \ref{thm: main 3} for $n=2$. 

\begin{prop}\label{thm: main}
    For $i = 1,2$, let $M_i$ be von Neumann algebras with a common von Neumann subalgebra $B\subset M_i$ and faithful normal conditional expectations $ E_{i}:M_i\to B $. For each $i=1,2$, let also $A_i\subset M_i$ be an ultraweakly dense unital C$^*$-algebra such that $ E_{1}(A_1)=E_{2}(A_2)=D $ is a common C$^*$-subalgebra of $A_i$. Suppose also that $ N_{i,0} $ is a C$^*$-subalgebra of $N_i=\langle M_i,e_B \rangle$ such that $C^*(A_i,B,e_B)\subset N_{i,0}\subset N_i$.

    If the inclusion $A_i\subset N_{i,0}$ is $(A_i\subset M_i)$-nuclear for $i=1,2$, then the inclusion $ A=A_1\ast_D A_2 \subset N_{1,0}\ast_{B} N_{2,0}$ is $ (A\subset M) $-nuclear.%\qed
\end{prop}
\begin{proof}
    As explained in the previous paragraph, one only need to check that $ \Phi^{\sharp *}\circ \pi \circ \Psi = i_{N_{1,0} \ast_B N_{2,0}}\circ\Phi\circ \Psi = i_{N_{1,0}\ast_B N_{2,0}}\big|_{A} $.

    %%%%%%%%%%%%%%%%%%copied from weakly exact amalgamated free product paper %%%%%%%%%%%%%%%%%%%%%%%%%%%%
    We first recall that by equation \eqref{eq: Phi of T2T*2}, $\Phi(T^{2}(T^{*})^{2}) = 0$. Applying Lemma \ref{Lem: Lemma 2.4 of Toy25} to the embedding $\mathcal{T}(\mathcal{H}) = C^{*}(A_{1}\oplus A_{2},T) \subset C^*(N_{1,0}\oplus N_{2,0},T)$, we obtain $ \pi(T) = i_{C^*(N_{1,0}\oplus N_{2,0},T)}( T ) = p_{\nor}Tp_{\nor} = Tp_{\nor} $. Therefore
  \begin{equation*}
  	\begin{aligned}
  		0\leqslant \Phi^{\sharp *}(\pi(T^{2}(T^{*})^{2})) &= \Phi^{\sharp *}(\pi(T)^{2}\pi(T^{*})^{2})= \Phi^{\sharp *}((p_{\nor}T p_{\nor})^{2}(p_{\nor}T^{*}p_{\nor})^{2})\\
  		&=\Phi^{\sharp *}(p_{\nor}T^2 p_{\nor}(T^{*})^2p_{\nor}) \leqslant \Phi^{\sharp *}(p_{\nor}T^{2} (T^{*})^{2}p_{\nor})= \Phi^{\sharp *}(i_{C^*(N_{1,0}\oplus N_{2,0},T)}(T^{2}(T^{*})^{2}))\\
  		&= i_{N_{1,0}\ast_B N_{2,0}}(\Phi(T^{2}(T^{*})^{2}))= 0,
  	\end{aligned}
  \end{equation*}
  so $\Phi^{\sharp *}(\pi(T^{2}(T^{*})^{2})) = 0$. Thus we have
  \begin{equation} \label{eqn: Phi(phi(p)) = q_nor}
    \Phi^{\sharp *}(\pi(p)) = \Phi^{\sharp *}(\pi (1-T^2 (T^{*})^{2})) = \Phi^{\sharp *}(\pi(1)) -  \Phi^{\sharp *}(\pi(T^{2}(T^{*})^{2})) = \Phi^{\sharp *}(\pi(1)) = q_{\nor}.
  \end{equation}
	In particular, $\pi(p)$ is in the multiplicative domain of $\Phi^{\sharp *}$. Also, since $u= p(T+T^*)p$ and $\Phi(T)= \Phi(T^*)=\frac{1}{2}$,
  \begin{equation} \label{eqn: Phi(pi(u)) = q_nor}
  	\begin{aligned}
    	\Phi^{\sharp *}(\pi (u)) &= \Phi^{\sharp *}(\pi(p(T+T^{*})p)) = \Phi^{\sharp *}(\pi(p)\pi(T+T^{*})\pi(p)) = \Phi^{\sharp *}(\pi(T+T^{*})) \\
    	&= \Phi^{\sharp *}(i_{C^*(N_{1,0}\oplus N_{2,0},T)}(T+T^{*})) = i_{N_{1,0}\ast_B N_{2,0}}\circ \Phi(T+ T^{*}) = q_{\nor},
  	\end{aligned}
  \end{equation}
  so $\pi(u)$ also lies in the multiplicative domain of $\Phi^{\sharp *}$.

 Now, for any $a\in A_{i}$, by (\ref{eqn: Phi(phi(p)) = q_nor}), (\ref{eqn: Phi(pi(u)) = q_nor}), and (\ref{eqn: Phi(a_1 oplus a_2)}),
  \begin{equation*} 
  	\begin{aligned}
    	(\Phi^{\sharp *}\circ \pi\circ \Psi)(a) &= \Phi^{\sharp *}(\pi(pap+uau)) 
    	= 2 \Phi^{\sharp *}(\pi(a))
    	=  2 \Phi^{\sharp *}(i_{C^*(N_{1,0}\oplus N_{2,0},T)}(a))
    	= 2 i_{N_{1,0}\ast_B N_{2,0}}(\Phi(a))\\
    	&= 2 i_{N_{1,0}\ast_B N_{2,0}}\bigl(\frac{1}{2}\lambda_{i}(a)\bigr)
    	= i_{N_{1,0}\ast_B N_{2,0}}(\lambda_{i}(a))
    	= (i_{N_{1,0}\ast_B N_{2,0}} \circ \Phi \circ \Psi)(a).
    \end{aligned}
  \end{equation*}
  In particular, $a$ is in the multiplicative domain of $\Phi^{\sharp *}\circ \pi\circ \Psi$. Since $A_{1}$ and $A_{2}$ generate $A = A_{1}*_{D}A_{2}$ as C$^{*}$-algebra, we conclude that equation holds on $A$.
\end{proof}

We can now prove Theorem \ref{thm: main 3} by induction.
\begin{proof}[Proof of Theorem \ref{thm: main 3}]
    We already showed the case when $n=2$. Now, assume that the inclusion \[ A_1\ast_D\cdots \ast_D A_{n-1} \subset N_{1,0}\ast_{B}\cdots \ast_B N_{n-1,0}\] is $ (A_1\ast_D\cdots \ast_D A_{n-1}\subset M_1\overline{*}_B\cdots \overline{*}_B M_{n-1}) $-nuclear. We view $M$ as the AFP von Neumann algebra of $ M_1\overline{*}_B\cdots \overline{*}_B M_{n-1} $ and $M_{n}$, so that $ A_1\ast_D\cdots \ast_D A_{n-1} $ is an ultraweakly dense subalgebra of $ M_1\overline{*}_B\cdots \overline{*}_B M_{n-1}$, and $ C^*(A_1{*}_D\cdots {*}_D A_{n-1},e_B )\subset N_{1,0}\ast_{B}\cdots \ast_B N_{n-1,0}\subset \langle  M_1\overline{*}_B\cdots \overline{*}_B M_{n-1},e_B\rangle $. Applying Proposition \ref{thm: main}, it follows that the inclusion
    \[A=A_1\ast_D\cdots \ast_D A_n \subset N_{1,0}\ast_{B}\cdots \ast_B N_{n,0}\] 
    is $ (A\subset M) $-nuclear.
\end{proof}

We have two notable applications of Theorem \ref{thm: main 3}: the case where $M_i$'s are weakly exact, and the case where they are injective.

\begin{corollary}\label{cor: nuclear embedding of free product of M_i}
    If $ M_i $ is weakly exact for each $i=1,\cdots, n$ and $B$ is an injective common von Neumann subalgebra with expectation, then the inclusion
    $$ M_r=M_1\ast_B \cdots \ast_B M_n \subset N_1\ast_B \cdots \ast_B N_n \subset \mathbb{B}(\FF)$$
    is $ (M_r\subset M) $-nuclear. If moreover each $M_i$ is injective, then the inclusion
    $$ M_r=M_1\ast_B \cdots \ast_B M_n \subset C^*(M_1,e_B)\ast_B \cdots \ast_B C^*(M_n,e_B) = C^*(M_i,\lambda_i(e_B):i=1,\cdots,n) \subset \mathbb{B}(\FF)$$
    is $ (M_r\subset M) $-nuclear.
\end{corollary}
\begin{proof}
    If each $M_i$ is weakly exact, then the inclusion $ M_i\subset N_i $ is $M_i$-nuclear by Lemma \ref{lem: weak exactness into smaller injective}. The statement now follows from applying Theorem \ref{thm: main 3} to $ A_i=M_i $ and $N_{i,0}=N_i$.

    If each $M_i$ is injective, then the identity inclusion $ M_i\subset M_i $ is $M_i$-nuclear. In particular, the inclusion $ M_i\subset C^*(M_i,e_B) $ is $ M_i $-nuclear, and the statement follows again by applying Theorem \ref{thm: main 3} to $ A_i=M_i $ and $N_{i,0}=C^*(M_i,e_B)$.
\end{proof}

\subsection{Relative biexactness: weakly exact case}
We will focus on von Neumann subalgebras satisfying the following condition. 
\begin{defn}[Condition (A)]
    For a von Neumann subalgebra $B\subset  M$ with a fixed faithful normal conditional expectation $E:M\to B$, we say $ (B\subset  M,E) $ satisfies \textbf{Condition (A)} if the $\mathbb{C}$-$M$-closure of $ e_BM $ contains $ e_B\langle M,e_B\rangle $:
    $$ e_B\langle M,e_B\rangle\subset  \overline{ e_BM}^{\mathbb{C}-M}.$$
\end{defn}

While we believe the condition above holds for a significantly broader class of inclusions $(B \subset  M, E)$ regardless of whether $M$ is tracial or not, we are currently able to establish it only in the case where $M$ admits a faithful normal tracial state $\tau$ satisfying $\tau \circ E = \tau$, or when $B$ is finite-dimensional.

\begin{lemma}\label{lem: tracial implies condition A}
    If $(M,\tau)$ is a tracial von Neumann algebra with a faithful normal tracial state $\tau$ such that $ \tau\circ E  = \tau $, then $ (B\subset  M,E) $ satisfies Condition (A).
\end{lemma}
\begin{proof}
    For any $T\in \langle M,e_B\rangle$ with $\|T\|\leqslant 1$, we consider $T\tau^{1/2} = T\widehat{1}\in L^2M$ as an unbounded operator $ L_{T\widehat{1}} $ affiliated with $M$. Let $ P_{n} $ be the spectral projection $ P_n\coloneqq E_{[0,n]}(L_{T\widehat{1}}L^*_{T\widehat{1}})$, so then $P_n L_{T\widehat{1}}$ and $L_{T\widehat{1}}^* P_n$ are bounded. Since $T$ is right $B$-linear, we have $Te_B = L_{T\widehat{1}} e_B$, and hence $e_B T^* P_n = e_B L_{T\widehat{1}}^* P_n$. For every normal positive functional $\rho$ on $M$, 
    \[s_\rho\bigl(e_B T^* - e_B L_{T\widehat{1}}^* P_n\bigr) = s_\rho(e_B T^*(1-P_n)) \leqslant \|T\| \rho(1-P_n)^{1/2}\to 0.\]
    This proves Condition (A).
\end{proof}

\begin{remark}
    If $B$ is finite-dimensional, then $ (B\subset M,E) $ satisfies Condition (A), even when $M$ is not tracial. Indeed, fix $T\in\langle M,e_B\rangle$. By Kaplansky's density theorem, choose a bounded net $T_i\in\Span\{ae_Bb:a,b\in M\}$ converging to $T$ in the strong-$*$ topology. Write $T_i=\sum_{j=1}^{k_i} a_{i,j}e_Bb_{i,j} $ and $
    m_i=\sum_{j=1}^{k_i} E(a_{i,j})b_{i,j}\in M$. The identity $e_Bae_B=E(a)e_B$ gives $e_BT_i = e_B m_i$. Since $e_B$ has finite rank and $T_i^*\to T^*$ strongly, we have $ \|e_Bm_i-e_BT\| =\|e_B(T_i-T)\| =\|(T_i^*-T^*)e_B\| \to 0$. Thus $e_BM$ is norm dense in $e_B\langle M,e_B\rangle$, and therefore it is also dense in the $\C$-$M$ topology. This proves Condition (A).
\end{remark}

\iffalse
\begin{comment}
    If $B$ is finite-dimensional, then $ (B\subset M,E) $ satisfies condition (A), even when $M$ is not tracial. Indeed, since $ e_B $ is now a finite dimensional projection, for any $T\in \langle M,e_B\rangle$, we can choose a net $ m_n\in M $ such that $e_Bm_n$ converges strongly to $e_BT$. (Again, by Kaplansky density theorem, we can choose $T_n\in Me_BM$ strongly converging to $T$, and then choose $ m_n \in M$ so that $ e_B m_n= e_B T_n $). Then $e_Bm_n$ converges in norm, and hence also in the $\mathbb{C}$-$M$ topology, to $ e_BT $.
\end{comment}
\fi

\begin{prop}\label{prop: N is in S}
    If $(B\subset  M_i,E) $ satisfies Condition (A) for each $i=1,\cdots,n$, then $ N_1\ast_B \cdots \ast_B N_n \subset  \mathbb{S}_{\{M_1,\cdots,M_n\}}(M)$.
\end{prop}
\begin{proof}
    Since polynomials in $ \bigcup_{i=1}^n \lambda_i({N_i}) $ are norm dense in $ N_1\ast_B \cdots \ast_B N_n $, it suffices to show that for each $ T_j\in N^o_{i_j} $ with $ i_j\neq i_{j+1}$, $1\leqslant j \leqslant m$, 
    \[[\lambda_{i_1}(T_1)\cdots \lambda_{i_m}(T_m), a ]\in \mathbb{K}^{\infty,1}_{\{M_1,\cdots,M_n\}}(M),\quad \forall a\in \bigcup_{i=1}^n J \lambda_{i}(M_i) J \subset  M'. \]

    For this, we begin with a computation similar to \cite[Proposition 6.16]{ding2023biexact} and \cite[Proposition 3.2]{MR2211141}. Let $e_i\in \mathbb{B}(L^2M) $ be the projection onto the subspace $L^2M_i\subset  L^2M$. For $ T\in N_i= \langle M_i,e_B \rangle $, $a\in J \lambda_{j}(M_j) J \subset  M'$, we have
    \[ [\lambda_i(T),a] = \begin{cases}
        0,&i\neq j;\\
        e_i[\lambda_i(T),a]e_i\in e_i \BB(L^2M)e_i,&i = j.
    \end{cases} \]
    Indeed, for $i \neq j$ the left and right free-product actions commute; for $i = j$,$ L^2 M_i$ reduces both actions, and on its orthogonal complement they act at different ends of the reduced tensors and commute.
    
    Notice that $\lambda_i(e_B) \in \BB(L^2M)$ is the projection onto  the direct sum of $L^2 B$ and the reduced tensors whose first index is different from $i$, so $ \lambda_i(e_B)^{\perp}\lambda_j(e_B)^{\perp} =0$ for all $i\neq j$ and $ \sum_{i=1}^n \lambda_i(e_B)^{\perp} = 1-e_B\in \mathbb{B}(\FF) $. Since $  [\lambda_{i_1}(T_1)\cdots \lambda_{i_n}(T_n), a ]$ is the sum of terms of the form 
    \begin{align}
     \begin{split}\label{equation for N1 ast N2 inside S}
         &\lambda_{i_1}(T_1)\cdots \lambda_{i_{k-1}}(T_{k-1})[\lambda_{i_k}(T_k),a]\lambda_{i_{k+1}}(T_{k+1}) \cdots\lambda_{i_m}(T_m)\\
         =&\lambda_{i_1}(T_1)\cdots \lambda_{i_{k-1}}(T_{k-1})e_{i_k}[\lambda_{i_k}(T_k),a]e_{i_k}\lambda_{i_{k+1}}(T_{k+1}) \cdots\lambda_{i_m}(T_m),
     \end{split}
     \end{align}
     it remains to check that such terms indeed belong to $ \mathbb{K}^{\infty,1}_{\{M_1,\cdots,M_n\}}(M)$.

     For $T\in N^o_i$ and $j\neq i$, we have $ \lambda_i(T)e_{j} = \lambda_i(T)\lambda_i(e_B)e_{j} = \lambda_i(Te_B)e_{j}$. Here $e_j \leqslant \lambda_i(e_B)$ when $i\neq j$. Also, $e_BTe_B = 0$ gives $\lambda_i(Te_B) = \lambda_i(e_B^\perp)\lambda_i(Te_B)$. These support relations may be applied successively along an alternating centered word.  Similarly, we have
    \[ \lambda_{i_1}(T_1)\cdots \lambda_{i_{k-1}}(T_{k-1})e_{i_k} = \lambda_{i_1}(T_1e_B)\cdots \lambda_{i_{k-1}}(T_{k-1}e_B)e_{i_k}. \]
    The analogous identity for the factors on the right follows by taking adjoints. Therefore, the term in \eqref{equation for N1 ast N2 inside S} is equal to
    \begin{equation}\label{equation 2 for N1 ast N2 inside S}
        \lambda_{i_1}(T_1e_B)\cdots \lambda_{i_{k-1}}(T_{k-1}e_B)e_{i_k}[\lambda_{i_k}(T_k),a]e_{i_k}\lambda_{i_{k+1}}(e_BT_{k+1}) \cdots\lambda_{i_m}(e_BT_m)
    \end{equation}
    Empty products are interpreted as 1.

    We now show that the term \eqref{equation 2 for N1 ast N2 inside S} belongs to $ \mathbb{K}^{\infty,1}_{\{M_1,\cdots,M_n\}}(M) $ by induction.
    \begin{steps}
        \item If $ T_j\in M^{o}_{i_j} $ for all $ j\neq k $, then the term in \eqref{equation 2 for N1 ast N2 inside S} is already inside $ \mathbb{K}^{\infty,1}_{\{M_1,M_2\}}(M) $ because it equals to the term in \eqref{equation for N1 ast N2 inside S} which belongs to $ Me_{i_k}\BB(L^2M)e_{i_k}  M\subset  \mathbb{X}_{M_1,\cdots,M_n}\subset  \mathbb{K}^{\infty,1}_{\{M_1,\cdots,M_n\}}(M) $.
        \item If $k>1$ and $ T_j \in M^{o}_{i_j} $ for all $ j\neq k-1,k $, by the adjoint of Condition~(A), choose a net $ m_t$ in $M_{i_{k-1}}^o $ with $ m_te_B \to T_{k-1}e_B $ in the $M_{i_{k-1}}$-$\mathbb{C}$ topology. By Step 1, we have for any $t$,
        \begin{equation}
            \label{equation 3 for N1 ast N2 inside S}
            \lambda_{i_1}(T_1e_B)\cdots \lambda_{i_{k-2}}(T_{k-2} e_B)\lambda_{i_{k-1}}(m_t e_B)e_{i_k}[\lambda_{i_k}(T_k),a]e_{i_k}\lambda_{i_{k+1}}(e_BT_{k+1}) \cdots\lambda_{i_m}(e_BT_m)\in \mathbb{K}^{\infty,1}_{\{M_1,\cdots,M_n\}}(M). 
        \end{equation} 
       Note that $\lambda_i$ is continuous from the $M_i$-$\mathbb C$
topology to the $\lambda_i(M_i)$-$\mathbb C$ topology
(cf.\ \cite[Lemma~3.1]{ding2023biexact}), indeed, for
$R=b^*Y\in N_i$ with $b\in M_i$ and $Y\in N_i$ and
$\rho\in M_*^+$,
\(
    s^\rho\bigl(\lambda_i(R)\bigr)
    \leq
    \rho(\lambda_i(b^*b))^{1/2}\|Y\|,
\)
and hence
\(
    s^\rho\bigl(\lambda_i(R)\bigr)
    \leq
    s^{\,\rho\circ\lambda_i|_{M_i}}(R).
\) Thus $ m_te_B \to T_{k-1}e_B$ in $M_{i_{k-1}}$-$\mathbb{C}$ topology also implies $ \lambda_{i_{k-1}}(m_te_B) \to  \lambda_{i_{k-1}}(T_{k-1}e_B)$ in the $\lambda_i(M_i)$-$\mathbb{C}$ topology, and hence also in the $M$-$\mathbb{C}$ topology.
        Since by assumption $ \lambda_{i_1}(T_1)\cdots \lambda_{i_{k-2}}(T_{k-2})\in M $, the terms in both \eqref{equation for N1 ast N2 inside S} and \eqref{equation 2 for N1 ast N2 inside S} can be approximated by the terms in \eqref{equation 3 for N1 ast N2 inside S} in $M$-$\mathbb{C}$ topology. This topology is stronger than the $M$-$M$ and $JMJ$-$JMJ$ topology, so the limit is in $\mathbb K^{\infty,1}_{\{M_1,\ldots,M_n\}}(M) $.

        \item Repeating the first two steps, with indices $k-2, k-3,\cdots, 1$ in that order, we deduce that the term in \eqref{equation 2 for N1 ast N2 inside S} belongs to $\mathbb{K}^{\infty,1}_{\{M_1,\cdots,M_n\}}(M)$ whenever $ T_j\in M_{i_j}^o $ for all $ j>k $.
        \item Consecutively approximating $e_BT_{j}$ with $j=k+1,k+2,\ldots,m$ by elements of $e_BM_{i_j}^o$ under the $\mathbb{C}$-$M_{i_j}$ topology and repeating all the above steps, we finally conclude that the term in \eqref{equation 2 for N1 ast N2 inside S} belongs to $\mathbb{K}^{\infty,1}_{\{M_1,M_2\}}(M)$ for all $T_j$.
    \end{steps}

\end{proof}

Therefore, we obtain the relative biexactness of $M$ for $(B \subset M_i,E)$ satisfying Condition (A). In particular, this covers Corollary \ref{cor: relative biexactness tracial}.
\begin{prop} \label{prop: amalg free prod rel biex}
    Assume that $ M_i $ is weakly exact for each $i$, and $B \subset M_i$ is a common injective von Neumann subalgebra with faithful normal conditional expectations $E_i\colon M_i\to B$.
    If $(B\subset  M_i,E_i) $ satisfies Condition (A) for every $i=1,\cdots,n$, then $ M=M_1\overline{\ast}_B\cdots \overline{*}_B M_n $ is biexact relative to $ \{M_1,\cdots,M_n\} $.

    In particular, $M=M_1\overline{\ast}_B\cdots \overline{*}_B M_n$ is biexact relative to $\{M_1,\cdots,M_n\}$ if one of the following conditions holds:
    \begin{itemize}
        \item For each $i$, $M_i$ has a faithful tracial state $\tau_i$ such that $\tau_i\circ E_i = \tau_i$.
        \item $B$ is finite-dimensional.
    \end{itemize}
\end{prop}
\begin{proof}
    By Corollary \ref{cor: nuclear embedding of free product of M_i} and Proposition \ref{prop: N is in S}, we have that the inclusion
    $$ M_r=M_1\ast_B \cdots \ast_B M_n \subset \mathbb{S}_{\{M_1,\cdots,M_n\}}(M) $$
    is $ ( M_r\subset M  ) $-nuclear. By \cite[Corollary 4.9]{ding2023biexact}, this implies that the inclusion $ M \subset \mathbb{S}_{\{M_1,\cdots,M_n\}}(M) $ is $M$-nuclear as $ M_r $ is ultraweakly dense in $ M $.
\end{proof}

\subsection{Relative biexactness: injective case}

\begin{prop}\label{prop: inside S_B}
    For $ N_{i,0}=C^*(M_i,e_B) $, we have
    $N_{1,0}\ast_{B} \cdots\ast_{B} N_{n,0} = C^*( M_i,\lambda_i(e_B):i=1,\cdots ,n )\subset  \mathbb{S}_{B}(M) $.
\end{prop}
\begin{proof}
    Recall that $\lambda_i(e_B)$ is the projection onto the space consisting of $L^2B$ and the reduced tensors beginning with
    $L^2M_j^\circ$ for $j\neq i$. In particular, $ \lambda_i(e_B)^\perp\lambda_j(e_B)^\perp=0$ for $i\neq j$, $\sum_{i=1}^n\lambda_i(e_B)^\perp=1-e_B$, and $e_B\in C^*(M_i,\lambda_i(e_B):i=1,\ldots,n)$.
    We need to show that a word in $\{\lambda_i(M_i^\circ),B,\lambda_i(e_B)^\perp,e_B\}$ always belongs to $\S_B(M)$.
    
    For $x\in M_i^\circ$, let $\lambda_i^+(x)\coloneqq \lambda_i(x)\lambda_i(e_B)$, $\lambda_i^0(x) \coloneqq \lambda_i(e_B)^\perp\lambda_i(x)\lambda_i(e_B)^\perp$, and $\lambda_i^-(x)\coloneqq \lambda_i(e_B)\lambda_i(x)$. We have the decomposition $\lambda_i(x) = \lambda_i^+(x) + \lambda_i^0(x)+\lambda_i^-(x)$ because $\lambda_i(e_B)^\perp\lambda_i(x)\lambda_i(e_B) =\lambda_i(x)\lambda_i(e_B) $ and $\lambda_i(e_B)\lambda_i(x)\lambda_i(e_B)^\perp =\lambda_i(e_B)\lambda_i(x)$.We use the same notation $\lambda_i^0(z) \coloneqq \lambda_i(e_B)^\perp\lambda_i(z)\lambda_i(e_B)^\perp$ for $z\in M_i$ without requiring $z$ to be centered. In particular, $\lambda_i^0(1)=\lambda_i(e_B)^\perp$. Note also that $\lambda_i^+(x)=(\lambda_i^-(x^*))^*$. Since the right support of $\lambda_i^+(x)$ is contained in $e_B+\sum_{k\neq i}\lambda_k(e_B)^\perp$, we can further decompose $\lambda_i^+(x)$, and similarly $\lambda_i^-(x)$, as
    \[\lambda_i^+(x) =\lambda_i^+(x)e_B +\sum_{k\neq i}\lambda_i^+(x)\lambda_k(e_B)^\perp
        =\lambda_i(x)e_B +\sum_{k\neq i}\lambda_i(x)\lambda_k(e_B)^\perp.\]
    Thus we may work with words in $B$, $e_B$, and the operators
    \[ \lambda_i^+(x)\lambda_j(e_B)^\perp,\quad
        \lambda_j(e_B)^\perp\lambda_i^-(x),\quad
        \lambda_i^+(x)e_B,\quad e_B\lambda_i^-(x),\quad
        \lambda_i^0(z),\]
    where $x\in M_i^\circ$, $z\in M_i$, and $i\neq j$ in the terms involving $j$. We regard the first and third operators
    as creation operators, the second and fourth as annihilation operators, and $B$, $e_B$, and $\lambda_i^0(z)$ as preservation operators.
    
    For $x,y\in M_i^\circ$ and $z,t\in M_i$, we have
    \begin{equation*}
        \begin{aligned}
            \lambda_i^-(x)\lambda_i^+(y)
                &=E(xy)\lambda_i(e_B),&
            \lambda_i^0(z)\lambda_i^+(y)
                &=\lambda_i^+(zy-E(zy)),\\
            \lambda_i^-(x)\lambda_i^0(z)
                &=\lambda_i^-(xz-E(xz)),&
            \lambda_i^0(z)\lambda_i^0(t)
                &=\lambda_i^0(zt) -\lambda_i^+(z-E(z))\lambda_i^-(t-E(t)).
        \end{aligned}
    \end{equation*}
    Each corresponding product with distinct component indices is zero. Also,
    \[e_B\lambda_i^+(x)=\lambda_i^-(x)e_B=e_B\lambda_i^0(z)=\lambda_i^0(z)e_B=0,\]
    and
    \[E(xy)\lambda_i(e_B)=E(xy)e_B+\sum_{k\neq i}\lambda_k^0(E(xy)).\]
    Elements of $B$ commute with $e_B$ and $\lambda_i(e_B)$ and can be absorbed into neighboring coefficients.
    
    Replacing each $\lambda_k(e_B)^\perp$ by $\lambda_k^0(1)$ and absorbing the $B$-coefficients, we can use these relations to move creation operators to the left and annihilation operators to the right. Counting $be_B$, $b\in B$, as a single preservation operator, the relations either shorten a word or replace two adjacent preservation operators by a creation--annihilation pair. Induction on the length and then on the number of preservation operators therefore gives a linear combination of elements of $B$ and words of the following three forms:
    \[\begin{cases}
        \lambda_{i_1}^+(x_1)\cdots\lambda_{i_n}^+(x_n)
        e_Bbe_B
        \lambda_{j_1}^-(y_1)\cdots\lambda_{j_m}^-(y_m),\\[2pt]
        \lambda_{i_1}^+(x_1)\cdots\lambda_{i_n}^+(x_n)
        \lambda_l^0(z)
        \lambda_{j_1}^-(y_1)\cdots\lambda_{j_m}^-(y_m),\\[2pt]
        \lambda_{i_1}^+(x_1)\cdots\lambda_{i_n}^+(x_n)
        \lambda_{j_1}^-(y_1)\cdots\lambda_{j_m}^-(y_m).
    \end{cases}\]
    Here $b\in B$, $z\in M_l$, $x_k\in M_{i_k}^\circ$, and $y_k\in M_{j_k}^\circ$; the two blocks have arbitrary finite lengths and may be empty. We may assume that consecutive indices within each block are distinct. For the second form, we may also assume $i_n\neq l$ and $l\neq j_1$ whenever the corresponding block is nonempty; otherwise the word is zero. There is no additional requirement that $i_n\neq j_1$. In carrying out the reduction, support projections between successive creation operators, or between successive annihilation operators, may be absorbed. A projection $\lambda_k(e_B)^\perp$ remaining between the two blocks is retained as $\lambda_k^0(1)$, rather than omitted.
    
    Whenever nonzero, these words respectively have the equivalent expressions
    \[\begin{cases}
        \lambda_{i_1}(x_1)\cdots\lambda_{i_n}(x_n)
        e_Bbe_B
        \lambda_{j_1}(y_1)\cdots\lambda_{j_m}(y_m),\\[2pt]
        \lambda_{i_1}(x_1)\cdots\lambda_{i_n}(x_n)
        \lambda_l^0(z)
        \lambda_{j_1}(y_1)\cdots\lambda_{j_m}(y_m),\\[2pt]
        \lambda_{i_1}(x_1)\cdots\lambda_{i_n}(x_n)
        \lambda_{i_n}(e_B)\lambda_{j_1}(e_B)
        \lambda_{j_1}(y_1)\cdots\lambda_{j_m}(y_m).
    \end{cases}\]
    In the third expression, $\lambda_{i_n}(e_B)$ or $\lambda_{j_1}(e_B)$ is omitted when its corresponding block is empty.
    
    We first show that, for every word $w$ above, $[w,a']\in\X_B$ for every $a'\in\bigcup_i JM_iJ$. Recall that $\X_B$ contains $e_B$ and is both an $M$-bimodule and an $M'$-bimodule. For words of the first form, we have for $a'\in M'$
    \[[\lambda_{i_1}(x_1)\cdots\lambda_{i_n}(x_n)  e_Bbe_B  \lambda_{j_1}(y_1)\cdots\lambda_{j_m}(y_m),a']\in M[e_B,a']M\subseteq\X_B.\]
    Similarly, for words of the third form with both blocks nonempty, we have
    \[[\lambda_{i_1}(x_1)\cdots\lambda_{i_n}(x_n)\lambda_{i_n}(e_B)\lambda_{j_1}(e_B) \lambda_{j_1}(y_1)\cdots\lambda_{j_m}(y_m),a'] \in M[\lambda_{i_n}(e_B)\lambda_{j_1}(e_B),a']M.\]
    We claim that $[\lambda_i(e_B)^\perp,a']=-\delta_{i,j}[e_B,a']$ for $a'\in J\lambda_j(M_j)J$. Indeed, since $\lambda_i(N_i)$ commutes with $J\lambda_j(M_j)J$ for $i\neq j$, the commutator is zero in this case. If $i=j$, then
    \[[\lambda_i(e_B)^\perp,a'] =\left[1-e_B-\sum_{k\neq i}\lambda_k(e_B)^\perp,a'\right]  =-[e_B,a'].\]
    Using the orthogonality of the projections $\lambda_k(e_B)^\perp$, we obtain, whether or not $i_n=j_1$,
    \[[\lambda_{i_n}(e_B)\lambda_{j_1}(e_B),a'] = \bigl(\delta_{i_n,j}+\delta_{j_1,j}  -\delta_{i_n,j_1}\delta_{i_n,j}\bigr)[e_B,a'] \in\X_B.\]
    If one block is empty, the same conclusion follows from $[\lambda_i(e_B),a']=\delta_{i,j}[e_B,a']$; if both blocks are empty, the word is $1$. Therefore the commutators of words of the third form also belong to $\X_B$.
    
    Finally, for words of the second form, we have
    \[[\lambda_{i_1}(x_1)\cdots\lambda_{i_n}(x_n) \lambda_l^0(z) \lambda_{j_1}(y_1)\cdots\lambda_{j_m}(y_m),a'] \in M[\lambda_l(e_B)^\perp\lambda_l(z) \lambda_l(e_B)^\perp,a']M.\]
    Thus it remains to show that the middle commutator belongs to $\X_B$ for $z\in M_l$ and $a'\in J\lambda_j(M_j)J$. If $l\neq j$, it is zero. If $l=j$, then
    \begin{equation*}
        \begin{aligned}
            &[\lambda_l(e_B)^\perp\lambda_l(z)
                \lambda_l(e_B)^\perp,a']\\
            &\quad=[\lambda_l(z)-\lambda_l(e_B)\lambda_l(z)
                -\lambda_l(z)\lambda_l(e_B)+E(z)\lambda_l(e_B),a']\\
            &\quad=-[\lambda_l(e_B),a']\lambda_l(z)
                -\lambda_l(z)[\lambda_l(e_B),a']
                +E(z)[\lambda_l(e_B),a']
                \in\X_B,
        \end{aligned}
    \end{equation*}
    where the last inclusion follows from
    \[[\lambda_l(e_B),a']=\left[e_B+\sum_{k\neq l}\lambda_k(e_B)^\perp,a'\right]=[e_B,a']\qquad(a'\in J M_lJ).\]
    Elements of $B$ commute with $M'$, so the desired commutator inclusion holds for every algebraic word $w$ and every $a'\in\bigcup_i JM_iJ$. Finally, norm density of the algebraic words and norm closedness of $\mathbb K_B^{\infty,1}(M)$ give the desired inclusion in $\S_B(M)$.
\end{proof}

\begin{thm}[\protect{Theorem \ref{thm: relative biexactness injective}}]\label{cor:injectiveAFPbiexact}
    If $M_1,\cdots, M_n$ are injective von Neumann algebras admitting a common subalgebra $B$ with expectation, then $ M=M_1\overline{*}_B \cdots \overline{*}_B M_n $ is biexact relative to $B$.
\end{thm}
\begin{proof}
    By Corollary \ref{cor: nuclear embedding of free product of M_i} and Proposition \ref{prop: inside S_B}, we obtain that the inclusion
    $$M_r= M_1\ast_B \cdots \ast_B M_n\xhookrightarrow{} C^*(M_1,e_B)\ast_B\cdots \ast_B C^*(M_n,e_B) \subset \mathbb{S}_B(M)$$
    is $ (M_r\subset M) $-nuclear, thus the inclusion
    $ M\subset \mathbb{S}_B(M) $
    is $M$-nuclear as $M_r$ is ultraweakly dense in $M$.
\end{proof}

The following corollary slightly generalizes \cite[Theorem 4.4]{MR3594283} under the additional assumption that $M_i$'s are injective.

\begin{corollary}%[cf.\ Corollary \ref{cor: dichotomy in  injective}]
    If $M_1, \cdots, M_n$ are $\sigma$-finite injective von Neumann algebras admitting a common unital von Neumann subalgebra $B$ with normal conditional expectations, then for every finite von Neumann subalgebra $N\subset M =  M_1\overline{\ast}_B \cdots \overline{\ast}_B M_n$ with expectation, either $N \preceq_{M} B$, or the relative commutant $N'\cap M$ is amenable. 
    
    In particular, every separable finite subfactor $N\subset M$ with expectation is either McDuff or prime.
\end{corollary}
\begin{proof}
    By the above theorem, $M$ is biexact relative to $B$. The first statement then follows from  Proposition \ref{prop DP23 6 13}. In particular, since $B$ is injective, any finite subalgebra $N$ with no amenable direct summand must have amenable relative commutant $N'\cap M$ as $N \npreceq_{M} B$.

    Now suppose $N \subset M$ is a separable finite subfactor with a conditional expectation $E\colon M\to N$. If $N$ is not prime, then $N = N_{1} \overline{\otimes} N_{2}$ for some diffuse tracial factors $N_{1}, N_{2}$, and there exist normal conditional expectations $E_{i}\colon N \to N_{i}$. If $N_{1}$ is amenable, then $N_{1}$ is isomorphic to the hyperfinite II$_{1}$ factor $\mathcal{R}$ and thus $N\simeq \mathcal{R} \overline{\otimes} N_{2}$ is McDuff. If $N_{1}$ is not amenable, then we know $N_{1}'\cap M$ is amenable by the first paragraph and contains $N_{2}$. The composition $E_{2}\circ \restr{E}{N_{1}'\cap M} \colon N_{1}'\cap M \to N_{2}$ gives a conditional expectation from $N_{1}'\cap M$ onto $N_{2}$, so $N_{2}$ is amenable. Hence $N_{2} \simeq \mathcal{R}$ and again we conclude that $N \simeq N_{1}\overline{\otimes} \mathcal{R}$ is McDuff.
\end{proof}

\section{Upgrading theorem} \label{sec: upgrading theorem}

We recall the notations from Section \ref{sec: biduals}. Given a (not necessarily unital) $*$-subalgebra $A \subset \B(L^2M)$, we let $\overline{q}_A \coloneqq \mathbf{1}_{A^{**}}$ denote the unit of $A^{**} \subset \B(L^2M)^{**}$. In the case when $[\overline{q}_{A},p_{\nor}^{\sharp}]=0$, we define $ q_{A} \coloneqq p^{\sharp}_{\nor}\overline{q}_A$. When $ \X $ is a $M$-boundary piece, since $ \mathbb{K}_{\X}(M) $ and $\X$ has the same $M$-$M$ and $JMJ$-$JMJ$ closure, and both $M $ and $JMJ$ is in the multiplier of $  \mathbb{K}_{\X}(M) $, we have $ p_{\nor}^{\sharp}\bar{q}_{\X}p_{\nor}^{\sharp} = \bar{q}_{  \mathbb{K}_{\X}(M) } p_{\nor}^{\sharp} = q_{ \mathbb{K}_{\X}(M)}$. For simplicity, we shall still denote $ q_{\X} = q_{ \mathbb{K}_{\X}(M)} $.

\subsection{Identities for bidual of nonunital $*$-subalgebras}

Suppose $N\subset M$ is a von Neumann subalgebra with expectation. We recall from \cite[Lemma 3.8]{DS24structure} that $[p_{\nor}^{\sharp},\iota(e_N)]=0$. In particular, $ C^*( M,JMJ,e_N ) $ always lies inside the multiplicative domain of $\iota^{\sharp}$.

Using the basic construction, we can build several non-unital C$^*$-subalgebras of $ \B(L^2M) $ commuting with $p^{\sharp}_{\nor}$ in the bidual: 
\[ \overline{Me_NM},\quad \overline{JMJe_NJMJ},\quad \overline{MJMJ C^*(e_{N}MJMJ e_N)MJMJ} \] 
(here we are taking the norm closure) and similarly their respective hereditary version 
\[\overline{Me_N \B(L^2M)e_NM},\quad \overline{JMJe_N \B(L^2M)e_NJMJ},\quad  \mathbb{X}_{N}\coloneqq \overline{MJMJe_N \B(L^2M)e_NMJMJ}.\] 
It is then natural to identify the identities of the biduals of these subalgebras, respectively.

\begin{lemma}\label{lem:identityofbidual}
    We have
    \[\overline{q}_{{Me_NM} }= \overline{q}_{ {Me_N \B(L^2M)e_NM} } = \vee_{u\in \mathcal{U}(M)} \iota(ue_Nu^*),\quad \overline{q}_{ {JMJe_NJMJ} }= \overline{q}_{ {JMJe_N \B(L^2M)e_NJMJ} } = \vee_{u\in \mathcal{U}(M)} \iota(JuJe_NJu^*J), \]
    and similarly,
    \[ \overline{q}_{\mathbb{X}_{N}}= \vee_{ u,v\in \mathcal{U}(M) } \iota( uJvJe_NJv^*Ju^* ). \]
\end{lemma}
\begin{proof}
    We only prove $ \overline{q}_{\mathbb{X}_{N}}= \vee_{ u,v\in \mathcal{U}(M) } \iota( uJvJe_NJv^*Ju^* ) $ as the proofs for other equalities are similar. For this, we note that the linear span of the elements of the form $ x=uJvJe_N Te_N\tilde{u}J\tilde{v}J$ with $u,\tilde{u}, v,\tilde{v}\in \mathcal{U}(M)$ is norm dense in $\mathbb{X}_{N}$. However, as $ ( uJvJe_NJv^*Ju^* )x =uJvJe_N Te_N\tilde{u}J\tilde{v}J=x  $, the left support of $\iota(x)$ is bounded by $ \iota(uJvJe_NJv^*Ju^*)\leqslant \vee_{ u,v\in \mathcal{U}(M) } \iota( uJvJe_NJv^*Ju^* )$, and the same holds for its right support. Since $ \iota( \mathbb{X}_{N} ) $ is weak$^*$-dense in $ \mathbb{X}_{N}^{**} $ and $ \vee_{ u,v\in \mathcal{U}(M) } \iota( uJvJe_NJv^*Ju^* ) \in \mathbb{X}_{N}^{**}$, we conclude that $  \vee_{ u,v\in \mathcal{U}(M) } \iota( uJvJe_NJv^*Ju^* ) $ is the identity of $ \mathbb{X}_{N}^{**} $.
\end{proof}

%\begin{remark}
    The same proof can be used to show that for any $M$-boundary piece $\mathbb{X}$ generated by a given projection $e\in \B(L^2M)$, it is always $ \overline{q}_{\mathbb{X}}= \vee_{u,v\in \mathcal{U}(M)} \iota\left( uJvJeJv^*Ju^* \right)  $.
    %$ \overline{q}_{\mathbb{X}}= \vee_{ \substack{u\in \mathcal{U}(M)\\ v\in \mathcal{U}(JMJ)} } \iota\left( uvev^*u^* \right)  $.
%\end{remark}

\begin{remark}\label{rmk: 5.2}
    Since all of $ \iota(M),\iota(JMJ),\iota(e_N)$ commute with $ p^{\sharp}_{\nor} $, by applying $ \Ad(p^{\sharp}_{\nor}) $ to $ \overline{q}_{Me_NM} $, $ \overline{q}_{JMJe_NJMJ} $, and $ \overline{q}_{\mathbb{X}_N} $, we also obtain
    \[ {q}_{{Me_NM} }= {q}_{ {Me_N \B(L^2M)e_NM} } = \vee_{u\in \mathcal{U}(M)} \iota^{\sharp}(ue_Nu^*),\hspace{0.6em} {q}_{ {JMJe_NJMJ} }= {q}_{ {JMJe_N \B(L^2M)e_NJMJ} } = \vee_{u\in \mathcal{U}(M)} \iota^{\sharp}(JuJe_NJu^*J), \]
    and similarly,
    \[ {q}_{\mathbb{X}_{N}} = \vee_{ u,v\in \mathcal{U}(M) } \iota^{\sharp}( uJvJe_NJv^*Ju^* ). \]
    We also note that $ \overline{q}_{Me_NM} $ is the smallest projection in $(M^{**})' $ that majorizes $ \iota(e_N) $, and ${q}_{Me_NM}$ is the smallest projection in $\iota^{\sharp}(M)'\cap \B(L^2M)^{\sharp *}_J$ that majorizes $\iota^{\sharp}(e_N)$. If we have a nondegenerate normal representation $\B(L^2M)^{**}\subset \B(\K)$, then $ \overline{q}_{Me_NM}\K = \overline{M^{**}\iota(e_N)\K} $ and $  {q}_{Me_NM}\K = \overline{q}_{Me_N M}p^{\sharp}_{\nor}\K= \overline{p^{\sharp}_{\nor}M^{**}\iota(e_N)\K} $. We can similarly describe $ \overline{q}_{\mathbb{X}_N} $ as the smallest projection in $ (M^{**}\vee JMJ^{**})' $ majorizing $\iota(e_N)$, and $ \overline{q}_{\mathbb{X}_N}\K = \overline{ M^{**}JMJ^{**}\iota(e_N)\K }$.
\end{remark}

\begin{lemma} \label{lem: q_X_F is union}
    If $\mathcal{F}=\{N_i\}_{i\in I}$ is a family of von Neumann subalgebras of the von Neumann algebra $M$ with expectation, and $ \mathbb{X}_{\mathcal{F}} $ is the $M$-boundary piece generated by $ \{e_{N_i}\}_{i\in I} $ (equivalently by $ \{\mathbb{X}_{N_i}\}_{i\in I} $), then
    \[\overline{q}_{\mathbb{X}_{\mathcal{F}}}=\vee_{i} \;\overline{q}_{\mathbb{X}_{N_i}}=\vee_{ u,v\in \mathcal{U}(M) } \vee_{i}\iota\left( uJvJ e_{N_i}Jv^*Ju^* \right). \]
    In particular, as $p^{\sharp}_{\nor}\in (M^{**}\vee JMJ^{**}\vee \C \iota(e_{N_i}))'$, we also have
    \[{q}_{\mathbb{X}_{\mathcal{F}}}=\vee_{ u,v\in \mathcal{U}(M) } \vee_{i}\iota^{\sharp}\left( uJvJ e_{N_i}Jv^*Ju^* \right) = \vee_{i\in I} {q}_{\mathbb{X}_{N_i}}. \]
\end{lemma}
\begin{proof}
    The identity $ \overline{q}_{\mathbb{X}_{\mathcal{F}}} =\vee_{i} \overline{q}_{\mathbb{X}_{N_i}} $ in fact holds in general as $\mathbb{X}_{\mathcal{F}}$ is the hereditary $C^*$-algebra generated by $ \{\mathbb{X}_{N_i}\} $. This follows from the fact that hereditary $C^*$-subalgebras of $\B(L^2M)$ bijectively to open projections of $ \B(L^2M)^{**} $, and that increasing strong limit of open projections are open (cf.\ \cite{akemann1969general} \cite{bosa2018open}). We provide a proof here for our specific cases when $ \mathbb{X}_{\mathcal{F}}  $ is generated by $e_{N_{i}}$'s. Notice that the span of $\{ uJvJe_{N_i}Te_{N_j}Jv'Ju': u,v,u',v'\in \mathcal{U}(M),i,j\in I\}$ is norm dense in ${\mathbb{X}_{\mathcal{F}}}$, and hence its image under $\iota$ is strongly dense in $ {\mathbb{X}_{\mathcal{F}}}^{**} $. Each of such elements has both left and right support bounded by $ \vee_{ u,v\in \mathcal{U}(M) } \vee_{i}\iota\left( uJvJ e_{N_i}Jv^*Ju^* \right) $, so $\overline{q}_{\mathbb{X_{\mathcal{F}}}}\leqslant \vee_{ u,v\in \mathcal{U}(M) } \vee_{i}\iota\left( uJvJ e_{N_i}Jv^*Ju^* \right) $. Since $ \iota\left( uJvJ e_{N_i}Jv^*Ju^* \right)\leqslant \overline{q}_{\mathbb{X_{\mathcal{F}}}}$ for each $e_{N_{i}}$, taking the supremum we also have $\vee_{ u,v\in \mathcal{U}(M) } \vee_{i}\iota\left( uJvJ e_{N_i}Jv^*Ju^* \right) \leqslant  \overline{q}_{\mathbb{X_{\mathcal{F}}}}$.
    
    The second claim again follows from the normality of $ \Ad(p^{\sharp}_{\nor}) $.
\end{proof}
\iffalse
\begin{remark}\label{rmk: 5.4}
    We also record the following observation for amalgamated free products.
Let $M=M_1\overline{\ast}_B M_2$, and write
$
q_i=q_{\mathbb{X}_{M_i}}$ and $q_B=q_{\mathbb{X}_B}.
$
Then
$
q_1q_2=q_2q_1=q_B.
$
Indeed, let $\mathcal K$ be a normal representation of
$\mathbb B(L^2M)^{\sharp *}_J$. By Remark \ref{rmk: 5.2},
\[
q_2\mathcal K
=
\overline{\iota^\sharp(MJMJe_{M_2})\mathcal K}.
\]
By the reduced word decomposition of $M=M_1\overline{\ast}_B M_2$,
$
(e_{M_1}-e_B)(M\ominus M_1)J(M\ominus M_1)Je_{M_2}=0,
$
and similarly
$
(e_{M_1}-e_B)J(M\ominus M_1)Je_{M_2}
=
(e_{M_1}-e_B)(M\ominus M_1)e_{M_2}=0.
$
It follows that
\[
\iota^\sharp(e_{M_1})q_2\mathcal K
=
\overline{\iota^\sharp(e_{M_1}MJMJe_{M_2})\mathcal K}
\subset q_B\mathcal K.
\]
Since $q_B\leqslant q_2$, we obtain
$
\iota^\sharp(e_{M_1})q_2
=
q_2\iota^\sharp(e_{M_1})q_2,
$
and taking adjoints shows that
$
[\iota^\sharp(e_{M_1}),q_2]=0.
$
As $q_2$ commutes with $\iota^\sharp(M)$ and
$\iota^\sharp(JMJ)$, we also get
\[
[q_1,q_2]=0.
\]
Moreover, the same computation gives $q_1q_2\leqslant q_B$, while
$q_B\leqslant q_1,q_2$, and hence $q_1q_2=q_B$. The same argument shows that for $M=M_1\overline{\ast}_B\cdots\overline{\ast}_B M_n$, 
$
q_{\mathbb{X}_{M_i}}q_{\mathbb{X}_{M_j}}=q_{\mathbb{X}_B},
$ for each $i,j$.
\end{remark}
\fi

\subsection{Mixing subalgebras}

Let $M$ be a tracial von Neumann algebra. A von Neumann subalgebra $N \subset M$ with a normal trace-preserving conditional expectation $E\colon M\to N$ is said to be \textit{mixing} if, for any bounded sequence $(u_n) \subset N$ converging weakly to $0$, $\left\Vert E_N(x u_n y)\right\Vert_2 \to 0$ for all $x, y \in M \ominus N \coloneqq \ker E$. When $N$ is diffuse, this is equivalent to testing only sequences $u_n$ of unitaries by \cite[Theorem 3.3]{cameron2013mixing}. Equivalently, by the proposition in \cite{MR2730894}, $N\subset M$ is a mixing subalgebra if and only if $ e_NxJyJe_N\in \mathbb{K}(M) = \mathbb{K}^L(M)\cap \mathbb{K}^L(M)^* $ whenever $x,y\in M$ and one of $x,y$ is in $ M\ominus N $. For a general (not necessarily tracial) von Neumann algebra $M$, we also attempt to call a subalgebra $N\subset M$ mixing if $ e_NxJyJe_N\in \mathbb{K}^{\infty,1}(L^2M) $ whenever one of $x,y$ is in $ M\ominus N $. Therefore, we propose the following definition of mixingness with respect to general $M$-boundary pieces.%, which may or may not be the correct notion.

\begin{defn}
    Let $\mathbb{Y}\subset \B(L^2M)$ be an $M$-boundary piece, and $ N\subset M $ be a von Neumann subalgebra with expectation. We say that $N\subset M$ is \textit{mixing relative to $\mathbb{Y}$}, or simply \textit{$\mathbb{Y}$-mixing}, if $ e_NxJyJe_N\in \mathbb{K}_{\mathbb{Y}}^{\infty,1}(M) $ whenever $x,y \in M$ and at least one of $x,y$ belongs to $M\ominus N$. When $\mathbb{Y} = \mathbb{K}(L^2M)$, we also say that $ N\subset M $ is mixing.
\end{defn}

We have following lemma for mixing subalgebras in AFP von Neumann algebras; see the proof of \cite[Corollary C]{vaes2014normalizers} for the case of finite von Neumann algebras.

\begin{lemma}
    If $B$ is mixing in $M_2$, then $ M_1 $ is mixing in $M_1\overline{\ast}_B M_2$.
\end{lemma}
\begin{proof}
    It suffices to show that $ e_{M_1}(M\ominus M_1)J(M\ominus M_1)Je_{M_1}\subset \mathbb{K}^{\infty,1}(M)$ since $\mathbb{K}^{\infty,1}(M)$ is both $M_1$ and $JM_1J$-bimodular. For this, we first check for alternating centered product: Consider $ x_{1}\cdots x_n $, $ y_{1}\cdots y_m $ such that $ x_k \in M_{i_k}^o $, $ y_k\in M_{j_k}^o $, $ i_1\neq \cdots \neq i_n $, and $j_1\neq \cdots \neq j_m$. We need to show $ e_{M_1}x_{1}\cdots x_nJy_{1}\cdots y_mJe_{M_1} \in \mathbb{K}^{\infty,1}(M) $ when at least one of $x_i$ and $y_j$ is not in $ M_{1} $. As $\mathbb{K}^{\infty,1}(M)$ is both $M_1$ and $JM_1J$-bimodular and $ e_{M_1} $ commutes with $M_1$ and $JM_1J$, we may assume that $i_1=i_n=j_1=j_m = 2$. Now, we observe that 
    \[ e_{M_1}x_{1}\cdots x_nJy_{1}\cdots y_mJ(e_{M_1}-e_B)=0, \]
    and therefore
    \[ e_{M_1}x_{1}\cdots x_nJy_{1}\cdots y_mJe_{M_1} = e_{M_1}x_{1}\cdots x_nJy_{1}\cdots y_mJe_B= e_{B}x_{1}\cdots x_nJy_{1}\cdots y_mJe_B.\]
    One then observe that
    \[ e_{B}x_{1}\cdots x_{n-1}Jy_{1}\cdots y_{m-1}J  (1-e_B)x_nJ y_mJe_B = e_{B}x_{1}\cdots x_{n-1}Jy_{1}\cdots y_{m-1}J  (e_{M_2}-e_B)x_nJ y_mJe_B =0,\]
    and therefore this term becomes
    \[ e_{B}x_{1}\cdots x_nJy_{1}\cdots y_mJe_B =e_{B}x_{1}\cdots x_{n-1}Jy_{1}\cdots y_{m-1}J e_B(x_nJ y_mJ)e_B.\]
    Inductively, the term is either $0$ or $m=n$, and in the latter case
    \[ e_{B}x_{1}\cdots x_nJy_{1}\cdots y_mJe_B = e_B(x_1Jy_1J)e_B(x_2Jy_2J)e_B\cdots e_B(x_nJ y_nJ)e_B.\]
    We claim that $e_B(x_1Jy_1J)e_BT e_B(x_nJ y_nJ)e_B\in \mathbb{K}^{\infty,1}(M)$ for all $T\in \B(L^2M)$, and therefore proving the statement 
    \[ e_{M_1}x_{1}\cdots x_nJy_{1}\cdots y_mJe_{M_1} = e_B(x_1Jy_1J)e_B(x_2Jy_2J)e_B\cdots e_B(x_nJ y_nJ)e_B\in \mathbb{K}^{\infty,1}(M) .\]
    
    To prove the claim, we note that since $B$ is mixing in $M_2$, $e_B(x_1Jy_1J)e_B$ and $ e_B(x_nJ y_nJ)e_B $ belong to $\mathbb{K}^{\infty,1}(M_2)$. Identifying $L^2M_2$ canonically as a subspace of $L^2M$, we can consider these operators as in $ \B(L^2M_2) \simeq e_{M_2} \B(L^2M) e_{M_2} $. Since the $ M $-$M$ and $JMJ$-$JMJ$ topology is weaker than the $ M_2 $-$M_2$ and $JM_2J$-$JM_2J$ topology, we have $e_B(x_1Jy_1J)e_B, e_B(x_nJ y_nJ)e_B\in  \mathbb{K}^{\infty,1}(M) $. Now, for any operator $T\in \mathbb{B}(L^2M)$, $ e_B(x_1Jy_1J)e_BT e_B(x_nJ y_nJ)e_B\in \mathbb{K}^{\infty,1}(M) $ iff $ \iota^{\sharp}(e_B(x_1Jy_1J)e_BT e_B(x_nJ y_nJ)e_B)\in \mathbb{K}(M)_J^{\sharp*}. $
    Since $ p^{\sharp}_{\nor} $ commutes with $ e_B(x_1Jy_1J)e_B $ and $ e_B(x_nJ y_nJ)e_B $, we have
    \begin{align*}
        \iota^{\sharp}(e_B(x_1Jy_1J)e_BT e_B(x_nJ y_nJ)e_B) = \iota^{\sharp}(e_Bx_1Jy_1Je_B) \iota^{\sharp}( T) \iota^{\sharp}(e_Bx_nJ y_nJe_B) \in \mathbb{K}(M)_J^{\sharp*} \iota^{\sharp}( T) \mathbb{K}(M)_J^{\sharp*}\subset \mathbb{K}(M)_J^{\sharp*}
    \end{align*} 
    since $\mathbb{K}(M)_{J}^{\sharp *} = p_{\nor}^{\sharp}q_{\mathbb{K}(M)}\mathbb{B}(L^2 M)p_{\nor}^{\sharp}$ is an ideal of $\mathbb{B}(L^2M)_J^{\sharp *} = p_{\nor}^{\sharp}\mathbb{B}(L^2 M)p_{\nor}^{\sharp}$.

    Finally, to show that $ e_{M_1}(M\ominus M_1)J(M\ominus M_1)Je_{M_1}\subset \mathbb{K}^{\infty,1}(M) $, we note that as $ [\iota^{\sharp}(e_{M_1}),p_{\nor}^\sharp]=0 $,
    $$ \iota^{\sharp}(e_{M_1}(M\ominus M_1)J(M\ominus M_1)Je_{M_1}  ) = e_{M_1}p^{\sharp}_{\nor}(M\ominus M_1)p^{\sharp}_{\nor}J(M\ominus M_1)Jp^{\sharp}_{\nor}e_{M_1}.$$
    Since alternating centered products (conjugated by $p^{\sharp}_{\nor}$) is strongly dense in $ p^{\sharp}_{\nor}(M\ominus M_1)p^{\sharp}_{\nor} $, we obtain $ \iota^{\sharp}(e_{M_1}(M\ominus M_1)J(M\ominus M_1)Je_{M_1})\subset \mathbb{K}(M)^{\sharp *} $.
\end{proof}

\begin{corollary}\label{cor:biexactrelativetoB}
    If $ B $ is mixing in both $M_1$ and $M_2$, then $ B $ is also mixing in $ M_1\overline{\ast}_B M_2 $.
\end{corollary}
\begin{proof}
    We simply note that as $ M_1 $ is mixing in $ M $, 
    \[e_B M^oJM^oJe_B = e_B e_{M_1}M^oJM^oJe_{M_1}Je_B\subset e_B M_1^o JM_1^{o}Je_B + \mathbb{K}^{\infty,1}(M)\subset \mathbb{K}^{\infty,1}(M). \qedhere\]
\end{proof}

Inductively applying the above corollary and lemma, we obtain the following
\begin{corollary} \label{cor: multiple biexact rel B}
    If $B$ is mixing in $ M_1, \cdots, M_n$, then $ B $, $M_1, \cdots, M_n$ are all mixing in $M_1\overline{*}_B\cdots \overline{*}_B M_n$.%\qedhere
\end{corollary}

\subsection{Another copy of basic construction inside $(\mathbb{X}_N)^{\sharp *}_J$ for mixing subalgebras}

For a mixing subalgebra $N\subset M$ with expectation, we will show that, inside $ (\mathbb{X}_N)^{\sharp *}_J\cap (JMJ^{**})' \subset \mathbb{B}(L^2M)^{**}$, the von Neumann subalgebra $ (\iota^{\sharp}(M)q_{\mathbb{X}_N}\vee \mathbb{C}q_{JMJe_NJMJ})\subset (\mathbb{X}_N)^{\sharp *}_J\cap (JMJ^{**})' $ behaves like the basic construction $\langle M,e_N\rangle$ modulo $\mathbb{K}(L^2M)^{\sharp *}_{J}$. 

\begin{lemma} \label{lem: corner of bidual of basic constr}
    If $M$ is a von Neumann algebra and $N\subset M$ is a von Neumann subalgebra with expectation, then  
    \[\overline{q}_{{Me_NM}} \langle M,e_N\rangle^{**}\overline{q}_{{Me_NM}} = (Me_NM)^{**} .\]
\end{lemma}
\begin{proof}
    We need to show that for any $x\in \langle M,e_N\rangle$, $ \overline{q}_{{Me_NM}}\iota(x)\overline{q}_{{Me_NM}}\subset (Me_NM)^{**} $. Note that as $\overline{q}_{{Me_NM}}$ is the identity of $ (Me_NM)^{**} = \iota(Me_NM)''\overline{q}_{{Me_NM}} $, by Kaplansky density theorem, we can take a net $\iota(a_n)$ in the unit ball of $\iota(Me_NM)$ that converges strongly to $ \overline{q}_{{Me_NM}} $. Then $ a_nxa_n\in Me_NM $ and $ \iota(a_nxa_n)\to \overline{q}_{{Me_NM}}\iota(x)\overline{q}_{{Me_NM}} $ strongly, so we obtain $ \overline{q}_{{Me_NM}}\iota(x)\overline{q}_{{Me_NM}}\subset  (Me_NM)^{**} $.
\end{proof}

%Since $ \mathbb{Y} $ is dense in $ \mathbb{K}_{\mathbb{Y}}(M) $ with respect to  the $M$-$M$, $JMJ$-$JMJ$-topology, we have $ p^{\sharp}_{\nor}  \mathbb{K}_{\mathbb{Y}}(M)^{**}p^{\sharp}_{\nor} =\mathbb{K}_{\mathbb{Y}}(M)_J^{\sharp*} = \mathbb{Y}_J^{\sharp*} = p^{\sharp}_{\nor}\mathbb{Y}^{**}p^{\sharp}_{\nor} $. In particular, $ q_{\mathbb{Y}}= q_{\mathbb{K}_{\mathbb{Y}}(M)}= p^{\sharp}_{\nor}\overline{q}_{\mathbb{Y}} = p^{\sharp}_{\nor}\overline{q}_{\mathbb{K}_{\mathbb{Y}}(M)}$ is a central projection in $ M( \mathbb{Y} )^{\sharp *}_{J} $, although $ \overline{q}_{\mathbb{Y}}$ might not be equal to $\overline{q}_{\mathbb{K}_{\mathbb{Y}}(M)} $ in general. Since $ \iota^{\sharp}(M),\iota^{\sharp}(JMJ) \subset M( \mathbb{Y} )^{\sharp *}_{J} $, we have $ [q_{\mathbb{Y}},  \iota^{\sharp}(M)] = [q_{\mathbb{Y}},  \iota^{\sharp}(JMJ)]=0 $.

For an $M$-boundary piece $\mathbb Y$, since $M,JMJ$ lie in the multiplier algebra of $\mathbb K_\Y(M)$, the projection $\bar{q}_{ \mathbb{K}_{\mathbb{Y}}(M)} = 1_{\mathbb{K}_\Y(M)^{**}}$ commutes with $p_{\nor}^\sharp$. In particular, $\bar{q}_{ \mathbb{K}_{\mathbb{Y}}(M)}$ commutes with $\iota^\sharp(M)$ and $\iota^\sharp (JMJ)$.

\begin{lemma} \label{lem: cutdown bidual}
    For any $M$-boundary piece $\mathbb{Y}$,% such that $[\overline{q}_{\mathbb{Y}},p^{\sharp}_{\nor}] = 0$,
    \[ q_{\mathbb{Y}}\B(L^2M)^{\sharp *}_J q_{\mathbb{Y}}  = \mathbb{K}_{\mathbb{Y}}(M)^{\sharp *}_J. \]
    %Namely, the bidual $ \mathbb{Y}^{\sharp *}_J $ of the $M$-boundary piece $\mathbb{Y}$ is in fact a corner of $ \B(L^2M)^{\sharp *}_J $.
\end{lemma}
\begin{proof}
    %Let $(e_i)$ be a bounded approximate identity of $\mathbb{Y}$. Since $\mathbb{Y}$ is a hereditary C$^{*}$-subalgebra of $\mathbb{B}(L^2M)$, we have $e_{i}\mathbb{B}(L^2M)e_{i} \subset \mathbb{Y}$ for every $e_i$. As $\overline{q}_{\mathbb{Y}} = \lim_i\iota(e_{i})$ in the ultrastrong topology, it follows that $\overline{q}_{\mathbb{Y}}\iota(\mathbb{B}(L^2M)\overline{q}_{\mathbb{Y}} = \lim_{i}\iota(e_{i}\mathbb{B}(L^2M)e_i) = \iota(\mathbb{Y})$. Taking adjoint by $p^{\sharp}_{\nor}$ on both sides, we obtain the lemma.
    
    %Since $\iota^{\sharp}(\mathbb{Y})''q_{\mathbb{Y}} = \mathbb{Y}^{\sharp *}_J $, we simply approximate $ q_{\mathbb{Y}} $ by $\mathbb{Y}$ using Kaplansky density theorem.

    %We have 
    %\[ q_{\mathbb{Y}}\B(L^2M)^{\sharp *}_J q_{\mathbb{Y}} = p^{\sharp}_{\nor}\overline{q}_{\mathbb{Y}}\B(L^2M)^{**}\overline{q}_{\mathbb{Y}}p^{\sharp}_{\nor} = p^{\sharp}_{\nor}\mathbb{Y}^{**}p^{\sharp}_{\nor} = \mathbb{Y}^{\sharp *}_J=\mathbb{K}_{\mathbb{Y}}(M)^{\sharp *}_J, \]
    %where $ \mathbb{Y}^{**} = \overline{q}_{\mathbb{Y}}\B(L^2M)^{**}\overline{q}_{\mathbb{Y}} $ follows from the Kaplansky density theorem and that $\mathbb{Y}$ is a hereditary subalgebra of $\B(L^2M)$.

    Since $\mathbb{K}_\Y(M)$ is hereditary, $\mathbb{K}_\Y(M)^{**} = \overline{q}_{\mathbb{K}_\Y(M)}\B(L^2M)\overline{q}_{\mathbb{K}_\Y(M)}$. Compressing by $p_{\nor}$, which commutes with $\overline{q}_{\mathbb{K}_\Y(M)}$, gives the asserted equality.
\end{proof}

We now illustrate the new copy of basic construction.

\begin{lemma}\label{lem:qisbasicconstruction}
    Suppose $N\subset M$ is a von Neumann subalgebra with expectation that is mixing relative to an $M$-boundary $\mathbb{Y}$. Let $ q\coloneqq q_{{JMJe_NJMJ}} = \overline{q}_{{JMJe_NJMJ}}p_{nor}^{\sharp}\in (\B(L^2M)^{\sharp}_J)^* $. We have
    \[ q\iota^{\sharp}(x)qq_{\mathbb{Y}}^\perp = q\iota^{\sharp}(E_N(x))q q_{\mathbb{Y}}^\perp = \iota^{\sharp}(E_N(x))q q_{\mathbb{Y}}^\perp,\quad \forall x\in M.\]
\end{lemma}
%\begin{remark}
%    Since $p^{\sharp}_{\nor}\in (M^{**}\vee JMJ^{**}\vee \C \iota(e_N))'$, we again have $ [\overline{q},p^{\sharp}_{\nor}]\in [ ( JMJe_NJMJ)^{**},p^{\sharp}_{\nor}] =\{0\}$.
%\end{remark}
\begin{proof}
    Suppose $x\in M\ominus N$. It suffices to show that $  q\iota^{\sharp}(x)q\in q_{\mathbb{Y}} (\B(L^2M)^{\sharp}_J)^* q_{\mathbb{Y}}=  \mathbb{K}_{\mathbb{Y}}(M)^{\sharp *}_J$. For this, we recall that $ \overline{q}_{JMJe_NJMJ}$ is the identity of $ ( JMJe_NJMJ)^{**} $, and therefore by Kaplansky density theorem, we can pick a net $(a_n)$ in $ JMJe_NJMJ $ converging strongly to $ \overline{q}_{JMJe_NJMJ}$ in the bidual, so then $ \iota^{\sharp}(a_n) = \iota(a_n)p^{\sharp}_{\nor} $ converges strongly to $q$. Since $x\in M\ominus N$ and $N\subset M$ is $\mathbb{Y}$-mixing, we have $ a_n xa_n\in \mathbb{K}^{\infty,1}_{\mathbb{Y}}(M) $, so $ \iota^{\sharp}(a_nxa_n) \in (\mathbb{K}_{\mathbb{Y}}(M)^{\sharp}_J)^{*}$. Finally, as $ \iota^{\sharp}(a_nxa_n) = \iota^{\sharp}(a_n)\iota^{\sharp}(x)\iota^{\sharp}(a_n)$ converges strongly to $ q\iota^{\sharp}(x)q $, we conclude that $ q\iota^{\sharp}(x)q\in \mathbb{Y}_J^{\sharp *} $ for all $x\in M\ominus N$.
\end{proof}

\begin{lemma} \label{lem: map identify basic constructions}
    Suppose $N\subset M$ is $\mathbb{Y}$-mixing and $ [q_{\mathbb{Y}},\iota^{\sharp}(e_N)]=0 $. Let $ q\coloneqq q_{{JMJe_NJMJ}} = \overline{q}_{{JMJe_NJMJ}}p_{nor}\in (\B(L^2M)^{\sharp}_J)^* $. The linear map
    \begin{equation*}
    \begin{aligned}
        \phi_0\colon Me_NM&\to \iota^{\sharp}(M)q\iota^{\sharp}(M){q}_{\mathbb{Y}}^{\perp} \subset  ( \iota^{\sharp}(M)q_{\mathbb{X}_N} \vee \C q){q}_{\mathbb{Y}}^{\perp} \subset (\mathbb{X}_N)_{J}^{\sharp *}{q}_{\mathbb{Y}}^{\perp}\\
        xe_Ny&\mapsto \iota^{\sharp}(x)q\iota^{\sharp}(y){q}_{\mathbb{Y}}^{\perp}
    \end{aligned}
    \end{equation*}
    is an $M$-bimodular norm continuous $*$-homomorphism. In particular, $\phi_0$ has a normal extension
    \[\phi\coloneqq \left(\phi_0^*\big|_{ ( \iota^{\sharp}(M)q_{\mathbb{X}_N} \vee \C q)_*{q}_{\mathbb{Y}}^{\perp} }\right)^* =\left(\phi_0^*\big|_{ (\mathbb{X}_N)^{\sharp}_J{q}_{\mathbb{Y}}^{\perp} }\right)^*\colon (Me_N M)^{**}\to ( \iota^{\sharp}(M)q_{\mathbb{X}_N} \vee \C q){q}_{\mathbb{Y}}^{\perp}\subset (\mathbb{X}_N)_{J}^{\sharp*}{q}_{\mathbb{Y}}^{\perp}.\]
\end{lemma}
\begin{proof}
    For an element $ \sum_{k=1}^{n} x_ke_Ny_k \in Me_NM $, we let $ X\coloneqq ( x_1e_N,\cdots,x_ne_N ) $ and $Y\coloneqq (y_1^*e_N,\cdots, y_n^*e_N)^*$. By considering the polar decomposition of $X$ and $Y$, we have the identity of norms
    \[ \|  \sum_{k=1}^n x_ke_Ny_k \| = \| (X^*X)^{1/2}(YY^{*})^{1/2} \|_{\M_n(N)} = \| [E_N(x_i^*x_j)]_{i,j}^{1/2}[ E_N(y_iy^*_j)]_{i,j}^{1/2} \|_{\M_n(N)}.\]
    Similarly, by Lemma \ref{lem:qisbasicconstruction}, we have
    \[ \|  \phi_0(\sum_{k=1}^n x_ke_Ny_k) \| = \|(\id_{\mathbb{M}_n}\otimes qq_{\mathbb{Y}}^{\perp})\cdot(\id\otimes \iota^{\sharp}) ((X^*X)^{1/2}(YY^{*})^{1/2}) \|_{\M_n( qq_{\mathbb{Y}}^{\perp}\iota_{\nor(N)} )} \leqslant \|\sum_{k=1}^n x_ne_Ny_n \|. \qedhere\]
\end{proof}
%\begin{remark}
    Note that without the assumption $ [q_{\mathbb{Y}},\iota^{\sharp}(e_N)] = 0$ in above lemma, $\phi_0$ is still a contraction if we set $ \phi_0(xe_Ny)=  \iota^{\sharp}(x)q{q}_{\mathbb{Y}}^{\perp}q\iota^{\sharp}(y)$, but it may no longer be a $*$-homomorphism.
%\end{remark}

\subsection{Upgrading proper proximality}

\begin{thm}\label{thm:upgrading}
    For a von Neumann algebra $M$ and a $ \mathbb{Y}$-mixing subalgebra $N\subset M$ with expectation, if $ [q_{\mathbb{Y}},\iota^{\sharp}(e_N)] = 0$, then there is a normal $M^{**}$-bimodular u.c.p.\ map
    $$ \Phi: \langle M,e_N\rangle^{**}\to ( \iota^{\sharp}(M)q_{\mathbb{X}_N} \vee \C q)q^{\perp}_{\mathbb{Y}} $$
    such that $ \Phi(x)= xq_{\mathbb{X}_N}q_{\mathbb{Y}}^\perp$ for all $x\in M^{**}$, where $ q= q_{JMJe_NJMJ} $.
\end{thm}
\begin{proof}
    We can simply take $ \Phi = \phi\circ\Ad(\overline{q}_{{Me_NM}})$, where $\phi$ is the normal $*$-homomorphism constructed in Lemma \ref{lem: map identify basic constructions}. By Lemma \ref{lem: corner of bidual of basic constr}, $\Phi$ is well-defined. $\Phi$ is normal because both $ \phi $ and $ \Ad(\overline{q}_{{Me_NM}}) $ are normal. Also, $\Phi$ is $M^{**}$-bimodular because $ \phi $ is $M^{**}$-bimodular by definition and $\overline{q}_{{Me_NM}} = \vee_{u\in \mathcal{U}(M)}\iota( ue_N u^{*})$ commutes with $\iota(M)$. Finally, to show $ \Phi(x)= xq_{\mathbb{X}_N}q_{\mathbb{Y}}^\perp$ for all $x\in M^{**}$, by normality of $\Phi$, it suffices to show that $ \Phi(\iota(x)) = \iota^{\sharp}(x)q_{\mathbb{X}_N}q_{\mathbb{Y}}^\perp $. Since $ \Phi(\iota(x))=  \Phi(1)\iota(x)$ for $x\in M$, it suffices to show $\phi(\overline{q}_{{Me_NM}}) = q_{\mathbb{X}_N}q^{\perp}_{\mathbb{Y}}$:
    \begin{equation*}
    \begin{aligned}
        \phi(\overline{q}_{{Me_NM}}) &= \phi( \vee_{u\in \mathcal{U}(M)}\iota(ue_Nu^{*}) ) = \vee_{u\in \mathcal{U}(M)}\phi(\iota(ue_Nu^*)) = \vee_{u\in \mathcal{U}(M)}\iota^{\sharp}(u)q\iota^{\sharp}(u^*)q_{\mathbb{Y}}^{\perp}\\ 
        &= q_{\mathbb{Y}}^{\perp}\vee_{u\in \mathcal{U}(M)}\iota^{\sharp}(u)q\iota^{\sharp}(u^*) = q_{\mathbb{Y}}^{\perp}\vee_{u\in \mathcal{U}(M)}\iota^{\sharp}(u)(\vee_{u\in \mathcal{U}(JMJ)}\iota^{\sharp}(v){q}\iota^{\sharp}(v^*))\iota^{\sharp}(u^*) \\
        &= q_{\mathbb{Y}}^{\perp}\vee_{u,v\in \mathcal{U}(M)}\iota^{\sharp}(uJvJe_NJv^*Ju^*)=p^{\sharp}_{\nor}q_{\mathbb{Y}}^{\perp}\overline{q}_{\mathbb{X}_N} =q_{\mathbb{Y}}^{\perp}{q}_{\mathbb{X}_N}  ,
    \end{aligned}
    \end{equation*}
    where we used the normality of $\phi$ and the fact that $ q_{\mathbb{Y}} $ commutes with $\iota^{\sharp}(e_N)$, $ M$, and $JMJ$.
\end{proof}

\begin{remark}
    From the computation above, we note that for our study of a von Neumann subalgebra $N \subset M$ with expectation, in many cases, it is enough to work with the subalgebra
    \[ \overline{MJMJ C^*(e_{N}MJMJ e_N)MJMJ}\subset  \mathbb{X}_{N},\]
    which is a hereditary C$^*$-subalgebra of $ C^*( M,JMJ,e_N ) $. We have the unital inclusion of von Neumann algebras
    \[ (MJMJ C^*(e_{N}MJMJ e_N)MJMJ)_{J}^{\sharp *}\subset \mathbb{X}_{N}^{\sharp *}\cap \left( M^{**}\vee JMJ^{**}\vee \C \iota^{\sharp}(e_N) \right). \]
    In particular, the new basic construction $ ( \iota^{\sharp}(M)q_{\mathbb{X}_N} \vee \C q)q^{\perp}_{\mathbb{Y}} $ lives inside $ (MJMJ C^*(e_{N}MJMJ e_N)MJMJ)_{J}^{\sharp *}q_{\mathbb{Y}}^{\perp} $.
\end{remark}

%Let $M$ be a finite von Neumann algebra, $\mathbb{X}$ an $M$-boundary piece, and $A\subset pMp$ a von Neumann subalgebra, where $p$ is a nonzero projection in $M$. Define $E \coloneqq \Ad(e_A)\circ \Ad(pJpJ)\colon \BB(L^2M)\to \BB(L^2A)$, which is u.c.p.\ and $A$-bimodular. By \cite[Remark 6.3]{MR4675043}, $\mathbb{X}^A\coloneqq E(\mathbb{K}_{\mathbb{X}}(M))$ is an $A$-boundary piece. By (essentially the same proof as) \cite[Lemma 3.11]{DS24structure}, we have $E(\mathbb{S}_{\mathbb{X}}(M)) \subset \mathbb{S}_{\mathbb{X}^A}(A)$. Therefore, if there exists an $A$-central state on $\mathbb{S}_{\mathbb{X}^A}(A)$ that is normal on $A$, then its composition with $E$ yields an $A$-central state on $\mathbb{S}_{\mathbb{X}}(M)$ that is normal on $M$. In particular, if $A$ is properly proximal relative to $\mathbb{X}$ inside $M$, then $A$ is properly proximal relative to $\mathbb{X}^A$. %We now state and prove Theorem \ref{thm: split rel biex} in the following general form.

\begin{thm} \label{thm: split rel prop prox}
    Let $M$ be a separable tracial von Neumann algebra, $p\in M$ be a projection, and $\mathcal{F}=\{N_1,\cdots,N_n\}$ be a finite family of mixing von Neumann subalgebras of $M$. %Suppose that $ [q_{\mathbb{X}_{N_i}},q_{\mathbb{X}_{N_j}}]=0 $ for all $i,j$. 
    If $A\subset pMp$ is a von Neumann subalgebra properly proximal relative to $\mathbb{X}_{\mathcal{F}}$ inside $M$, then there exist projections $f_0\in \mathcal{Z}(A)$ and $f_i\in \mathcal{Z}(A'\cap pMp)$, $i=1,\cdots,n$, such that $\sum_{i=0}^n f_i=p$, $Af_0$ is properly proximal, and $Af_i$ is amenable relative to $N_i$ inside $M$ for all $i=1,\cdots,n$.
\end{thm}
\begin{proof}
    This follows essentially by applying Theorem \ref{thm:upgrading} to the proof of \cite[Theorem 1.1]{DS24structure}, so we only sketch the argument here for completeness.

    Throughout the proof, write $f_0^\perp\coloneqq p-f_0$, with this complement taken inside $pMp$. We identify $M$ and $JMJ$ with their images under $\iota^\sharp$ when working in $\B(L^2M)^{\sharp *}_J$.

    By assumption, $A$ is properly proximal relative to $\X_{\FF}$ inside $M$. In particular, $A$ has no amenable direct summand. Let $f_0\in\mathcal Z(A)$ be the projection such that $Af_0$ is the maximal properly proximal direct summand of $A$. If $f_0=p$, the theorem follows by taking $f_i=0$ for $i\geqslant 1$. Thus we assume $f_0\neq p$. Then $Af_0^\perp$ has no amenable direct summand, is properly proximal relative to $\X_{\FF}$ inside
    $M$, and has no properly proximal direct summand. By \cite[Theorem 8.5]{MR4675043} and the maximal-support argument in the proof of
    \cite[Theorem 1.1]{DS24structure}, there exists an $Af_0^\perp$-central state $\mu$ on $\widetilde{\S}(Af_0^\perp)$ such that $\mu|_{Af_0^\perp}$ is normal and $\mu|_{\mathcal Z(Af_0^\perp)}$ is faithful. Since $\mu|_{Af_0^\perp}$ is tracial, its support projection is central; hence $\mu|_{Af_0^\perp}$ is faithful as well.
    
    Let
    \[ E\coloneqq \Ad (e_{Af_0^\perp})  \circ\Ad (f_0^\perp Jf_0^\perp J) \colon \B(L^2M)\to \B(L^2(Af_0^\perp)).\]
    By \cite[Lemma 3.11]{DS24structure}, $E$ maps $\S(M)$ into $\S(Af_0^\perp)$, and its bidual extension
    \[\widetilde E\coloneqq  \left(\left.E^*\right|_{\B(L^2(Af_0^\perp))^\sharp_J}\right)^* \colon  \B(L^2M)^{\sharp *}_J \to \B(L^2(Af_0^\perp))^{\sharp *}_J\]
    maps $\widetilde{\S}(M)$ into $\widetilde{\S}(Af_0^\perp)$. Moreover, $\widetilde E|_{f_0^\perp Mf_0^\perp}$ is the trace-preserving conditional expectation onto $Af_0^\perp$, and $\widetilde E(1-f_0^\perp)=0$. Now consider the $Af_0^\perp$-central state
    \[ \varphi\coloneqq \mu\circ\widetilde E\colon \widetilde{\S}(M)\to \C.\]
    Its restriction to $M$ is normal, its restriction to $f_0^\perp Mf_0^\perp$ is faithful, and $\varphi(f_0^\perp)=1$.
    
    Let $q_{\mathbb K}$ and $q_{\X_{\FF}}$ denote the support projections as above. We have $\varphi(q_{\mathbb K})=0$ and $\varphi(q_{\X_{\FF}}^\perp)=0$. Indeed, $q_{\mathbb K}$ is central in $\B(L^2M)^{\sharp *}_J$, and $q_{\X_{\FF}}$ commutes with both $M$ and $JMJ$. By Lemma \ref{lem: cutdown bidual}, for either $q=q_{\mathbb K}$ or $q=q_{\X_{\FF}}^\perp$, we have $q\widetilde{\S}_{\X_{\FF}}(M)q \subseteq\widetilde{\S}(M)$. For the second choice of $q$, this follows because compression by $q_{\X_{\FF}}^\perp$ annihilates the commutators with $JMJ$. If $\varphi(q)>0$, the functional
    \[ \varphi'\coloneqq \frac{1}{\varphi(q)}\, \varphi\circ\Ad (q) \colon \widetilde{\S}_{\X_{\FF}}(M) \to \C\]
    is therefore an $Af_0^\perp$-central state. Since $q$ commutes with $M$, $0\leqslant\varphi'(x) =\frac{\varphi(xq)}{\varphi(q)} \leqslant\frac{\varphi(x)}{\varphi(q)}$ for $x\in M_+$, so $\varphi'|_M$ is normal. Also, $\varphi'(f_0^\perp)=1$ because $\varphi(1-f_0^\perp)=0$. Its restriction along the canonical inclusion of $f_0^\perp\S_{\X_{\FF}}(M)f_0^\perp$ contradicts proper proximality of $Af_0^\perp$ relative to $\X_{\FF}$ inside $M$. Consequently, $\varphi(q_{\mathbb K}^\perp q_{\X_{\FF}})=1$.
    
    By Lemma \ref{lem: q_X_F is union}, $\vee_{i=1}^n q_{\X_{N_i}}=q_{\X_{\FF}}$. Define $r_0\coloneqq 0$, and for $1\leqslant i\leqslant n$, define $ r_i\coloneqq q_{\mathbb K}^\perp (\vee_{j=1}^i q_{\X_{N_j}})$, $d_i\coloneqq r_i-r_{i-1}$. These are projections commuting with both $M$ and $JMJ$, and the $d_i$ are mutually orthogonal with $\sum_{i=1}^n d_i=q_{\mathbb K}^\perp q_{\X_{\FF}}$. For each $i$, take the polar decomposition
    \[(1-r_{i-1})q_{\mathbb K}^\perp q_{\X_{N_i}} =v_i\left|(1-r_{i-1})q_{\mathbb K}^\perp q_{\X_{N_i}}\right|\]
    inside the common commutant of $M$ and $JMJ$ in $\B(L^2M)^{\sharp *}_J$. Then $v_i$ commutes with $M$ and $JMJ$ $v_i^*v_i\leqslant q_{\mathbb K}^\perp q_{\X_{N_i}}$. Moreover, we note that the range projection of $(1-r_{i-1})q_{\mathbb K}^\perp q_{\X_{N_i}}$ is $\left(r_{i-1}\vee q_{\mathbb K}^\perp q_{\X_{N_i}}\right)-r_{i-1} =r_i-r_{i-1} = d_i$, so $v_iv_i^*=d_i$. 
    
    For each $i$, consider the $M$-bimodular completely positive contraction
    \[\Phi_i\coloneqq \left.\Phi\circ\iota\right|_{\langle M,e_{N_i}\rangle}\colon \langle M,e_{N_i}\rangle \to  \bigl(\iota^\sharp(M)q_{\X_{N_i}}\vee\C q\bigr) q_{\mathbb K}^\perp\subseteq q_{\mathbb K}^\perp\widetilde{\S}(M),\]
    where $q=q_{JMJe_{N_i}JMJ}=\vee_{u\in\mathcal U(M)}\iota^\sharp(JuJe_{N_i}Ju^*J)$ and $\Phi$ is the map constructed in Theorem \ref{thm:upgrading}, applied with $\Y=\mathbb K(L^2M)$. Its commutation hypothesis holds because
    $q_{\mathbb K}$ is central. Since $q$ commutes with $JMJ$, the range of $\Phi_i$ commutes with $JMJ$, $\Phi_i(1)=q_{\mathbb K}^\perp q_{\X_{N_i}}$, and $\Phi_i(x)=xq_{\mathbb K}^\perp q_{\X_{N_i}}$ for $x\in M$. In particular, $v_i\Phi_i(T)v_i^*\in\widetilde{\S}(M)$ for every $T\in\langle M,e_{N_i}\rangle$.
    
    For each $i$ with $\varphi(d_i)>0$, define the normalized state
    \[\nu_i(T)\coloneqq \frac{\varphi(v_i\Phi_i(T)v_i^*)}{\varphi(d_i)},\qquad T\in\langle M,e_{N_i}\rangle.\]
    Then $\nu_i$ is an $Af_0^\perp$-central state. For $x\in M$, we have $\nu_i(x)  = \tfrac{\varphi(xd_i)}{\varphi(d_i)}$. Consequently, $0\leqslant\nu_i(x)\leqslant\frac{\varphi(x)}{\varphi(d_i)}$ for $x\in M_+$ and $\nu_i(f_0^\perp)=1$, so $\nu_i|_M$ is normal. Let $f_i\in\mathcal Z((Af_0^\perp)'\cap f_0^\perp Mf_0^\perp)$ be the support projection of the restriction of $\nu_i$ to this center. Its restriction to $f_i\langle M,e_{N_i}\rangle f_i$ is an $Af_i$-central state, normal on $f_iMf_i$ and faithful on
    \[ \mathcal Z\bigl((Af_i)'\cap f_iMf_i\bigr) =f_i\mathcal Z\bigl((Af_0^\perp)'\cap f_0^\perp Mf_0^\perp\bigr),\]
    so $Af_i$ is amenable relative to $N_i$ inside $M$. For indices $i$ with $\varphi(d_i)=0$, set $f_i=0$.
    
    It remains to show that $\bigvee_{i=1}^n f_i=f_0^\perp$. Let $z\coloneqq f_0^\perp-\bigvee_{i=1}^n f_i$. For each $i$ with $\varphi(d_i)>0$, the definition of $f_i$ gives $\nu_i(z)=0$. If $\varphi(d_i)=0$, then $0\leqslant\varphi(zd_i)\leqslant\varphi(d_i)=0$. Therefore,
    \[\varphi(z) =\varphi\bigl(zq_{\mathbb K}^\perp q_{\X_{\FF}}\bigr) =\sum_{i=1}^n\varphi(zd_i) = \sum_{\{i:\varphi(d_i)>0\}} \varphi(d_i)\nu_i(z)=0.\]
    
    Faithfulness of $\varphi|_{f_0^\perp Mf_0^\perp}$ implies $z=0$. Since $\mathcal Z\bigl((Af_0^\perp)'\cap f_0^\perp Mf_0^\perp\bigr)=f_0^\perp\mathcal Z(A'\cap pMp)$, all the $f_i$ belong to $\mathcal Z(A'\cap pMp)$. Replacing $f_i$ simultaneously by $f_i\prod_{j<i}(f_0^\perp-f_j)$, where the empty product is $f_0^\perp$, makes them mutually orthogonal without changing their join. Relative amenability is preserved under these central cutdowns. Thus $\sum_{i=1}^n f_i=f_0^\perp$, as required.
\end{proof}

We also have the following proposition by the same proof as in \cite[Proposition 4.2]{ding2025unique}.
\begin{prop}
    Let $M$ be a separable tracial von Neumann algebra, $p\in M$ be a projection, $\mathbb{Y}$ be an $M$-boundary piece, and $\mathcal{F}=\{N_1,\cdots,N_n\}$ be a finite family of $\mathbb{Y}$-mixing von Neumann subalgebras of $M$ satisfying $ [q_{\mathbb{Y}},\iota^{\sharp}(e_{N_i})]=0 $ for all $i$. Assume $A\subset pMp$ is a von Neumann subalgebra properly proximal relative to $ \mathbb{X}_{\mathcal{F}} $ inside $M$ and has no amenable direct summand.
    
    If there exists an $A$-central state $\varphi \colon \tilde{\mathbb{S}}_{\mathbb{Y}}(M)\to \C$ such that $ \varphi|_{pMp} $ is faithful and normal, then there exist projections $f_i\in \mathcal{Z}(A'\cap pMp)$, $i=1,\cdots,n$, such that $\sum_{i=1}^n f_i=p$, and $Af_i$ is amenable relative to $N_i$ inside $M$ for all $i=1,\cdots,n$.
\end{prop}

\begin{lemma} \label{lem: corner rel pp}
    Let $M$ be a tracial von Neumann algebra that is biexact relative to an $M$-boundary piece $\mathbb X$, $p\in M$ a projection, and $A\subset pMp$ a von Neumann subalgebra. If $A$ has no amenable direct summand, then $A$ is properly proximal relative to $\mathbb X$ inside $M$.
\end{lemma}
\begin{proof}
    Suppose, toward a contradiction, that $\varphi \colon p\S_{\X}(M)p\to \C$ is an $A$-central state whose restriction to $pMp$ is normal. Since $\varphi|_A$ is a normal tracial state, its support projection $z\coloneqq s(\varphi|_A)$ belongs to $\mathcal Z(A)$. Thus $z\neq0$, $\varphi(z)=1$, and $\varphi|_{Az}$ is a faithful normal tracial state. We write $L^2(Az,\varphi)$ for the associated standard Hilbert space.
    
    Extend the restriction of $\varphi$ to $z\S_{\X}(M)z$ to a state on $C^*(z\S_{\X}(M)z) \subset \B(zL^2M)$, whose unit is $z$, and let $(\pi,\mathcal H,\xi)$ be the corresponding GNS representation. The formula $V(a\varphi^{1/2})\coloneqq \pi(a)\xi$ for $a\in Az$ defines an isometry $V\colon L^2(Az,\varphi)\to\mathcal H$, since $\langle\pi(a)\xi,\pi(b)\xi\rangle=\varphi(a^*b)$ for any $a,b\in Az$. Define the u.c.p.\ map $\Theta\colon \S_{\X}(M)\to \B(L^2(Az,\varphi))$ by $\Theta(T)\coloneqq V^*\pi(zTz)V$. For $a,b\in Az$ and $T\in\S_{\X}(M)$, we have
    \begin{equation}\label{eqn: ambient pp gns coefficients}
        \langle a\varphi^{1/2},\Theta(T)b\varphi^{1/2}\rangle =\varphi(a^*Tb).
    \end{equation}
    For every $c\in Az$, centrality of $\varphi$ gives
    \[\langle a\varphi^{1/2},\Theta(T)bc\varphi^{1/2}\rangle =\varphi(a^*Tbc) =\varphi(ca^*Tb) =\langle ac^*\varphi^{1/2},\Theta(T)b\varphi^{1/2}\rangle.\]
    Here $a^*Tb\in z\S_{\X}(M)z$ because $\S_{\X}(M)$ is an $M$-bimodule. Since $\varphi|_{Az}$ is tracial, this identity shows that $\Theta(T)$ commutes with right multiplication by $c$. The commutant of the right copy of $Az$ on $L^2(Az,\varphi)$ is its standard left copy, so we may regard $\Theta$ as a u.c.p.\ map $\Theta\colon \S_{\X}(M)\longrightarrow Az$. Moreover, $\pi(a)V=Va$ for $a\in Az$, so $\Theta$ fixes $Az$ and is $Az$-bimodular. Taking $a=b=z$ in \eqref{eqn: ambient pp gns coefficients} gives $\varphi(\Theta(T))=\varphi(zTz)$ for $T\in\S_{\X}(M)$.
    
    We claim that $\Theta|_M$ is normal. If $0\leqslant x_j\to x$ is a bounded increasing net in $M$, then $\sup_j\Theta(x_j)\leqslant\Theta(x)$. By normality of $\varphi|_{Az}$ and $\varphi|_{pMp}$,
    \[ \varphi\bigl(\sup_j\Theta(x_j)\bigr) =\lim_j\varphi(zx_jz) =\varphi(zxz) =\varphi(\Theta(x)).\]
    Faithfulness of $\varphi|_{Az}$ therefore gives $\sup_j\Theta(x_j)=\Theta(x)$, proving the claim.
    
    We next verify the continuity needed to compose $\Theta$ with the approximations witnessing relative biexactness.
    Fix a positive functional $\eta\in(Az)_*^+$ and let $\omega\coloneqq (\eta\circ\Theta)|_M\in M_*^+$. The positive functional $\eta\circ\Theta$ extends to a positive functional on $C^*(\S_{\X}(M))$. Applying Cauchy-Schwarz in its GNS representation, for any factorization $T=a^*Yb$ with $a,b\in M$ and $Y\in\S_{\X}(M)$, we obtain
    \[|\eta(\Theta(T))|\leqslant\omega(a^*a)^{1/2}\,\|Y\|\,\omega(b^*b)^{1/2}.\]
    Taking the infimum over these factorizations yields $|\eta(\Theta(T))|\leqslant s_\omega^\omega(T)$ for $T\in\S_{\X}(M)$.  Since every normal linear functional on $Az$ is a linear combination of normal positive functionals, $\Theta$ is continuous from the $M$-topology on $\S_{\X}(M)$ to the ultraweak topology on $Az$.
    
    By relative biexactness, there exist nets of c.c.p.\ maps
    \[M\xrightarrow{\varphi_i}\M_{n(i)}(\C)\xrightarrow{\psi_i}\S_{\X}(M)\]
    such that $\psi_i\circ\varphi_i(x)\to x$ in the $M$-topology for every $x\in M$. Hence, for every $a\in Az$, $\Theta\circ\psi_i\circ\varphi_i(a)\to a$ ultraweakly in $Az$. Thus the identity map on $Az$ is weakly nuclear.
    
    For completeness, extend each c.c.p\. map $\varphi_i|_{Az}$, by Arveson's extension theorem, to a c.c.p.\ map $\widetilde\varphi_i\colon \B(L^2(Az,\varphi)) \to \M_{n(i)}(\C)$. By point-ultraweak compactness of the set of c.c.p.\ maps into $Az$, a subnet of $\Theta\circ\psi_i\circ\widetilde\varphi_i\colon  \B(L^2(Az,\varphi))\to Az$ converges point-ultraweakly to a c.c.p.\ map $\Phi\colon \B(L^2(Az,\varphi))\to Az$. The preceding convergence gives $\Phi(a)=a$ for every $a\in Az$. In particular, $\Phi$ is unital and is a projection onto $Az$. Thus $Az$ is injective, contradicting the assumption that $A$ has no amenable direct summand.
\end{proof}

\begin{thm}[Theorem \protect{\ref{thm: split rel biex}}]
      Let $M$ be a separable tracial von Neumann algebra and $\mathcal{F}=\{N_1,\cdots,N_n\}$ a finite family of mixing von Neumann subalgebra of $M$. Let $p\in M$ be a projection. If $M$ is biexact relative to $ \mathbb{X}_\FF$, then for any von Neumann subalgebra $ A\subset pMp $, there exist projections $f_0, f_{n+1}\in \mathcal{Z}(A)$ and $f_i\in \mathcal{Z}(A'\cap pMp)$, $i=1,\cdots,n$, such that $\sum_{i=0}^{n+1} f_i=p$, $Af_0$ is properly proximal, $Af_{n+1}$ is amenable, and $Af_i$ is amenable relative to $N_i$ inside $M$ for all $i=1,\cdots,n$.
\end{thm}

\begin{proof}
    If $p=0$, take $f_i=0$ for every $i$. Otherwise, let $f_{n+1}\in\mathcal Z(A)$ be the maximal projection such that $Af_{n+1}$ is amenable. Write $f_{n+1}^{\perp}\coloneqq p-f_{n+1}$, with the complement taken inside $pMp$.

    If $f_{n+1}=p$, take $f_i=0$ for $0\leqslant i \leqslant n$. Otherwise, $Af_{n+1}^{\perp}$ has no amenable direct summand. By Lemma \ref{lem: corner rel pp}, it is properly proximal relative to $\mathbb X_{\mathcal F}$ inside $M$. Applying Theorem \ref{thm: split rel prop prox} to $Af_{n+1}^{\perp} \subset f_{n+1}^{\perp}Mf_{n+1}^{\perp}$, with ambient algebra $M$, we obtain mutually orthogonal projections $f_0\in\mathcal Z(Af_{n+1}^{\perp})$ and $f_i\in\mathcal Z((Af_{n+1}^{\perp})' \cap f_{n+1}^{\perp}Mf_{n+1}^{\perp})$ for $1\leqslant i \leqslant n$ such that $\sum_{i=0}^n f_i=f_{n+1}^{\perp}$, $Af_0$ is properly proximal, and $Af_i$ is amenable relative to $N_i$ inside $M$ for every nonzero $f_i$, $1\leqslant i\leqslant n$.
    
    Since $f_{n+1}^{\perp}\in\mathcal Z(A)$, we have $\mathcal Z(Af_{n+1}^{\perp})=f_{n+1}^{\perp}\mathcal Z(A)$, and $\mathcal Z((Af_{n+1}^{\perp})' \cap f_{n+1}^{\perp}Mf_{n+1}^{\perp})  =f_{n+1}^{\perp}\mathcal Z(A'\cap pMp)$. Thus $f_0\in\mathcal Z(A)$ and $f_i\in\mathcal Z(A'\cap pMp)$ for $1\leqslant i\leqslant n$. All these projections are orthogonal to $f_{n+1}$, and $\sum_{i=0}^{n+1}f_i=f_{n+1}^{\perp}+f_{n+1}=p$. This proves the theorem.
    \iffalse
    Since $M$ is biexact relative to $\mathbb{X}_\FF $ and $A\subset M$, it follows from \cite[Proposition 6.10]{ding2023biexact} that $A$ is biexact relative to $\mathbb{X}_\FF^A$. By \cite[Corollary 6.11]{ding2023biexact}, there exists a projection $f_{n+1}\in \mathcal{Z}(A)$ such that $Af_{n+1}$ is amenable and $Af_{n+1}^\perp$ is properly proximal relative to $\mathbb{X}_{\FF}^A$. The conclusion then follows from Theorem \ref{thm: split rel prop prox}.
    \fi
\end{proof}

\begin{corollary}[\protect{Corollary \ref{cor: AFP absorbing}}]
    Let $(M_1,\tau_1)$ and $(M_2,\tau_2)$ be two separable tracial weakly exact von Neumann algebras with a common amenable subalgebra $B$ such that $\tau_1|_B=\tau_2|_B$, and $ M_1\overline{\ast}_{B}M_2 $ be the AFP von Neumann algebra. Assume that $B$ is mixing in $M_i$ for $i=1,2$. If $A\subset M$ is a von Neumann subalgebra with no properly proximal direct summand and $A\cap M_1 \npreceq_{M_1} B $, then $ A\subset M_1 $.
\end{corollary}
\begin{proof}
    $ M_1\overline{\ast}_{B}M_2 $ is biexact relative to $\{M_1,M_2\}$ by Corollary \ref{cor: relative biexactness tracial} and $ A $ has no properly proximal direct summand by presumption, so by Theorem \ref{thm: split rel biex}, there exist central projections $f_1,f_2\in \mathcal{Z}(A'\cap M)$ such that $f_1+f_2=1$ and $Af_i$ is amenable relative to $M_i$ inside $M$ for each $i=1,2$.
    \iffalse
    Suppose $Af_2$ is not amenable relative to $B$, then by \cite[Corollary 2.12]{Ioana15cartan}, we must have $ Af_2\preceq_M M_2 $. Since $A\cap M_1 \npreceq_{M_1} B$ by presumption, there exists a net $(u_n)\subset \mathcal{U}(A\cap M_1)$ such that $ \lim_{n} \|E_{B}( xu_ny )\|_2 =0 $ for all $x,y\in M_1$, and such net $(u_n)$ is automatically converging to $0$ weakly. Since $M_2$ is mixing inside $M$, we then have for all $x,y\in M$, 
    \[\lim_{n} \|E_{M_2}(xu_nf_2y)\|_2 = \lim_{n} \|E_{M_2}(x)E_{M_2}(u_n)E_{M_2}(f_2y)\|_2= \lim_{n} \|E_{M_2}(x)E_{B}(u_n)E_{M_2}(f_2y)\|_2 =0,\]
    contradicting to that $ Af_2\preceq_M M_2$. Therefore, $Af_2$ must be amenable relative to $B$. In particular, $ A $ is amenable relative to $M_1$. Since $ A\cap M_1 \not\preceq_{M_1} B $, we obtain $ A\subset M_1 $ by the main theorem of \cite{BH18amenableabsorption}.
    \fi
    Suppose $Af_2$ is not amenable relative to $B$, then by
\cite[Corollary 2.12]{Ioana15cartan}, we must have
$Af_2\preceq_M M_2$. Since $A\cap M_1\npreceq_{M_1}B$ by presumption,
there exists a sequence $(u_n)\subset\mathcal U(A\cap M_1)$ such that
$
    \lim_n\|E_B(a u_n b)\|_2=0
$
for all $a,b\in M_1$. We claim that
$
    \lim_n\|E_{M_2}(x u_n y)\|_2=0
$
for all $x,y\in M$. Indeed, by density it suffices to consider reduced
words $x=x_0a$ and $y=by_0$, where $a,b\in M_1$ and $x_0,y_0$ are
either trivial or respectively end and begin with an $M_2\ominus B$
letter. By freeness with amalgamation,
$
E_{M_2}(x u_n y)
=
E_{M_2}\bigl(x_0E_B(a u_n b)y_0\bigr),
$
and hence
\[
\|E_{M_2}(x u_n y)\|_2
\leqslant
\|x_0\|_\infty\|y_0\|_\infty\|E_B(a u_n b)\|_2
\longrightarrow0.
\]
The general case follows by Kaplansky density. Since $f_2\in A'\cap M$,
we have $u_nf_2=f_2u_n$, and therefore
\[
\lim_n\|E_{M_2}(xu_nf_2y)\|_2
=
\lim_n\|E_{M_2}(xf_2u_ny)\|_2=0
\]
for all $x,y\in M$, contradicting $Af_2\preceq_M M_2$. Therefore,
$Af_2$ must be amenable relative to $B$. In particular, $ A $ is amenable relative to $M_1$. Since $ A\cap M_1 \not\preceq_{M_1} B $, we obtain $ A\subset M_1 $ by the main theorem of \cite{BH18amenableabsorption}.
\end{proof}

When $M$ is an AFP von Neumann algebra of weakly exact tracial von Neumann algebras, relative amenability in Theorem \ref{thm: split rel biex} can be replaced by Popa's intertwining-by-bimodule condition in the following form.

\begin{prop} \label{prop: amalg free split intertwining}
    Let $(M_1,\tau_1),\cdots, (M_n,\tau_n)$ be separable weakly exact tracial von Neumann algebras, $B$ be a common injective mixing von Neumann algebra with expectation such that $\tau_i|_B=\tau_j|_B$ for $1\leqslant i,j\leqslant n$, and $M = M_1 \overline{*}_B \cdots \overline{*}_B M_n$ be the AFP von Neumann algebra over $B$. Then for any projection $p\in M$ and any von Neumann subalgebra $ A\subset pMp $, there exist projections $p_0, p_{n+1}\in \mathcal{Z}(A)$ and $p_i\in \mathcal{Z}(A'\cap pMp)$, $i=1,\cdots,n$, such that $\sum_{i=0}^{n+1} p_i=p$, $Ap_0$ is properly proximal, $Ap_{n+1}$ is amenable, and $Ap_i \preceq_M M_i$ for all $i=1,\cdots,n$.
\end{prop}
\begin{proof}
    By Proposition \ref{prop: amalg free prod rel biex}, $M$ is biexact relative to $\{M_1,\cdots, M_n\}$. By Corollary \ref{cor: multiple biexact rel B}, each $M_i$ is mixing in $M$. $M$ is separable since every $M_i$ is separable. By Theorem \ref{thm: split rel biex}, there exist mutually orthogonal projections $f_0, f_{n+1}\in \mathcal{Z}(A)$ and $f_i\in \mathcal{Z}(A'\cap pMp)$, $i=1,\cdots,n$, such that $\sum_{i=0}^{n+1} f_i=p$, $Af_0$ is properly proximal, $Af_{n+1}$ is amenable, and $Af_i$ is amenable relative to $M_i$ inside $M$ for all $i=1,\cdots,n$ with $f_i\neq0$.

    Let $p_i' \in \mathcal{Z}(Af_i)$ be the maximal projection such that $Ap_i'$ is amenable for each $i=1,\cdots, n$, where we take $p_i'=0$ if $f_i=0$. Let $z\in\mathcal Z(A)$ be the maximal projection such that $Az$ is amenable. The map $A\to Af_i$, $a\mapsto af_i$, is a normal surjective $*$-homomorphism with central kernel, so it identifies $Af_i$ with a central direct summand of $A$. Consequently, its maximal amenable direct summand is the image of $Az$, and $p_i'=zf_i$. In particular, both $p_i'$ and $f_i-p_i'$ belong to $\mathcal Z(A'\cap pMp)$.

    Then for $p_i\coloneqq f_i - p_i'$ with $p_i\neq0$, $Ap_i$ has no amenable direct summand, so $Ap_i$ cannot be amenable relative to $B$ inside $M$ since $B$ is injective. Moreover, $Ap_i$ remains amenable relative to $M_i$ inside $M$, by restriction to this central corner of $Af_i$. Writing $M=M_i\overline{*}_B(\overline{*}_{B,j\neq i}M_j)$, by \cite[Corollary 2.12]{Ioana15cartan}, we therefore have $\mathcal N_{p_iMp_i}(Ap_i)''\preceq_M M_i$. Hence $Ap_i \preceq_M M_i$ whenever $p_i\neq 0$.

    Finally, we let $p_0 =f_0$ and $p_{n+1} = f_{n+1} + \sum_{i=1}^n p_i'$. Since $Af_0$ is properly proximal, it has no amenable direct summand, so $zf_0=0$. Also, since $Af_{n+1}$ is amenable, $zf_{n+1}=f_{n+1}$. Multiplying $\sum_{i=0}^{n+1}f_i=p$ by $z$ gives
    \[p_{n+1}=f_{n+1}+\sum_{i=1}^n zf_i=z\in\mathcal Z(A).\]
    Thus $Ap_{n+1}$ is amenable, all the $p_i$ are mutually orthogonal and belong to the required centers, and
    \[p_0+\sum_{i=1}^n p_i+p_{n+1}=f_0+\sum_{i=1}^n(f_i-p_i')+f_{n+1}+\sum_{i=1}^n p_i'=p.\]
    This proves the proposition.
    \iffalse
    By Corollary \ref{cor: relative biexactness tracial}, $M$ is biexact relative to $\{M_1,\cdots, M_n\}$. By Theorem \ref{thm: split rel biex}, there exist projections $f_0, f_{n+1}\in \mathcal{Z}(A)$ and $f_i\in \mathcal{Z}(A'\cap pMp)$, $i=1,\cdots,n$, such that $\sum_{i=0}^{n+1} f_i=p$, $Af_0$ is properly proximal, $Af_{n+1}$ is amenable, and $Af_i$ is amenable relative to $M_i$ inside $M$ for all $i=1,\cdots,n$. Let $p_i' \in \mathcal{Z}(Af_i)$ be the maximal projection such that $Ap_i'$ is amenable for each $i=1,\cdots, n$. Then for $p_i\coloneqq f_i - p_i'$, $Ap_i$ has no amenable direct summand, so $Ap_i$ cannot be amenable relative to $B$ inside $M$ since $B$ is injective. By \cite[Corollary 2.12]{Ioana15cartan}, we therefore have $Ap_i \preceq_M M_i$. Finally, we let $p_0 =f_0$ and $p_{n+1} = f_{n+1} + \sum_{i=1}^n p_i$.
    \fi
\end{proof}

%\begin{corollary}
 %   Let $M = M_1 \overline{\ast}\cdots \overline{\ast}M_m = N_1 \overline{\ast}\cdots \overline{\ast}N_n$, where $M_i$ and $N_j$ are separable $\text{II}_1$ factors. If $M_i$ 
    %and $N_j$ 
 %   are weakly exact, nonamenable, and non-properly proximal for every $i$, then $ m=n $, and up to a permutation of indices, $M_i$ is unitarily conjugate to $N_i$ in $M$.
%\end{corollary}

\begin{corollary}[\protect{Corollary \ref{cor: Kurosh-type}}]
    If $M = M_1 \overline{\ast}\cdots \overline{\ast}M_m = N_1 \overline{\ast}\cdots \overline{\ast}N_n$, where each $M_i, N_j$ is a separable weakly exact nonamenable non-properly proximal $\text{II}_1$ factor, then $ m=n $, and up to a permutation of indices, $M_i$ is unitarily conjugate to $N_i$ in $M$.
\end{corollary}
\begin{proof}
    %First we assume all $M_i$ and $N_j$ are weakly exact nonamenable non-properly proximal separable $\text{II}_1$ factors. 
    Since $\mathcal{Z}(M_1'\cap M) = \mathbb{C}$, Proposition \ref{prop: amalg free split intertwining} implies that either $M_1 \preceq_{M} N_1$ or $M_1 \preceq_{M} N_2*\cdots N_n$. Then the same argument as in \cite[Corollary 8.1]{Drimbe23} yields the statement.

\end{proof}

\subsection{Upgrading biexactness}
\begin{lemma}\label{lem:Upgrading biexactness}
    Let $M$ be a von Neumann algebra and $N\subset M$ be a mixing subalgebra with expectation. If $M$ is biexact relative to $N$ and the embedding $ M\xhookrightarrow{} \langle M,e_N\rangle $ is $M$-nuclear, then the inclusion $\iota^{\sharp}|_{M}: M \xhookrightarrow{} \tilde{\mathbb{S}}(M) $ is weak${}^*$-nuclear.
\end{lemma}
\begin{proof}
    Since $M$ is biexact relative to $N$, the embedding $M\xhookrightarrow{} \mathbb{S}_{\mathbb{X}_N}(M) $ is $M$-nuclear. Since $\iota^\sharp$ is continuous from the $M$-topology to the weak-$*$ topology, the inclusion
    $$ M \xhookrightarrow{\iota^{\sharp}} p^{\sharp}_{\nor}\mathbb{S}_{\mathbb{X}_N}(M)^{**}p^{\sharp}_{\nor} \subset \tilde{\mathbb{S}}_{\mathbb{X}_N}(M) $$
    is weak${}^*$-nuclear. In particular, by Lemma \ref{lem: cutdown bidual} and normality of compression maps, 
    \[\Ad(q_{\mathbb{K}})\circ\iota^{\sharp}|_{M}\colon M\to \mathbb{K}(L^2M)^{\sharp *}_J,\quad \Ad(q_{\mathbb{X}_N}^{\perp})\circ\iota^{\sharp}|_{M}\colon M\to q_{\mathbb{X_N}}^{\perp}\tilde{\mathbb{S}}_{\mathbb{X}_N}(M)q_{\mathbb{X}_N}^{\perp}  \]
    are also weak${}^*$-nuclear maps.
    We trivially have $  \mathbb{K}(L^2M)^{\sharp *}_J\subset \tilde{\mathbb{S}}(M) $. We also observe that $ q_{\mathbb{X}_N}^{\perp}\tilde{\mathbb{S}}_{\mathbb{X}_N}(M)q_{\mathbb{X}_N}^{\perp}  $ is precisely
    \[ q_{\mathbb{X}_N}^{\perp}\tilde{\mathbb{S}}_{\mathbb{X}_N}(M)q_{\mathbb{X_N}}^{\perp} = \iota^\sharp(JMJ)'\cap  q_{\mathbb{X}_N}^{\perp}\B(L^2M)^{\sharp *}_J q_{\mathbb{X}_N}^{\perp} \subset \tilde{\mathbb{S}}(M)q_{\mathbb{K}}^{\perp}. \]
    Therefore, we have two weak${^*}$-nuclear maps,
    $$ \Ad(q_{\mathbb{K}})\circ\iota^{\sharp}|_{M}: M\to \tilde{\mathbb{S}}(M),\quad \Ad(q_{\mathbb{X}_N}^{\perp})\circ\iota^{\sharp}|_{M}: M\to \tilde{\mathbb{S}}(M)q_{\mathbb{K}}^{\perp}. $$
    
    By Theorem 5.12, we obtain a normal c.c.p.\ map
    \[\Phi\colon \langle M,e_N\rangle^{**}\to\tilde{\mathbb{S}}(M)q_{\mathbb{K}}^{\perp} \]
    sending $\iota(x)$ to $\iota^{\sharp}(x)q_{\mathbb{X}_N}q_{\mathbb{K}}^{\perp}$ for every $x\in M$. The theorem's commutation hypothesis holds because $q_{\mathbb{K}}$ is central. Its range algebra $(\iota^{\sharp}(M)q_{\mathbb{X}_N}\vee\mathbb{C}q_{JMJe_NJMJ})q_{\mathbb{K}}^{\perp}$
    is contained in $\tilde{\mathbb{S}}(M)q_{\mathbb{K}}^{\perp}$ because both generators commute with $\iota^{\sharp}(JMJ)$. The map is unital into this range algebra, whose unit is $q_{\mathbb{X}_N}q_{\mathbb{K}}^{\perp}$.
    We note that $p_{\nor}\in M^{**}\subset \langle M,e_N\rangle^{**}$, so $p_{\nor}\langle M,e_N\rangle^{**}p_{\nor}\subset \langle M,e_N\rangle^{**}$.
    Since $x\mapsto\Phi(\iota(x))$ is normal on $M$, its normal extension to $M^{**}$ annihilates $(1-p_{\nor})M^{**}$. Thus
    \[\Phi(p_{\nor})=\Phi(1)=q_{\mathbb{X}_N}q_{\mathbb{K}}^{\perp}.\]
    Completely positive Cauchy-Schwarz gives $\Phi(T)=\Phi(p_{\nor}Tp_{\nor})$ for all $T\in\langle M,e_N\rangle^{**}$. Hence $\Phi$ restricts to a normal c.c.p.\ map on the source corner $p_{\nor}\langle M,e_N\rangle^{**}p_{\nor}$.
    Now, by our assumption, the inclusion $M\xhookrightarrow{}\langle M,e_N\rangle$ is $M$-nuclear. Since $\Phi\circ\Ad(p_{\nor})$ is normal and the weak${}^*$ topology on $p_{\nor}\langle M,e_N\rangle^{**}p_{\nor}$ pulls back under $T\mapsto p_{\nor}\iota(T)p_{\nor}$ to the weak $M$-topology on $\langle M,e_N\rangle$, we obtain the weak${}^*$-nuclear map
    \[\Phi\circ\Ad(p_{\nor})\circ\iota|_M\colon  M\to\tilde{\mathbb{S}}(M)q_{\mathbb{K}}^{\perp}.\]
    Notice that this map sends any $x\in M$ to $\iota^{\sharp}(x)q_{\mathbb{X}_N}q_{\mathbb{K}}^{\perp}$, or in other words, the map $\Ad(q_{\mathbb{X}_N}q_{\mathbb{K}}^{\perp})\circ\iota^{\sharp}|_M$ is nuclear in the weak-$*$ topology.

    The projections $q_{\mathbb{K}}$, $q_{\mathbb{X}_N}^{\perp}$, and $q_{\mathbb{X}_N}q_{\mathbb{K}}^{\perp}$ are mutually orthogonal, sum to $p^{\sharp}_{\nor}$, and commute with $\iota^{\sharp}(M)$.
    Finally, since
    \[ \iota^{\sharp}\big|_M =\Ad(q_{\mathbb{K}})\circ\iota^{\sharp}|_M + \Ad(q_{\mathbb{X}_N}^{\perp})\circ\iota^{\sharp}|_M +\Ad(q_{\mathbb{X}_N}q_{\mathbb{K}}^{\perp})\circ\iota^{\sharp}\big|_M\]
    is the sum of three weak${}^*$-nuclear maps from $M$ to $\tilde{\mathbb{S}}(M)$, the map $\iota^{\sharp}\big|_M:M\to\tilde{\mathbb{S}}(M)$ must also be nuclear in the weak-$*$ topology.
\end{proof}

The following proposition is an application of \cite[Theorem 7.3]{ding2023biexact}.
\begin{prop}
    Let $M$ be a separable von Neumann algebra with a faithful normal state $\mu$, and $\mathbb{X}$ a $M$-boundary piece. The following conditions are equivalent:
    \begin{enumerate}
        \item $M$ is biexact relative to $\mathbb{X}$.
        \item $M$ is weakly exact, and the inclusion $\iota^{\sharp}|_M : M\hookrightarrow \tilde{\mathbb{S}}_\mathbb{X}(M)$ is weak${}^*$ nuclear.
        \item $M$ is weakly exact, and for every $\varepsilon>0$ and finite-dimensional operator systems $E\subset F\subset \B(L^2M)$ with $E\subset M$ there exists a u.c.p. map $\phi:\B(L^2M)\to \B(L^2M)$ such that $ d_{r_\mu}(x-\phi(x),\mathbb{K}_{\mathbb{X}}^{\infty,1}(M))<\varepsilon $, and $d_{r_\mu}([JxJ,\phi(T)],\mathbb{K}_\mathbb{X}^{\infty,1}(M))<\varepsilon$ for all $x\in E_1$ and $T\in F_1$.
    \end{enumerate}
\end{prop}
\begin{proof}
    (3)$\implies$(1) is \cite[Theorem 7.3]{ding2023biexact}. (1)$\implies$(2) is trivial as $ p^{\sharp}_{\nor} {\mathbb{S}}_\mathbb{X}(M)^{**}p^{\sharp}_{\nor}\subset \tilde{\mathbb{S}}_\mathbb{X}(M)$. It remains to show (2)$\implies$(3).

Fix $\varepsilon>0$ and finite-dimensional operator systems
$E\subset F\subset\B(L^2M)$ with $E\subset M$.
By (2) and \cite[Corollary 5.7 and Lemma 4.2]{ding2023biexact}, there are nets of u.c.p.\ maps
\[
    M\xrightarrow{\phi_n}\M_{k_n}(\C)
    \xrightarrow{\psi_n}\tilde{\mathbb S}_{\mathbb X}(M)
\]
such that $\psi_n\circ\phi_n(x)\to\iota^\sharp(x)$
weak${}^*$ for every $x\in M$.
By Arveson's extension theorem, extend each $\phi_n$ to
$\B(L^2M)$, keeping the same notation. As in the proof of \cite[Lemma~4.4]{ding2023biexact},
by Kaplansky density, followed by
normalization, gives u.c.p.\ maps
\[
    \widetilde\psi_{n,m}\colon
    \M_{k_n}(\C)\to\B(L^2M)
\]
such that $\iota^\sharp\circ\widetilde\psi_{n,m}\to\psi_n$
in the point-weak${}^*$ topology for each fixed $n$.

Choose a basis $x_1,\ldots,x_s$ of $E$.
Since
\[
    [\iota^\sharp(Jx_iJ),\psi_n(A)]
    \in\mathbb K_{\mathbb X}(M)^{\sharp *}_J,
    \qquad A\in\M_{k_n}(\C),
\]
choosing approximants on matrix units gives linear maps
\(
    k_{n,m}^i\colon
    \M_{k_n}(\C)\to\mathbb K_{\mathbb X}(M)
\)
such that
\[
    \iota^\sharp(k_{n,m}^i(A))
    \longrightarrow
    [\iota^\sharp(Jx_iJ),\psi_n(A)]
\]
weak${}^*$ for every $i$ and $A$.
Here we use a common index $m$ for these approximations.
Consequently,
\(
    [Jx_iJ,\widetilde\psi_{n,m}(A)]-k_{n,m}^i(A)
    \to 0
\)
in the weak $M$-$M$ and $JMJ$-$JMJ$ topology. By Hahn-Banach, taking convex combinations
simultaneously for the basis vectors and matrix units,
we may assume the strong $M$-$M$ and $JMJ$-$JMJ$ topology convergence
\[
    r_\mu\bigl(
        [Jx_iJ,\widetilde\psi_{n,m}(A)]-k_{n,m}^i(A)
    \bigr)\longrightarrow 0
\]
for every $i$ and $A$, while retaining the point-weak${}^*$
limit $\iota^\sharp\circ\widetilde\psi_{n,m}\to\psi_n$.
Since $E$ and $\M_{k_n}(\C)$ are finite-dimensional,
after passing to sufficiently large $m$, we obtain
\begin{equation}\label{eq:bidualbiexact}
    \sup_{\substack{x\in E_1\\ A\in\M_{k_n}(\C)_1}}
    d_{r_\mu}\bigl(
        [JxJ,\widetilde\psi_{n,m}(A)],
        \mathbb K_{\mathbb X}^{\infty,1}(M)
    \bigr)
    <\frac{\varepsilon}{2}.
\end{equation}

On the other hand,
\(
    \lim_n\lim_m
    \iota^\sharp\bigl(
        \widetilde\psi_{n,m}\circ\phi_n(x_i)
    \bigr)
    =\iota^\sharp(x_i),
    \; i=1,\ldots,s,
\)
in the weak${}^*$ topology.
Applying Hahn--Banach once more, simultaneously for
$x_1,\ldots,x_s$,
we obtain a u.c.p. convex combination
\[
    \phi
    :=\sum_{\ell=1}^p
    \lambda_\ell\,
    \widetilde\psi_{n_\ell,m_\ell}\circ\phi_{n_\ell}
    \colon\B(L^2M)\to\B(L^2M)
\]
such that
\(
    \sup_{x\in E_1}r_\mu(\phi(x)-x)
    <\frac{\varepsilon}{2}.
\)
Since each $\phi_n$ is contractive and
\eqref{eq:bidualbiexact} is preserved under convex
combinations, we also have
\(
    d_{r_\mu}\bigl(
        [JxJ,\phi(T)],
        \mathbb K_{\mathbb X}^{\infty,1}(M)
    \bigr)
    <\frac{\varepsilon}{2}.
\) for all $x\in E_1$ and $T\in F_1$.

\end{proof}

Combining Lemma \ref{lem:Upgrading biexactness} with Condition (2) in the previous proposition, we obtain
\begin{thm}[\protect{Theorem \ref{thm: upgrading biexactness}}] 
    Let $M$ be a separable von Neumann algebra and $N\subset M$ be a mixing subalgebra with expectation. If $M$ is biexact relative to $N$ and the embedding $ M\hookrightarrow \langle M,e_N\rangle $ is $M$-nuclear, then $M$ is biexact.
\end{thm}

\begin{corollary}[\protect{Corollary \ref{cor: biexact AFP}}]
    Let $M_1,\cdots, M_n$ be separable injective von Neumann algebras admitting a common subalgebra $B$ with expectation. If $B$ is mixing in $M_i$ for each $i=1,\cdots, n$, then the AFP von Neumann algebra $ M= M_1\overline{*}_B \cdots \overline{*}_B M_n $ is biexact.
\end{corollary}
\begin{proof}
    By Theorem \ref{thm: relative biexactness injective}, $M$ is biexact relative to the amalgam $B$, and hence $M$ is weakly exact. Since $B$ is the range of a normal conditional expectation from the injective von Neumann algebra $M_1$, $B$ is also injective. Applying Lemma \ref{lem: weak exactness into smaller injective}, we obtain that the inclusion $M\hookrightarrow\langle M,e_B\rangle$ is $M$-nuclear. Moreover, Corollary \ref{cor: multiple biexact rel B} shows that $B$ is mixing in $M$. The conclusion now follows from the previous theorem.
\end{proof}

\bibliographystyle{amsalpha}
\bibliography{bib}
%\nocite{MR1217253}
%\printbibliography

\end{document}